\documentclass[reqno,12pt]{amsart}
\usepackage{rotating}
\usepackage{multicol}
\usepackage{enumitem}
\usepackage{mathrsfs}
\usepackage{mathtools}
\usepackage{multirow,makecell}
\usepackage{centernot}
\usepackage{xfrac}
\usepackage{longtable}
\usepackage{appendix}
\usepackage[normalem]{ulem}
\usepackage{xpatch}
\xpatchcmd{\paragraph}{\normalfont}{{\normalfont\bfseries}}{}{}
\usepackage{eufrak}
\usepackage{hyperref}
\usepackage{array}
\usepackage{comment}
\usepackage{tipa}
\usepackage[usenames,dvipsnames]{xcolor}
\usepackage[margin=1in]{geometry}
\usepackage{tikz-cd}
\usepackage{float} 
\usetikzlibrary{positioning,arrows}
\usetikzlibrary{calc}
\usetikzlibrary{arrows.meta}

\usepackage{amsmath}
\usepackage[noabbrev]{cleveref}

\theoremstyle{plain} 
\newtheorem{theorem}{Theorem}[section]
\newtheorem{prop}[theorem]{Proposition}
\newtheorem{lemma}[theorem]{Lemma}
\newtheorem{cor}[theorem]{Corollary}
\newtheorem{conj}[theorem]{Conjecture}

\theoremstyle{remark}
\newtheorem{rmk}{Remark}[section]

\theoremstyle{definition}
\newtheorem{defn}{Definition}[section]
\newtheorem{example}{Example}

\renewcommand{\mod}[1]{{\ifmmode\text{\rm\ (mod~$#1$)}\else\discretionary{}{}{\hbox{ }}\rm(mod~$#1$)\fi}}

\renewcommand{\a}{\alpha}
\renewcommand{\b}{\beta}
\newcommand{\g}{\gamma}

\newcommand{\Z}{\mathbb Z}

\newcommand{\C}{\mathbb C}

\newcommand{\G}{\mathcal G}

\renewcommand{\l}{\lambda}

\DeclareMathOperator{\wt}{wt}
\DeclareMathOperator{\spin}{spin}
\DeclareMathOperator{\SSRT}{SSRT}

\newcolumntype{M}[1]{>{\centering\arraybackslash}m{#1}}
\usepackage{tikz}
\usetikzlibrary{decorations.markings}
\newcommand{\calL}[0]{\mathcal L}
\newcommand{\calR}[0]{\mathcal R}
\newcommand{\calQ}[0]{\mathcal Q}

\newcommand{\arrowleft}[2]{\filldraw[color=yellow!10!red] (#1-0.1,#2)--(#1+0.1,#2-0.07)--(#1+0.1,#2+0.07)--cycle;}

\newcommand{\arrowdown}[2]{\filldraw[color=black!20!blue] (#1,#2-0.1)--(#1-0.07,#2+0.1)--(#1+0.07,#2+0.1)--cycle;}

\newcommand{\arrowright}[2]{\filldraw[color=black!30!blue] (#1-0.1,#2+0.07)--(#1+0.1,#2)--(#1-0.1,#2-0.07)--cycle;}

\newcommand{\arrowup}[2]{\filldraw[color=yellow!10!red] (#1,#2+0.1)--(#1+0.07,#2-0.1)--(#1-0.07,#2-0.1)--cycle ;}

\newcommand{\arrowleftcolor}[3]{\filldraw[color=#3] (#1-0.09063,#2+0.04226)--(#1+0.05682,#2-0.11477)--(#1+0.12444,#2+0.03024)--cycle;}

\newcommand{\arrowdowncolor}[3]{\filldraw[color=#3] (#1,#2-0.1)--(#1-0.07,#2+0.1)--(#1+0.07,#2+0.1)--cycle;}

\newcommand{\arrowrightcolor}[3]{\filldraw[color=#3] (#1-0.05682,#2+0.11477)--(#1+0.09063,#2-0.04226)--(#1-0.12444,#2-0.03024)--cycle;}

\newcommand{\arrowupcolor}[3]{\filldraw[color=#3] (#1,#2+0.1)--(#1+0.07,#2-0.1)--(#1-0.07,#2-0.1)--cycle ;}

\newcommand{\arrowul}[2]{\filldraw[color=yellow!10!red,draw=none] (#1-0.07071067811,#2+0.07071067811)--(#1+0.1202081528,#2-0.02121320343)--(#1+0.02121320343,#2-0.1202081528);}
\newcommand{\arrowdl}[2]{\filldraw[color=yellow!10!red,draw=none] (#1-0.07071067811,#2-0.07071067811)--(#1+0.1202081528,#2+0.02121320343)--(#1+0.02121320343,#2+0.1202081528);}
\newcommand{\arrowur}[2]{\filldraw[color=black!30!blue,draw=none] (#1+0.07071067811,#2+0.07071067811)--(#1-0.1202081528,#2-0.02121320343)--(#1-0.02121320343,#2-0.1202081528);}
\newcommand{\arrowdr}[2]{\filldraw[color=black!30!blue,draw=none] (#1+0.07071067811,#2-0.07071067811)--(#1-0.1202081528,#2+0.02121320343)--(#1-0.02121320343,#2+0.1202081528);}

\newcommand{\ribbonvertex}[3]{
\coordinate (h#3) at (#1,#2+1);
		\coordinate (t#3) at (#1,#2-1);
		\coordinate (ra#3) at (0.5+#1,#2+3/5);
		\coordinate (rb#3) at (0.5+#1, #2+1/5);
		\coordinate (rc#3) at (0.5+#1, #2-1/5);
		\coordinate (rd#3) at (0.5+#1,#2-3/5);
		\coordinate (la#3) at (-0.5+#1,#2+3/5);
		\coordinate (lb#3) at (-0.5+#1,#2+ 1/5);
		\coordinate (lc#3) at (-0.5+#1, #2-1/5);
		\coordinate (ld#3) at (-0.5+#1,#2-3/5);
	    \draw (h#3)--(t#3);
	    \draw (ld#3) cos (#1,#2) sin (ra#3); 
	    \draw (la#3)--(rb#3);
	    \draw (lb#3)--(rc#3);
	    \draw (lc#3)--(rd#3);
}

\newcommand{\ribbonvertexsix}[3]{
\coordinate (h#3) at (#1,#2+1.3);
		\coordinate (t#3) at (#1,#2-1.3);
		\coordinate (rz#3) at (0.5+#1,#2+1);
		\coordinate (ra#3) at (0.5+#1,#2+3/5);
		\coordinate (rb#3) at (0.5+#1, #2+1/5);
		\coordinate (rc#3) at (0.5+#1, #2-1/5);
		\coordinate (rd#3) at (0.5+#1,#2-3/5);
		\coordinate (re#3) at (0.5+#1,#2-1);
		\coordinate (lz#3) at (-0.5+#1,#2+1);
		\coordinate (la#3) at (-0.5+#1,#2+3/5);
		\coordinate (lb#3) at (-0.5+#1,#2+ 1/5);
		\coordinate (lc#3) at (-0.5+#1, #2-1/5);
		\coordinate (ld#3) at (-0.5+#1,#2-3/5);
		\coordinate (le#3) at (-0.5+#1,#2-1);
	    \draw (h#3)--(t#3);
	    \draw (le#3) cos (#1,#2) sin (rz#3); 
	    \draw (lz#3)--(ra#3);
	    \draw (la#3)--(rb#3);
	    \draw (lb#3)--(rc#3);
	    \draw (lc#3)--(rd#3);
	    \draw (ld#3)--(re#3);
}

\newcommand{\ribbonarrow}[3]{

}
\newcommand{\goodlookingvertex}[3]{
\coordinate (h#3) at (#1,#2+1);
		\coordinate (t#3) at (#1,#2-1);
		\coordinate (ra#3) at (#1+0.67,#2+0.67);
		\coordinate (rb#3) at (#1+0.67, #2+0.33);
		\coordinate (rc#3) at (#1+0.67, #2);
		\coordinate (rd#3) at (#1+0.67,#2-0.33);
		\coordinate (re#3) at (#1+0.67,#2-0.67);
		\coordinate (la#3) at (#1-0.67,#2+0.67);
		\coordinate (lb#3) at (#1-0.67, #2+0.33);
		\coordinate (lc#3) at (#1-0.67, #2);
		\coordinate (ld#3) at (#1-0.67,#2-0.33);
		\coordinate (le#3) at (#1-0.67,#2-0.67);
	    \draw [thick](h#3)--(t#3);
	    \draw (le#3) cos (#1,#2) sin (ra#3); 
	    \draw (la#3)--(rb#3);
	    \draw [dotted] (lb#3)--(rc#3);
	    \draw (lc#3)--(rd#3);
	    \draw (ld#3)--(re#3);
}

\newcommand{\spiderR}[3]{
\coordinate(0#3) at (#1,0+#2);
\coordinate (1#3) at (#1+1.2,0.4+#2);
	\coordinate (2#3) at (#1+1.2,0.8+#2);
	\coordinate (3#3) at (#1+1.2,1.2+#2);
	\coordinate (a#3) at (#1+1.2,1.6+#2);
	\coordinate (b#3) at (#1-1.2,1.6+#2);
	\coordinate (4#3) at (-1.2+#1,1.2+#2);
	\coordinate (5#3) at (-1.2+#1,0.8+#2);
	\coordinate (6#3) at (-1.2+#1,0.4+#2);
	\coordinate (7#3) at (-1.2+#1,-0.4+#2);
	\coordinate (8#3) at (-1.2+#1,-0.8+#2);
	\coordinate (9#3) at (-1.2+#1,-1.2+#2);
	\coordinate (c#3) at (#1-1.2,-1.6+#2);
	\coordinate (d#3) at (#1+1.2,-1.6+#2);
	\coordinate (10#3) at (1.2+#1,-1.2+#2);
	\coordinate (11#3) at (1.2+#1,-0.8+#2);
	\coordinate (12#3) at (1.2+#1,-0.4+#2);
	\draw (1#3)--(c#3);
	\draw (2#3)--(9#3);
	\draw (3#3)--(8#3);
	\draw (4#3)--(11#3);
	\draw (5#3)--(10#3);
	
	\draw (a#3)--(7#3);
	\draw (6#3)--(d#3);
	\draw (b#3)--(12#3);
	}

\newcommand{\ybevertexl}[0]{
\node at (-2.3,1.6) {$b_1$};
		\node at (-2.3,1.2) {$b_2$};
		\node at (-2.3,0.8) {$\cdots$};
		\node at (-2.3,-1.2) {$\cdots$};
		\node at (-2.3,-0.8) {$c_2$};
		\node at (-2.3,-0.4) {$c_1$};
		\node at (1.65,-1.6) {$\beta_n$};
		\node at (1.65,-1.2) {$\cdots$};
		\node at (1.65,-0.8) {$\beta_2$};
		\node at (1.65,1.2) {$\gamma_2$};
		\node at (1.65,0.8) {$\cdots$};
		\node at (1.65,0.4) {$\gamma_n$};
		\node [fill=white, inner sep= .1]at (-0.9,1.18) {$b_1$};
		\node [fill=white, inner sep= .2]at (-0.94,0.8) {$b_2$};
		\node [fill=white, inner sep= .1]at (-0.94,0.4) {$\cdots$};
		\node [fill=white, inner sep= .2]at (-0.9,-1.6) {$\cdots$};
		\node [fill=white, inner sep= .3]at (-0.9,-1.2) {$c_2$};
		\node [fill=white, inner sep= .3]at (-0.9,-0.8) {$c_1$};
}

\newcommand{\ybevertexr}[0]{
\node at (2.3,-1.6) {$\beta_n$};
		\node at (2.3,-1.2) {$\cdots$};
		\node at (2.3,-0.8) {$\beta_2$};
        \node at (2.3,1.2) {$\gamma_2$};
		\node at (2.3,0.8) {$\cdots$};
		\node at (2.3,0.4) {$\gamma_n$};
		\node [fill=white, inner sep= .3]at (0.93,1.6) {$\gamma_2$};
		\node [fill=white, inner sep= .1]at (0.95,1.2) {$\cdots$};
		\node [fill=white, inner sep= .3]at (0.93,0.8) {$\gamma_n$};
		\node [fill=white, inner sep= .3]at (0.94,-1.16) {$\beta_n$};
		\node [fill=white, inner sep= .1]at (0.94,-0.8) {$\cdots$};
		\node [fill=white, inner sep= .3]at (0.94,-0.36) {$\beta_2$};
		\node at (-1.65,-1.2) {$\cdots$};
		\node at (-1.65,-0.8) {$c_2$};
		\node at (-1.65,-0.4) {$c_1$};
		\node at (-1.65,1.6) {$b_1$};
		\node at (-1.65,1.2) {$b_2$};
		\node at (-1.65,0.8) {$\cdots$};
}

\newcommand{\lybearrows}[6]{
\ifthenelse{#1=1}{\arrowup{-1.5}{2}}{\arrowdown{-1.5}{2}}
\ifthenelse{#2=1}{\arrowleft{-2.1}{0.4}}{\arrowright{-2.1}{0.4}}
\ifthenelse{#3=1}{\arrowleft{-2.1}{-1.6}}{\arrowright{-2.1}{-1.6}}
\ifthenelse{#4=1}{\arrowup{-1.5}{-2}}{\arrowdown{-1.5}{-2}}
\ifthenelse{#5=1}{\arrowul{1.42}{-0.42}}{\arrowdr{1.45}{-0.45}}
\ifthenelse{#6=1}{\arrowdl{1.33}{1.55}}{\arrowur{1.345}{1.55}}
}

\newcommand{\lybearrowsin}[3]{
\ifthenelse{#1=1}{\arrowul{-0.943}{1.541}}{\arrowdr{-0.943}{1.557}}
\ifthenelse{#2=1}{\arrowup{-1.5}{0}}{\arrowdown{-1.5}{0}}
\ifthenelse{#3=1}{\arrowdl{-0.94}{1.55-1.89}}{\arrowur{-0.94}{1.55-1.91}}
}

\newcommand{\rybearrows}[6]{
\ifthenelse{#1=1}{\arrowup{1.5}{2}}{\arrowdown{1.5}{2}}
\ifthenelse{#2=1}{\arrowul{-1.42}{-0.42+0.84}}{\arrowdr{-1.38}{-0.45+0.84}}
\ifthenelse{#3=1}{\arrowdl{-1.35}{-1.55}}{\arrowur{-1.345}{-1.56}}
\ifthenelse{#4=1}{\arrowup{1.5}{-2}}{\arrowdown{1.5}{-2}}
\ifthenelse{#5=1}{\arrowleft{2.1}{-1.6+1.2}}{\arrowright{2.1}{-1.6+1.2}}
\ifthenelse{#6=1}{\arrowleft{2.1}{0.4+1.2}}{\arrowright{2.1}{0.4+1.2}}
}

\newcommand{\rybearrowsin}[3]{
\ifthenelse{#1=1}{\arrowdl{0.943}{1.556-1.2}}{\arrowur{0.935}{0.338}}
\ifthenelse{#2=1}{\arrowup{1.5}{0}}{\arrowdown{1.5}{0}}
\ifthenelse{#3=1}{\arrowul{0.94}{1.536-1.89-1.2}}{\arrowdr{0.935}{1.57-1.91-1.2}}
}

\newcommand{\ybeleft}[9]{
\tikz[baseline=.1ex,scale=1]{
		\spiderR{0.2}{0}{1}
		\ribbonvertex{-1.5}{1}{1}
		\ribbonvertex{-1.5}{-1}{1}
		\lybearrows{#1}{#2}{#3}{#4}{#5}{#6}
		\lybearrowsin{#7}{#8}{#9}
		\ybevertexl
		\node [fill=white, inner sep = 1]at (0.2,0){$\calL$};
		\node [fill=white, inner sep = 1]at (-1.5,1) {$v_j$};
		\node [fill=white, inner sep = 1]at (-1.5,-1) {$v_i$};
		}
		}
		
\newcommand{\ybeleftb}[9]{
\tikz[baseline=.1ex,scale=1]{
		\spiderR{0.2}{0}{1}
		\ribbonvertex{-1.5}{1}{1}
		\ribbonvertex{-1.5}{-1}{1}
		\lybearrows{#1}{#2}{#3}{#4}{#5}{#6}
		\lybearrowsin{#7}{#8}{#9}
		\ybevertexl
		\node [fill=white, inner sep = 1]at (0.2,0){$\calK$};
		\node [fill=white, inner sep = 1]at (-1.5,1) {$u_j$};
		\node [fill=white, inner sep = 1]at (-1.5,-1) {$u_i$};
		}
		}

\newcommand{\yberight}[9]{
\tikz[baseline=.1ex,scale=1]{
		\spiderR{-0.2}{0}{1}
		\ribbonvertex{1.5}{1}{1}
		\ribbonvertex{1.5}{-1}{1}
		\rybearrows{#1}{#2}{#3}{#4}{#5}{#6}
		\rybearrowsin{#7}{#8}{#9}
		\ybevertexr
		\node [fill=white, inner sep = 1]at (-0.2,0){$\calR$};
		\node [fill=white, inner sep = 1]at (1.5,1) {$w_i$};
		\node [fill=white, inner sep = 1]at (1.5,-1) {$w_j$};
		}
		}
		
\newcommand{\yberightb}[9]{
\tikz[baseline=.1ex,scale=1]{
		\spiderR{-0.2}{0}{1}
		\ribbonvertex{1.5}{1}{1}
		\ribbonvertex{1.5}{-1}{1}
		\rybearrows{#1}{#2}{#3}{#4}{#5}{#6}
		\rybearrowsin{#7}{#8}{#9}
		\ybevertexr
		\node [fill=white, inner sep = 1]at (-0.2,0){$\calQ$};
		\node [fill=white, inner sep = 1]at (1.5,1) {$y_i$};
		\node [fill=white, inner sep = 1]at (1.5,-1) {$y_j$};
		}
		}
		
\newcommand{\ybeleftn}[9]{
\tikz[baseline=.1ex,scale=0.8]{
		\spiderR{0.2}{0}{1}
		\ribbonvertex{-1.5}{1}{1}
		\ribbonvertex{-1.5}{-1}{1}
		\lybearrows{#1}{#2}{#3}{#4}{#5}{#6}
		\lybearrowsin{#7}{#8}{#9}
		}
		}

\newcommand{\yberightn}[9]{
\tikz[baseline=.1ex,scale=0.8]{
		\spiderR{-0.2}{0}{1}
		\ribbonvertex{1.5}{1}{1}
		\ribbonvertex{1.5}{-1}{1}
		\rybearrows{#1}{#2}{#3}{#4}{#5}{#6}
		\rybearrowsin{#7}{#8}{#9}
		}
		}

\newcommand{\north}{
\begin{tikzpicture}
\draw [Latex-,thick](0,1) -- (1,0);
\draw [-Latex,thick](0,0) -- (1,1);
\draw [-Latex,thick](-0.05,-0.05) -- (0.25,0.25);
\draw [thick](1.05,1.05) -- (0.75,0.75);
\draw [thick](-0.05,1.05) -- (0.25,0.75);
\draw [-Latex,thick](1.05,-0.05) -- (0.75,0.25);
\end{tikzpicture} 
}

\newcommand{\southsouth}{
\begin{tikzpicture}
\draw [thick](0,1) -- (1,0);
\draw [Latex-Latex,thick](0,0) -- (1,1);
\draw [thick](-0.05,-0.05) -- (0.25,0.25);
\draw [thick](1.05,1.05) -- (0.75,0.75);
\draw [-Latex,thick](-0.05,1.05) -- (0.25,0.75);
\draw [-Latex,thick](1.05,-0.05) -- (0.75,0.25);
\end{tikzpicture}}

\newcommand{\west}{
\begin{tikzpicture}
\draw [Latex-, thick](0,1) -- (1,0);
\draw [Latex-,thick](0,0) -- (1,1);
\draw [thick](-0.05,-0.05) -- (0.25,0.25);
\draw [-Latex, thick](1.05,1.05) -- (0.75,0.75);
\draw [thick](-0.05,1.05) -- (0.25,0.75);
\draw [-Latex,thick](1.05,-0.05) -- (0.75,0.25);
\end{tikzpicture}}

\newcommand{\east}{
\begin{tikzpicture}
\draw [-Latex, thick](0,1) -- (1,0);
\draw [-Latex,thick](0,0) -- (1,1);
\draw [-Latex,thick](-0.05,-0.05) -- (0.25,0.25);
\draw [thick](1.05,1.05) -- (0.75,0.75);
\draw [-Latex,thick](-0.05,1.05) -- (0.25,0.75);
\draw [thick](1.05,-0.05) -- (0.75,0.25);
\end{tikzpicture}}

\newcommand{\northnorth}{
\begin{tikzpicture}
\draw [Latex-Latex, thick](0,1) -- (1,0);
\draw [thick](0,0) -- (1,1);
\draw [-Latex,thick](-0.05,-0.05) -- (0.25,0.25);
\draw [-Latex, thick](1.05,1.05) -- (0.75,0.75);
\draw [thick](-0.05,1.05) -- (0.25,0.75);
\draw [thick](1.05,-0.05) -- (0.75,0.25);
\end{tikzpicture}
}

\newcommand{\south}{
\begin{tikzpicture}
\draw [-Latex, thick](0,1) -- (1,0);
\draw [Latex-, thick](0,0) -- (1,1);
\draw [thick](-0.05,-0.05) -- (0.25,0.25);
\draw [-Latex, thick](1.05,1.05) -- (0.75,0.75);
\draw [-Latex,thick](-0.05,1.05) -- (0.25,0.75);
\draw [thick](1.05,-0.05) -- (0.75,0.25);
\end{tikzpicture}
}

\newcommand{\northp}{
\begin{tikzpicture}
\draw [{Latex[color=red]}-,thick, red](0,1) -- (1,0);
\draw [-{Latex[color=red]},thick, red](0,0) -- (1,1);
\draw [-{Latex[color=red]},thick,red](-0.05,-0.05) -- (0.25,0.25);
\draw [thick,red](1.05,1.05) -- (0.75,0.75);
\draw [thick,red](-0.05,1.05) -- (0.25,0.75);
\draw [-{Latex[color=red]},thick,red](1.05,-0.05) -- (0.75,0.25);
\end{tikzpicture}
}

\newcommand{\southsouthp}{
\begin{tikzpicture}
\draw [thick](0,1) -- (0.5,0.5);
\draw [thick,red] (0.5,0.5) -- (1,0);
\draw [Latex-,thick] (0,0) -- (0.5,0.5);
\draw[-{Latex[color=red]},thick,red](0.5,0.5) -- (1,1);
\draw [thick](-0.05,-0.05) -- (0.25,0.25);
\draw [thick,red](1.05,1.05) -- (0.75,0.75);
\draw [-Latex,thick](-0.05,1.05) -- (0.25,0.75);
\draw [-{Latex[color=red]},thick,red](1.05,-0.05) -- (0.75,0.25);
\end{tikzpicture}
}

\newcommand{\westp}{
\begin{tikzpicture}
\draw [{Latex[color=red]}-, thick,red](0,1) -- (1,0);
\draw [Latex-,thick](0,0) -- (1,1);
\draw [thick](-0.05,-0.05) -- (0.25,0.25);
\draw [-Latex, thick](1.05,1.05) -- (0.75,0.75);
\draw [thick,red](-0.05,1.05) -- (0.25,0.75);
\draw [-{Latex[color=red]},thick,red](1.05,-0.05) -- (0.75,0.25);
\end{tikzpicture}}

\newcommand{\eastp}{
\begin{tikzpicture}
\draw [-Latex, thick](0,1) -- (1,0);
\draw [-{Latex[color=red]},thick,red](0,0) -- (1,1);
\draw [-{Latex[color=red]},thick,red](-0.05,-0.05) -- (0.25,0.25);
\draw [thick,red](1.05,1.05) -- (0.75,0.75);
\draw [-Latex,thick](-0.05,1.05) -- (0.25,0.75);
\draw [thick](1.05,-0.05) -- (0.75,0.25);
\end{tikzpicture}}

\newcommand{\northnorthp}{
\begin{tikzpicture}
\draw [{Latex[color=red]}-,thick,red] (0,1) -- (0.5,0.5);
\draw [-Latex, thick](0.5,0.5) -- (1,0);
\draw [thick,red](0,0) -- (0.5,0.5);
\draw [thick](0.5,0.5) -- (1,1);
\draw [-{Latex[color=red]},thick,red](-0.05,-0.05) -- (0.25,0.25);
\draw [-Latex, thick](1.05,1.05) -- (0.75,0.75);
\draw [thick,red](-0.05,1.05) -- (0.25,0.75);
\draw [thick](1.05,-0.05) -- (0.75,0.25);
\end{tikzpicture}
}

\newcommand{\southp}{
\begin{tikzpicture}
\draw [-Latex, thick](0,1) -- (1,0);
\draw [Latex-, thick](0,0) -- (1,1);
\draw [thick](-0.05,-0.05) -- (0.25,0.25);
\draw [-Latex, thick](1.05,1.05) -- (0.75,0.75);
\draw [-Latex,thick](-0.05,1.05) -- (0.25,0.75);
\draw [thick](1.05,-0.05) -- (0.75,0.25);
\end{tikzpicture}
}

\newcommand{\calK}[0]{\mathcal{K}}
\newcommand\Top{\rule{0pt}{2.6ex}}       
\newcommand\Bottom{\rule[-1.2ex]{0pt}{0pt}} 

\begin{document}

\title{A Lattice Model for super LLT Polynomials}
\author[M. J. Curran]{Michael J. Curran}
\address{Mathematical Institute, University of Oxford, Oxford, OX2 6GG, United Kingdom.}
\email{michael.curran@maths.ox.ac.uk} 

\author[C. Frechette]{Claire Frechette}
\address{School of Mathematics, University of Minnesota, Minneapolis, MN 55455, United States}
\email{frech014@umn.edu}

\author[C. Yost--Wolff]{Calvin Yost--Wolff}
\address{}
\email{}

\author[S. W. Zhang]{Sylvester W. Zhang}
\address{School of Mathematics, University of Minnesota, Minneapolis, MN 55455, United States.}
\email{swzhang@umn.edu}

\author[V. Zhang]{Valerie Zhang}
\address{}
\email{}

\date{} 
\maketitle
\thispagestyle{empty}

\begin{abstract}
We introduce a solvable lattice model for supersymmetric LLT polynomials, also known as \emph{super LLT polynomials}, based upon particle interactions in super $n$-ribbon tableaux. Using operators on a Fock space, we prove a Cauchy identity for super LLT polynomials, simultaneously generalizing the Cauchy and dual Cauchy identities for LLT polynomials. Lastly, we construct a solvable semi-infinite Cauchy lattice model with a surprising Yang-Baxter equation and examine its connections to the Cauchy identity.
\end{abstract}


\section{Introduction}

LLT polynomials, also known as \emph{ribbon functions}, are a $q$-analog of products of Schur functions and were first introduced by Lascoux, Leclerc, and Thibon in \cite{LLT97} as a family of  polynomials realized as generating functions over $n$-ribbon tableaux. Like many other families of functions which originated over tableaux, LLT polynomials have since been shown to satisfy certain nice properties:
\begin{itemize}
    \item they are symmetric,
    \item they may be written as operators on a Fock space representation, in this case that of the quantum group $U_q(\mathfrak{sl}_n)$,
    \item they satisfy Cauchy and Dual Cauchy identities, and
    \item they may be realized as the partition function of a lattice model.
\end{itemize}

The Cauchy and Dual Cauchy identities were first proved by Thomas Lam in \cite{lam-ribbon-heisenberg}, in which Lam also constructed a supersymmetric analogue of LLT polynomials, called \emph{superLLT polynomials}, defined as a generating function over super $n$-ribbon tableaux. By adding an extra set of parameters, super LLT polynomials enable us to track vertical and horizontal information about a tableaux at the same time, corresponding to considering two different types of operators on the Fock space. In this paper, we prove a general Cauchy/Dual Cauchy identity for the super LLT polynomials, making a supersymmetric analogue of the approach using in \cite{LamBF} to prove general Cauchy identities for symmetric polynomials using Heisenberg algebras. In fact, we find that the Cauchy identity for super LLT polynomials specializes to both the Cauchy and Dual Cauchy identities for LLT polynomials. This remarkable property comes from the fact that the super LLT polynomials may be specialized to LLT polynomials in two separate ways, corresponding to the relationship between a partition and its conjugate, or equivalently to the duality between horizontal and vertical ribbon strips. Consequently, a Dual Cauchy identity for super LLT polynomials is merely a rephrasing of the Cauchy identity. Our Cauchy identity also generalizes the Cauchy identity for metaplectic symmetric functions proven by Brubaker, Buciumas, Bump, and Gustafsson in \cite{BBBG}.

A lattice model for LLT polynomials was first developed by Curran, Yost-Wolff, Zhang, and Zhang \cite{LLTREU} as part of the 2019 REU at the University of Minnesota, using the definition over $n$-ribbon tableaux. This model was proven solvable for n=1,2,3, and was conjectured to be solvable for all $n$. Independently, Corteel, Gitlin, Keating, Meza \cite{CorteelCo} developed a lattice model for coinversion LLT polynomials using an alternate definition as generating polynomials over tuples of skew tableaux. These tuples are related to ribbon tableaux by a weight-preserving bijection developed by Stanton and White \cite{stantonwhite}. Since the coinversion LLT polynomials give the LLT polynomials up to a correction factor of $q$ depending only on the shape $\lambda/\mu$, these lattice models are equivalent under this bijection up to that factor of $q$, but the underlying structures of each reveal different properties of the polynomials. Later, Aggarwal, Borodin and Wheeler \cite{ABW} developed a general lattice model that specializes to both of the above models, which they use to reprove many interesting combinatorial identities satisfied by these and similar families of symmetric polynomials.

In this paper, we generalize the Curran-Yost-Wolff-Zhang-Zhang model to a novel lattice model whose partition function produces the super LLT polynomials. We then prove that this super model is solvable, thus covering the solvability conjecture of \cite{LLTREU}, and use this solvability to provide a lattice model proof of certain symmetry properties of the super LLT polynomials. Recently, Gitlin and Keating \cite{GitlinKeating} have independently generalized the coinversion LLT model of \cite{CorteelCo} to produce coinversion supersymmetric LLT polynomials; see section 5 of \cite{GitlinKeating} for a discussion of the relationship between our supersymmetric models.

In Section \ref{sec:tableaux}, we recall the details of the definition of super LLT polynomials as a generating function over $n$-ribbon tableaux. In Section \ref{sec:heisenberg}, we reformulate these polynomials in terms of operators generated from a Heisenberg algebra acting on a Fock space and prove Cauchy and Dual Cauchy identities of super LLT polynomials using these additional operators. Changing gears, Section \ref{sec:latticemodel} explains the (super)ribbon lattice model and proves that its partition function gives the super LLT polynomial. We then show in Section \ref{sec:solvability} that this model is solvable for all $n$ and use this solvability to reprove interesting symmetry relations on the super LLT polynomials. Finally, Section \ref{sec:cauchyish} examines the implications of this model for Cauchy identities on semi-infinite lattices.

\subsection*{Acknowledgements}

This research project was started as part of the 2019 Combinatorics REU program at the School of Mathematics of the University of Minnesota - Twin Cities, supported by NSF RTG grant DMS-1148634. The authors would like to thank Ben Brubaker and Katy Weber for their mentorship and guidance.

\section{Background on Tableaux and Symmetric Functions}\label{sec:tableaux}

Like many other interesting symmetric functions, super LLT polynomials arise as generating functions over a specific class of tableaux. 

A \emph{partition} $\l$ is any decreasing sequence $\l = (\l_1, \l_2, \ldots, \l_r)$ of nonnegative integers. We denote the \textit{length}, or number of parts, of $\l$ by $\ell(\l)$ and  the \textit{size} of $\l$ by $| \l | = \l_1 + \l_2 + \ldots + \l_r$. We will occasionally use the notation $\lambda \vdash n$ to say $\lambda$ partitions $n$. A \emph{composition} $\a = (\a_1,...,\a_r)$ is a non-ordered list of nonnegative integers; it is important not to confuse these with partitions. We will also define $m_k(\lambda)$ to be the number of parts $\lambda_i = k$ and set $\lambda'$ to be the conjugate partition of $\lambda$.

A partition $\l = (\l_1, \l_2, \ldots, \l_r)$  can be visualized by its \emph{Young diagram}, which consists of horizontal boxes arranged in left-justified rows, where  the $i^\text{th}$ row contains $\l_i$ boxes; abusing notation, we will frequently conflate a partition with its Young diagram. If $\mu \subset \l$ for two partitions $\l,\mu$, that is if $\mu_i\leq \l_i$ for all $i$, then $\l/\mu$ is a \textit{skew shape} or \emph{skew partition} with size $|\l/\mu| = |\l| - |\mu|.$ We will identify a partition $\l$ with the skew shape $\l/\emptyset$, and will also occasionally identify partitions that differ only by a sequence of parts $\lambda_i = 0$ when it should not cause any confusion. 

Fillings of the Young diagram according to a given alphabet are called \emph{tableaux}: there are many different kinds of tableaux, each defined by different rules for how one fills the shape. In this paper, we focus on \emph{ribbon tableaux}, first defined by Stanton and White in \cite{stantonwhite}. Let $n$ be a fixed positive integer. An \emph{$n$-ribbon} is a skew shape $r$ containing $n$ boxes that is connected and contains no 2 by 2 squares. The \textit{spin} of a ribbon $r$ is defined to be the height of $r$ minus one, and is denoted $\spin(r).$ 

\begin{example} Of the six arrangements possible of 3 boxes, four of them are 3-ribbons, which we see below labelled by their spin. The last two are not $3$-ribbons because they cannot be realized as skew shapes $\l/\mu$. 
\begin{center}
\begin{tabular}{|cccc|cc|}
\hline
\multicolumn{4}{|c|}{$3$-ribbons} & \multicolumn{2}{c|}{non-ribbons}
\\
\hline
&&&&&\\
  \begin{tikzpicture}[scale=0.5]
  	\foreach \i in {0,...,5}{
  	\foreach \j in {0,...,5}{
  	\coordinate (\i\j) at (\i,\j);
  	}
  	}
  	\draw (00)--(10)--(13)--(03)--(00);
  \end{tikzpicture} &
  \begin{tikzpicture}[scale=0.5]
  	\foreach \i in {0,...,5}{
  	\foreach \j in {0,...,5}{
  	\coordinate (\i\j) at (\i,\j);
  	}
  	}
  	\draw (01)--(11)--(12)--(22)--(23)--(03)--cycle;
  \end{tikzpicture}&
  \begin{tikzpicture}[scale=0.5,rotate=90]
  	\foreach \i in {0,...,5}{
  	\foreach \j in {0,...,5}{
  	\coordinate (\i\j) at (\i,\j);
  	}
  	}
  	\draw (00)--(10)--(13)--(03)--(00);
  \end{tikzpicture} 
  &
  \begin{tikzpicture}[scale=0.5,rotate=180]
  	\foreach \i in {0,...,5}{
  	\foreach \j in {0,...,5}{
  	\coordinate (\i\j) at (\i,\j);
  	}
  	}
  	\draw (01)--(11)--(12)--(22)--(23)--(03)--cycle;
  \end{tikzpicture}
  & 
   \begin{tikzpicture}[scale=0.5,rotate=-90]
  	\foreach \i in {0,...,5}{
  	\foreach \j in {0,...,5}{
  	\coordinate (\i\j) at (\i,\j);
  	}
  	}
  	\draw (01)--(11)--(12)--(22)--(23)--(03)--cycle;
  \end{tikzpicture}&
   \begin{tikzpicture}[scale=0.5,rotate=90]
  	\foreach \i in {0,...,5}{
  	\foreach \j in {0,...,5}{
  	\coordinate (\i\j) at (\i,\j);
  	}
  	}
  	\draw (01)--(11)--(12)--(22)--(23)--(03)--cycle;
  \end{tikzpicture}
  	
    \\
  2 & 1& 0& 1&&\\
  \hline
  
\end{tabular}
\end{center}
\end{example}

Given a skew partition $\l/\mu$, a tiling of $\l/\mu$ by $n$-ribbons is called a \emph{$n$-horizontal strip} if
the top-right-most square of each ribbon touches the northern boundary of $\l/\mu$ (\cref{fig:h-strip-example}). A tiling is called a \emph{n-vertical strip} if the bottom-left-most square of each ribbon lies along the northern boundary of $\lambda/\mu$. Equivalently, a vertical strip is a collection of ribbons that, when collectively flipped over the $y=-x$ antidiagonal, produces a $n$-horizontal strip.

\begin{figure}[h!]
\centering
\begin{tikzpicture}[scale=0.7]
	\draw (0,0)--(3,0)--(3,1)--(4,1)--(4,2)--(1,2)--(1,1)--(0,1)--cycle;
	\draw (2,0)--(2,2);
	\draw [green,thick] (1.5,1.5) circle (0.3); 
	\draw [green,thick] (3.5,1.5) circle (0.3);
\end{tikzpicture}
	\quad\quad\quad
\begin{tikzpicture}[scale=0.7]
	\draw (0,0)--(3,0)--(3,1)--(4,1)--(4,2)--(1,2)--(1,1)--(0,1)--cycle;
	\draw (1,1)--(3,1);
	\draw [red,thick] (2.3,0.7)--(2.7,0.3);
	\draw [red,thick] (2.3,0.3)--(2.7,0.7);
	\draw [green,thick] (3.5,1.5) circle (0.3);
\end{tikzpicture}
	\quad\quad\quad
\begin{tikzpicture}[scale=0.7]
	\draw (0,0)--(0,3)--(1,3)--(1,4)--(2,4)--(2,1)--(1,1)--(1,0)--cycle;
	\draw (0,2)--(2,2);
	\draw [green,thick] (0.5,0.5) circle (0.3); 
	\draw [green,thick] (0.5,2.5) circle (0.3);
\end{tikzpicture}
\quad\quad\quad
\begin{tikzpicture}[scale=0.7]
	\draw (0,0)--(0,3)--(1,3)--(1,4)--(2,4)--(2,1)--(1,1)--(1,0)--cycle;
	\draw (1,1)--(1,3);
	\draw [green,thick] (0.5,0.5) circle (0.3);
	\draw [red,thick] (1.3,1.7)--(1.7,1.3);
	\draw [red,thick] (1.3,1.3)--(1.7,1.7);
\end{tikzpicture}
\caption{From left to right: a $3$-horizontal strip, a non-example of a horizontal strip, a $3$-vertical strip, and a non-example of a vertical strip.}
\label{fig:h-strip-example}
\end{figure}
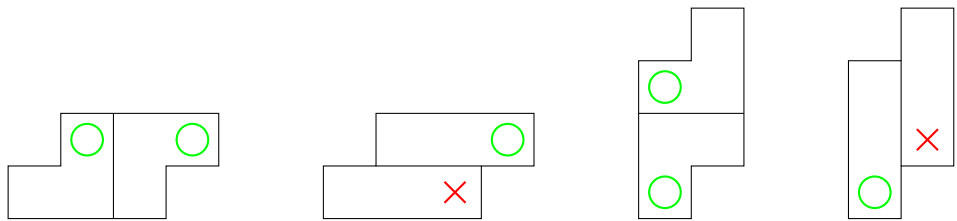

A tiling of $\lambda/\mu$ by $n$-ribbons labelled with positive integers is called a \emph{ribbon tableau}. Given such a tableau, we can define a sequence of partitions $\mu = \l^0 \subset \l^1 \subset \cdots \subset \l^{\ell} = \l$, where $\l_i$ consists of the subshape of the tableau with labels less than or equal to $i$.

\begin{defn}
A \textit{semistandard $n$-ribbon tableau} (SSRT) of skew shape $\l/\mu$ is a tiling of $\l/\mu$ with $n$-ribbons such that the induced sequence $\mu = \l^0 \subset \l^1 \subset \cdots \subset \l^{\ell} = \l$ has the property that for each $i$ the skew shape $\l_i/\l_{i-1}$ is a {horizontal $n$-ribbon strip}.
 We define the \textit{weight} of such a tableau to be  the composition $\wt(T)$ such that $$\wt(T)_i = \#\{ n\text{-ribbons labelled }i\text{ in }T\}$$ and we define the \textit{spin} of $T$ to be the sum of the spins of the $n$-ribbon tiles of $T$. The set of all semistandard $n$-ribbon tableaux is denoted by $\SSRT_n(\l)$. Note that is not always possible to tile a given partition $\l$ with $n$-ribbons, and we will restrict our attention to partitions or skew shapes that are tileable by $n$-ribbons.
\end{defn}

We can now define the polynomials formerly known as \emph{$n$-ribbon Schur functions}, first introduced by Lascoux, Leclerc, and Thibon in \cite{LLT97}.
\begin{defn}
Let $n \geq 1$ be fixed and $\l/\mu$ a skew shape tileable by $n$-ribbons. Then the \textit{($n$-)LLT polynomial of shape $\l/\mu$} is defined as the generating function
\begin{equation}
    \G_{\l/\mu}^{(n)}(X;q) = \sum_{T \in \SSRT_n(\l/\mu)} q^{\spin(T)} x^{\wt(T)}.
\end{equation}
\end{defn}

\begin{example}\label{exampleLLT}
For example, suppose that $\l = (3,3)$ and $n = 2$. The only possible 2-ribbon tableaux with shape $\l$ are those depicted below.

\begin{center}
\begin{tikzpicture}[scale = 0.5]
\draw (0,0)--(0,2)--(3,2)--(3,0)--cycle;
    \draw (1,0)--(1,2);
    \draw (2,0)--(2,2);
    \draw (0.5,1) node {$1$};
    \draw (1.5,1) node {$1$};
    \draw (2.5,1) node {$1$};
    
\draw (3.5,0)--(3.5,2)--(6.5,2)--(6.5,0)--cycle;
    \draw (4.5,0)--(4.5,2);
    \draw (5.5,0)--(5.5,2);
    \draw (4,1) node {$1$};
    \draw (5,1) node {$1$};
    \draw (6,1) node {$2$};
    
\draw (7,0)--(7,2)--(10,2)--(10,0)--cycle;
    \draw (8,0)--(8,2);
    \draw (9,0)--(9,2);
    \draw (7.5,1) node {$1$};
    \draw (8.5,1) node {$2$};
    \draw (9.5,1) node {$2$};
    
\draw (10.5,0)--(10.5,2)--(13.5,2)--(13.5,0)--cycle;
    \draw (11.5,0)--(11.5,2);
    \draw (12.5,0)--(12.5,2);
    \draw (11,1) node {$2$};
    \draw (12,1) node {$2$};
    \draw (13,1) node {$2$};
    
\draw (14,0)--(14,2)--(17,2)--(17,0)--cycle;
    \draw (15,0)--(15,2);
    \draw (15,1)--(17,1);
    \draw (14.5,1) node {$1$};
    \draw (16,1.5) node {$1$};
    \draw (16,0.5) node {$2$};
    
\draw (17.5,0)--(17.5,2)--(20.5,2)--(20.5,0)--cycle;
    \draw (19.5,0)--(19.5,2);
    \draw (17.5,1)--(19.5,1);
    \draw (18.5,1.5) node {$1$};
    \draw (18.5,0.5) node {$2$};
    \draw (20,1) node {$2$};
\end{tikzpicture}
\end{center}
Then,
\begin{equation*} 
    \G_{(3,3)}^{(2)}(x_1, x_2; q) = q^3({x_1}^3 + {x_1}^2x_2 + x_1{x_2}^2 + {x_2}^3) + q({x_1}^2x_2 + x_1{x_2}^2).
\end{equation*}
\end{example}

\noindent We will generally write $\G_{\l/\mu}^{(n)}(X;q) = \G_{\l/\mu}(X;q)$ when $n$ is understood to be fixed. Note: the LLT polynomials are $q$-analogues of the Schur functions in the sense that $\G_{\l/\mu}(X;1)$ is equal to a product of $n$-Schur functions, as proven in \cite{LLT97}.

The central object of this paper is the supersymmetric generalization of the LLT polynomials, known as \emph{super LLT polynomials}. As with LLT polynomials, these new polynomials may be expressed as a generating function over tableaux, in this case \emph{super ribbon tableaux}, which were first defined in Lam in \cite{lam-ribbon-heisenberg}.

\begin{defn}
Choose a total order $\prec$ on two ordered alphabets $A, A'$. Let a \emph{super $n$-ribbon tableau} of shape $\lambda/\mu$ be a tiling by $n$-ribbons labeled in the alphabets $A$ and $A'$ such that the ribbon shapes labeled by $a \in A$ form horizontal ribbon strips and those labeled by $b \in A'$ form vertical ribbon strips. Furthermore, we require that the skew shapes generated by this labelling respect the total order on $A$ and $A'$. That is to say, the shape formed by removing the ribbons $\succ i$ gives a skew shape $\lambda_{\prec i}/\mu$ for every $i \in A\cup A'$. We denote the set of super $n$-ribbon tableaux of shape $\l/\mu$ by $^s\negthinspace RT_n(\lambda/\mu)$.
\end{defn}

In this paper, we will consistently choose the first alphabet to be a subset of our usual alphabet of positive integers, so $A = \{1,2,3,...,r\}$ under the usual ordering, and set $A' = \{ 1',2',...,s'\}$ under a similar ordering. The total ordering between the two will vary, although two common choices we will use in examples will be $1 < 1' < 2 < 2' < \cdots$ and $1 < 2 < \cdots < r < 1' < 2' < \cdots <s'$.

\begin{defn} Fix $n \geq 1$ and let $\lambda/\mu$ be a shape tileable by $n$-ribbons. 
The \emph{super LLT polynomial}, or \emph{super ribbon function}, $\mathcal{G}_{\lambda/\mu}^{(n)}(X/Y;q)$ is the generating function
\[ \mathcal{G}_{\lambda/\mu}^{(n)}(X/Y;q) = \sum_{T\in ^s\negthinspace RT_n(\lambda/\mu)} q^{\spin(T)}x^{\wt(T)}(-y)^{\wt'(T)}\]
where $\wt(T),\wt'(T)$ are the weights in the alphabets $A,A'$ respectively.
\end{defn}

\begin{example}\label{examplesuperLLT}
Suppose that $\lambda = (3,3), \mu = \emptyset$ and $n=2$. Let $A = \{1,2\}$ and $A' = \{1'\}$ under the total order $1 < 1' < 2$. The super ribbon tableaux in $^{s}RT(\lambda/\mu)$ for this ordering, which include those from Example \ref{exampleLLT}, are 
\begin{center}
    \begin{tikzpicture}[scale = 0.5]
\draw (0,0)--(0,2)--(3,2)--(3,0)--cycle;
    \draw (1,0)--(1,2);
    \draw (2,0)--(2,2);
    \draw (0.5,1) node {$1$};
    \draw (1.5,1) node {$1$};
    \draw (2.5,1) node {$1$};
    
\draw (3.5,0)--(3.5,2)--(6.5,2)--(6.5,0)--cycle;
    \draw (4.5,0)--(4.5,2);
    \draw (5.5,0)--(5.5,2);
    \draw (4,1) node {$1$};
    \draw (5,1) node {$1$};
    \draw (6,1) node {$2$};
    
\draw (7,0)--(7,2)--(10,2)--(10,0)--cycle;
    \draw (8,0)--(8,2);
    \draw (9,0)--(9,2);
    \draw (7.5,1) node {$1$};
    \draw (8.5,1) node {$2$};
    \draw (9.5,1) node {$2$};
    
\draw (10.5,0)--(10.5,2)--(13.5,2)--(13.5,0)--cycle;
    \draw (11.5,0)--(11.5,2);
    \draw (12.5,0)--(12.5,2);
    \draw (11,1) node {$2$};
    \draw (12,1) node {$2$};
    \draw (13,1) node {$2$};
    
\draw (14,0)--(14,2)--(17,2)--(17,0)--cycle;
    \draw (15,0)--(15,2);
    \draw (16,0)--(16,2);
    \draw (14.5,1) node {$1$};
    \draw (15.5,1) node {$1$};
    \draw (16.5,1) node {$1'$};

\draw (17.5,0)--(17.5,2)--(20.5,2)--(20.5,0)--cycle;
    \draw (18.5,0)--(18.5,2);
    \draw (19.5,0)--(19.5,2);
    \draw (18,1) node {$1$};
    \draw (19,1) node {$1'$};
    \draw (20,1) node {$2$};
    
\draw (21,0)--(21,2)--(24,2)--(24,0)--cycle;
    \draw (22,0)--(22,2);
    \draw (23,0)--(23,2);
    \draw (21.5,1) node {$1'$};
    \draw (22.5,1) node {$2$};
    \draw (23.5,1) node {$2$};

\draw (0,-2.5)--(0,-0.5)--(3,-0.5)--(3,-2.5)--cycle;
    \draw (1,-2.5)--(1,-0.5);
    \draw (1,-1.5)--(3,-1.5);
    \draw (0.5,-1.5) node {$1$};
    \draw (2,-1) node {$1$};
    \draw (2,-2) node {$2$};
    
\draw (3.5,-2.5)--(3.5,-0.5)--(6.5,-0.5)--(6.5,-2.5)--cycle;
    \draw (4.5,-2.5)--(4.5,-0.5);
    \draw (4.5,-1.5)--(6.5,-1.5);
    \draw (4,-1.5) node {$1$};
    \draw (5.5,-1) node {$1$};
    \draw (5.5,-2) node {$1'$};
    
\draw (7,-2.5)--(7,-0.5)--(10,-0.5)--(10,-2.5)--cycle;
    \draw (8,-0.5)--(8,-2.5);
    \draw (8,-1.5)--(10,-1.5);
    \draw (7.5,-1.6) node {$1$};
    \draw (9,-1) node {$1'$};
    \draw (9,-2) node {$2$};
    
\draw (10.5,-0.5)--(10.5,-2.5)--(13.5,-2.5)--(13.5,-0.5)--cycle;
    \draw (11.5,-0.5)--(11.5,-2.5);
    \draw (11.5,-1.5)--(13.5,-1.5);
    \draw (11,-1.5) node {$1$};
    \draw (12.5,-1) node {$1'$};
    \draw (12.5,-2) node {$1'$};
    
\draw (14,-0.5)--(14,-2.5)--(17,-2.5)--(17,-0.5)--cycle;
    \draw (14,-1.5)--(16,-1.5);
    \draw (16,-0.5)--(16,-2.5);
    \draw (15,-1) node {$1$};
    \draw (15,-2) node {$2$};
    \draw (16.5,-1.6) node {$2$};

\draw (17.5,-0.5)--(17.5,-2.5)--(20.5,-2.5)--(20.5,-0.5)--cycle;
    \draw (17.5,-1.5)--(19.5,-1.5);
    \draw (19.5,-0.5)--(19.5,-2.5);
    \draw (18.5,-1) node {$1$};
    \draw (18.5,-2) node {$1'$};
    \draw (20,-1.5) node {$2$};
    
\draw (21,-0.5)--(21,-2.5)--(24,-2.5)--(24,-0.5)--cycle;
    \draw (21,-1.5)--(23,-1.5);
    \draw (23,-0.5)--(23,-2.5);
    \draw (22,-1) node {$1'$};
    \draw (22,-2) node {$2$};
    \draw (23.5,-1.5) node {$2$};
    
\draw (24.5,-0.5)--(24.5,-2.5)--(27.5,-2.5)--(27.5,-0.5)--cycle;
    \draw (24.5,-1.5)--(26.5,-1.5);
    \draw (26.5,-0.5)--(26.5,-2.5);
    \draw (25.5,-1) node {$1'$};
    \draw (25.5,-2) node {$1'$};
    \draw (27,-1.5) node {$2$};
    
\end{tikzpicture}
\end{center}

We then compute that
\begin{align*}
    \mathcal{G}^{(2)}_{(3,3)}(X/Y;q) &= q^3(x_1^3 + x_1^2x_2 +x_1x_2^2 + x_2^3 -x_1^2y_1 - x_1x_2y_1 - x_2^2y_1)\\ &\quad\quad + q(x_1^2x_2 + x_1x_2^2 -x_1^2y_1 - 2x_1x_2y_1 - x_2^2y_1 + x_1y_1^2 + x_2y_1^2).
\end{align*}
\end{example}

Observe that this super LLT polynomial is symmetric in the $X$ and $Y$ variables separately. Also, while the set of tableaux in Example \ref{examplesuperLLT} depends on the total order, it turns out the super LLT polynomial does not: the reader may find it interesting to compute the tableaux for the total order $1 < 2 < 1'$ and note that they generate the same polynomial. These are in fact both general properties of super LLT polynomials, and were originally proven in \cite{lam-ribbon-heisenberg} using the machinery we will discuss in Section \ref{sec:heisenberg}. We also reprove them using the solvability of lattice models in Section \ref{sec:latticemodel}.

\begin{prop}[\cite{lam-ribbon-heisenberg}, Prop 30]\label{prop:orderdoesntmatter}
The function $\mathcal{G}_{\lambda/\mu}(X/Y;q)$ is symmetric in each of $X$ and $Y$ and does not depend on the total order fixed between $A$ and $A'$.
\end{prop}

It would be natural to suppose that the super LLT polynomials for a skew partition $\lambda/\mu$ and its conjugate skew partition $\lambda'/\mu'$ are closely related, since flipping a super ribbon tableau results in a new super ribbon tableau for the conjugate shape. Combining this with Proposition \ref{prop:orderdoesntmatter} gives us a precise relationship between the two polynomials.

\begin{prop}\label{prop:flipover} Let $\lambda'$ denote the conjugate partition of $\lambda$. For any $n$ and any skew partition $\lambda/\mu$,
\[ \mathcal{G}_{\lambda/\mu}(X/Y;q) = (-1)^{\frac{|\lambda-\mu|}{n}} q^{(n-1)\frac{|\lambda-\mu|}{n}}\mathcal{G}_{\lambda'/\mu'}(Y/X;q^{-1}).\]
\end{prop}

\begin{proof}
Consider the left hand side as a generating function over super ribbon tableaux of shape $\lambda/\mu$ under the alphabet $A, A'$. Given one of these tableau $T$, flipping it over the anti-diagonal produces a new ribbon tableau $T'$ of shape $\lambda'/\mu'$ with the alphabets flipped, since horizontal ribbon strips become vertical strips and vice versa. We then compare the weights of $T$ and $T'$. In $T$, ribbons labeled $i\in A$ add a factor of $x_i$ and those labeled $j \in A'$ add a factor of $(-y_{j})$. In $T'$, the alphabets swap roles, which amounts to changing the sign of each ribbon. That is, if we let $wt(-)$ be the weight in the horizontal alphabet and $wt'(-)$ that in the vertical alphabet on each side, we have
\[ x^{wt(T)}(-y)^{wt'(T)} = (-1)^{\frac{|\lambda-\mu|}{n}} (-x)^{wt'(T')}y^{wt(T')}.\]
For the sake of clarity, we maintain the same total order on our alphabets, regardless of the fact that their roles have changed, since by Proposition \ref{prop:orderdoesntmatter} the polynomials are independent of the chosen ordering.
Regarding the power of $q$, each individual ribbon $r$ has the property that its spin satisfies $s(r) + s(r') = n-1$, so the full ribbon tableaux are related by 
\[ s(T) + s(T') = (n-1)\frac{|\lambda-\mu|}{n}.\]
\end{proof}

\section{Heisenberg Algebras and the Cauchy Identity}\label{sec:heisenberg}

Many symmetric polynomials may be represented using representations of a Heisenberg algebra generalized from the fermion side of the classical Boson-Fermion correspondence. In this context, elements of the algebra become operators on a Fock space indexed by partitions $\lambda$, encoding the algebraic framework behind adding or removing boxes in the generating tableaux shapes. Examining the interactions of these operators provides elegant proofs of identities that are difficult or impossible to prove from the tableaux definition. Conversely, families of polynomials generated by such operators may frequently be written as generating functions over tableaux satisfying a similar list of identities (see \cite{LamBF} for a convenient reference). In this section, we use the operator definition for super LLT polynomials to prove a Cauchy identity.

In order to properly construct and examine these operators, we first recall a few concepts from symmetric function theory. Let $\Lambda$ denote the ring of symmetric functions over $\mathbb{C}$. As is standard, let $h_k$ denote the homogeneous symmetric functions, $e_k$ the elementary symmetric functions, and $p_k$ the power sum symmetric functions. Define $p_\lambda := p_{\lambda_1}p_{\lambda_2}\cdots p_{\lambda_r}$ for $\lambda = (\lambda_1,\lambda_2...,\lambda_r)$. Noting that $p_\lambda$ form a basis for $\Lambda$, the expansions for $h_k$ and $e_k$ are
\begin{align*}
    h_k = \sum_{\lambda \vdash k} z_\lambda^{-1} p_\lambda
\hspace{2cm}
    e_k = (-1)^k \sum_{\lambda \vdash k} (-1)^{\ell(\lambda)} z_\lambda^{-1} p_\lambda
\end{align*}
where $z_\lambda = 1^{m_1(\lambda)} m_1(\lambda)! \cdots n^{m_n(\lambda)} m_n(\lambda)!$.

For the remainder of this section, fix a positive integer $n$. Let the Heisenberg algebra $H = H[a_i]$ be the associative algebra with identity generated by the set $\{B_k : k \in \Z \setminus 0\}$ with the relation
\begin{align}
     [B_j,B_k] = j\cdot a_k \cdot \delta_{j,-k}
\end{align}\label{eq:commutator}
\noindent with $a_k := \frac{1-q^{2nk}}{1-q^{2k}} = 1 + q^{2k} + \cdots + q^{2k(n-1)}$. This algebra and its representations are well studied, first by Lascoux, Leclerc, and Thibon \cite{LLT97} and then by Lam for generic $a_i$ \cite{LamBF}. Here, we only need the original representation used by LLT on the Fock space $\boldsymbol{F}$ generated over $\mathbb{C}$ by the symbols $|\lambda\rangle$ for partitions $\lambda$. This space is equipped with an inner product $\langle -,-\rangle$, so for notational convenience, we will simplify $|\lambda\rangle$ to $\lambda$ when it appears within the inner product, which is defined by $\langle \lambda, \mu\rangle = \delta_{\lambda \mu}$. It is useful to note that the actions of $B_k$ and $B_{-k}$ on $\boldsymbol{F}$ are adjoint with respect to this inner product.

Let $B_\lambda : = B_{\lambda_1}B_{\lambda_2} \cdots B_{\lambda_{\ell(\lambda)}}$ and define $B_{-\lambda}$ similarly. We define four operators within the Heisenberg algebra:
\begin{align*}
    D_k &:= \sum_{\lambda \vdash k} z_\lambda^{-1} B_\lambda
    &\widetilde{D}_k := (-1)^k \sum_{\lambda \vdash k} (-1)^{\ell(\lambda)} z_\lambda^{-1} B_\lambda\\
    U_k &:= \sum_{\lambda \vdash k} z_\lambda^{-1} B_{-\lambda}
    &\widetilde{U}_k := (-1)^k \sum_{\lambda \vdash k} (-1)^{\ell(\lambda)} z_\lambda^{-1} B_{-\lambda}.
\end{align*}    
Note that we have designed $D_k$ (respectively $\widetilde{D}_k$) to relate to $B_\lambda$ in the same way as $h_k$ (respectively $e_k$) does to $p_\lambda$.

\begin{defn}
Let $\lambda,\mu$ be partitions. Then, we define the generating functions
\begin{align*}
    F_{\lambda/\mu}(X/Y;q) &= \sum_{\a,\b} X^\a (-Y)^\b \langle U_{\a_l}U_{\a_{l-1}}\cdots U_{\a_1}\widetilde{U}_{\b_m}\widetilde{U}_{\b_{m-1}}\cdots \widetilde{U}_{\b_1} \cdot \mu, \lambda\rangle,\\
    G_{\lambda/\mu}(X/Y;q) &= \sum_{\a,\b} X^\a (-Y)^\b \langle D_{\a_l}D_{\a_{l-1}}\cdots D_{\a_1}\widetilde{D}_{\b_m}\widetilde{D}_{\b_{m-1}}\cdots \widetilde{D}_{\b_1} \cdot \lambda, \mu\rangle
\end{align*}
where the sums run over all compositions $\a,\b$.
\end{defn}

It is helpful to think of these as tableaux style definitions: the $U$ operators add $n$-ribbon strips to a tableaux of shape $\mu$: $U_k$ adding horizontal strips and $\widetilde{U}_k$ vertical strips, and similarly, the $D$ operators remove $n$-ribbon strips from a tableaux of shape $\lambda$: $D_k$ horizontal and $\widetilde{D}_k$ vertical. Note: in \cite{lam-ribbon-heisenberg}, the operators $D_k$ are denoted $\mathcal{V}_k$. From this viewpoint, we realize that both $F_{\lambda/\mu}$ and $G_{\lambda/\mu}$ recover the super LLT polynomials defined in the previous section. This reformulation allowed Lam to prove the symmetry properties of super LLT polynomials stated in Proposition \ref{prop:orderdoesntmatter}, which are much more difficult to see from the tableaux definition. On the other hand, using the operator definition, symmetry in $X$ (respectively $Y$) results from the fact that the $D_\a$ (respectively   $\widetilde{D}_\a$) operators commute with themselves. Likewise, the fact that the total order does not affect the polynomial comes from the fact that the $D_\a$ and $\widetilde{D}_\a$ operators commute.

We next consider the interaction between the actions of adding and removing ribbon strips of different types. Let $\kappa: \Lambda \rightarrow \C$ be the algebra homomorphism defined by $\kappa(p_k) = a_k$. 

\begin{lemma}\label{lemma:commute} We have the following identities as elements in $H[a_i]$:
\begin{align*} D_b U_a &= \sum_{j=0}^m \kappa(h_j) U_{a-j}D_{b-j}
&\widetilde{D}_b \widetilde{U}_a = \sum_{j=0}^m \kappa(h_j) \widetilde{U}_{a-j}\widetilde{D}_{b-j}\\
\widetilde{D}_b U_a &= \sum_{j=0}^m \kappa(e_j) U_{a-j}\widetilde{D}_{b-j}
&D_b \widetilde{U}_a = \sum_{j=0}^m \kappa(e_j) \widetilde{U}_{a-j}D_{b-j}
\end{align*}
where $m = \min\{a,b\}$. Equivalently, if we let $U(x) := 1 + \sum_{i\geq 0} U_ix^i$ and define $D(x), \widetilde{U}(x), \widetilde{D}(x)$ similarly, we have the following commutation relations
\begin{align*}
    [D(y), U(x)] = [\widetilde{D}(y),\widetilde{U}(x)] &= 1 + \sum_{i\geq 0} \kappa(h_i)(xy)^i\\
    [\widetilde{D}(y), U(x)] = [D(y),\widetilde{U}(x)] &= 1 + \sum_{i\geq 0} \kappa(e_i)(xy)^i\\
\end{align*}
\end{lemma}

\begin{proof} The relation of $D_b$ and $U_a$ is proven as Lemma 8 in \cite{LamBF}. We follow the general thread of that proof for the remaining relations: consider the second desired identity with $\widetilde{D}_b$ and $\widetilde{U}_a$. Expanding via definition, the left hand side becomes 
\begin{align*}\widetilde{D}_b \widetilde{U}_a  &= (-1)^a(-1)^b\sum_{\substack{\rho \vdash a\\\pi \vdash b}} (-1)^{\ell(\rho) + \ell(\pi)} z_\rho^{-1}z_\pi^{-1} B_\pi B_{-\rho}
\intertext{Then, apply the commutation relation (\ref{eq:commutator}) to swap the order of positively and negatively-signed $B$ operators. Effectively, this process builds a smaller partition $\lambda$ out of parts that are in both $\pi$ and $\rho$; we must account for all the ways to build a given $\lambda$ in this manner. If $\mu = \pi \setminus \lambda$ and $\nu = \rho \setminus \lambda$, the expression above becomes}
&= (-1)^a(-1)^b \sum_{\substack{\rho \vdash a\\\pi \vdash b}} (-1)^{\ell(\pi) + \ell(\rho)} \sum_{\substack{\lambda \subset \rho\\\lambda \subset \pi}} z_{\mu}^{-1}z_{\nu}^{-1} z_\lambda^{-1} \kappa(p_\lambda) B_{-\mu} B_{\nu}
\intertext{Reindexing over the possible sizes $j$ of $\lambda$, cancelling matching powers of -1, and using the expansion of $h_j$ in terms of $p_\lambda$, we have}
&= \sum_{j=0}^m (-1)^{a-j}(-1)^{b-j} \kappa(h_j) \sum_{\substack{\mu \vdash a-j\\\nu \vdash b-j}}  (-1)^{\ell(\mu) + \ell(\nu)}z_{\mu}^{-1} z_{\nu}^{-1} B_{-\mu}B_\nu
\end{align*}
which by definition is the desired right hand side. The remaining identities follow similarly, except that since we only have one tilde-type operator, the power of $(-1)^{j + \ell(\lambda)}$ doesn't cancel out. Instead, it folds in with the $\kappa(p_\lambda)$ according to the expansion of $e_j$ in terms of $p_\lambda$, so we have a coefficient of $\kappa(e_j)$ instead of $\kappa(h_j)$.
\end{proof}

Using this operator definition, we can construct a generalized Cauchy identity. Note that this theorem applies to any function that can be achieved through specialization of $a_i$ in the operators, not just to super LLT polynomials, however, we will focus on the implications for super LLT polynomials in this paper. See also Hardt \cite{Hardt} or Aggarwal, Borodin, Petrov, and Wheeler \cite{ABPW} for an interesting investigation of this identity in terms of Hamiltonian operators; our Theorem \ref{thm:cauchy} is similar to their Proposition 2.10 and Proposition 3.7, respectively.

\begin{theorem}\label{thm:cauchy} Let $\mu,\nu$ be partitions. Then
\[\sum_{\lambda} F_{\lambda/\mu}(X/Y) G_{\lambda/\nu} (W/Z) = [\ast]\cdot \sum_{\lambda} F_{\nu/\lambda}(X/Y) G_{\mu/\lambda}(W/Z)\]
where
\[ [\ast] = \prod_{i,j,k,\ell} [D(w_k), U(x_i)]\cdot  [\widetilde{D}(-z_\ell),\widetilde{U}(-y_j)]\cdot  [\widetilde{D}(-z_\ell),U(x_i)] \cdot [D(w_k), \widetilde{U}(-y_j)].\]
\end{theorem}

\begin{proof}
By definition, we have $F_{\lambda/\mu}(X/Y) = \langle \cdots U(x_2)U(x_1) \cdots \widetilde{U}(-y_2) \widetilde{U}(-y_1) \cdot \mu, \lambda\rangle$ and $G_{\lambda/\mu}(X/Y) = \langle \cdots D(x_2)D(x_1) \cdots \widetilde{D}(-y_2) \widetilde{D}(-y_1) \cdot \lambda, \mu\rangle$. Then using the properties of the inner product and the commutation relation of Lemma \ref{lemma:commute}, we have
\begin{align*}
    \sum_{\lambda} F_{\lambda/\mu}(X/Y) &G_{\lambda/\nu} (W/Z) = \sum_\lambda \langle \cdots U(x_2)U(x_1) \cdots \widetilde{U}(-y_2) \widetilde{U}(-y_1) \cdot \mu, \lambda\rangle \\ &\hspace{3cm} \cdot\langle \cdots D(w_2)D(w_1) \cdots \widetilde{D}(-z_2) \widetilde{D}(-z_1) \cdot \lambda, \nu\rangle\\
    &= \langle \cdots D(w_2)D(w_1) \cdots \widetilde{D}(-z_2) \widetilde{D}(-z_1) \cdots U(x_2)U(x_1) \cdots \widetilde{U}(-y_2) \widetilde{U}(-y_1) \cdot \mu, \nu\rangle\\
    &= [\ast] \langle \cdots U(x_2)U(x_1) \cdots \widetilde{U}(-y_2) \widetilde{U}(-y_1) \cdots D(w_2)D(w_1) \cdots \widetilde{D}(-z_2) \widetilde{D}(-z_1)  \cdot \mu, \nu\rangle\\
    &=[\ast] \cdot\sum_\lambda \langle \cdots D(w_2)D(w_1) \cdots \widetilde{D}(-z_2) \widetilde{D}(-z_1) \cdot \mu, \lambda\rangle \\ &\hspace{3cm} \cdot \langle \cdots U(x_2)U(x_1) \cdots \widetilde{U}(-y_2) \widetilde{U}(-y_1) \cdot \lambda, \nu\rangle\\
    &= [\ast] \cdot \sum_\lambda F_{\nu/\lambda}(X/Y) G_{\mu/\lambda}(W/Z).
\end{align*}
\end{proof}

Specializing $\mu$ and $\nu$ to an $n$-core partition gives a more ``standard" Cauchy identity formula with a sum side and a product side. 

\begin{cor}\label{cor:cauchy}
For $\delta$ an $n$-core, 
\[\sum_{\lambda: \widetilde{\lambda} = \delta} \mathcal{G}_{\lambda/\delta}(X/Y;q) \mathcal{G}_{\lambda/\delta}(W/Z;q) = \prod_{i,j,k,\ell} \prod_{t = 0}^{n-1} \frac{(1- q^{2t}x_iz_\ell)(1-q^{2t}y_jw_k)}{(1-q^{2t}x_iw_k)(1-q^{2t}y_jz_\ell)}.\]
\end{cor}

\begin{proof}
Set $\mu = \nu = \delta$ and consider Theorem \ref{thm:cauchy}, recalling that for the given Heisenberg parameters $a_k =1 + q^{2k} + \cdots q^{2k(n-1)}$, both $F_{\lambda/\mu}$ and $G_{\lambda/\mu}$ produce the super LLT polynomial $\mathcal{G}_{\lambda/\mu}(X/Y;q)$. There is only one nonzero term in the right-hand sum, because you can't take any $n$-ribbon strips away from an $n$-core, so the only term is 1 from $\a = \b = \emptyset$. It remains then to simplify the commutation factor: we have
\[  [D(y),U(x)] = \prod_{t = 0}^{n-1} \frac{1}{1-q^{2t}xy}
\quad\quad \text{and}\quad\quad
 [\widetilde{D}(y), U(x)] = \prod_{t = 0}^{n-1} (1+q^{2t}xy)\]
so taking the product over the appropriate combinations of signed $x,y,z,w$ variables produces the desired statement.
\end{proof}

\begin{rmk}
Specializing $Y= vX$ and $Z = vW$ in  recovers the Cauchy identity for metaplectic Whittaker functions $\mathcal{M}^n_{\lambda}(\boldsymbol{z})$ given in \cite{BBBG}, as $\mathcal{M}^n_{\lambda/\mu}(\boldsymbol{z}) = \mathcal{G}^n_{\lambda/\mu}(\boldsymbol{z}/ v\boldsymbol{z})$.
\end{rmk}

In search of a Dual Cauchy Identity to match this Cauchy Identity, we revisit the Cauchy and Dual Cauchy identities for LLT polynomials: since $\mathcal{G}_{\lambda/\mu}(X/0;q) = \mathcal{G}_{\lambda/\mu}(X;q)$, Corollary \ref{cor:cauchy} recovers the Cauchy identity for LLT polynomials:

\begin{cor}[Cauchy identity for LLT] For an $n$-core $\delta$,
\[ \sum_{\lambda: \widetilde{\lambda} = \delta} \mathcal{G}_{\lambda/\delta}(X/0;q)\cdot \mathcal{G}_{\lambda/\delta}(Y/0;q)= \prod_{i,j} \prod_{t=0}^{n-1} \frac{1}{1-x_iy_jq^{2t}}.\]
\end{cor}

Using the relationship between $\mathcal{G}(X/Y)$ and $\mathcal{G}(Y/X)$ developed in Proposition \ref{prop:flipover}, we see that $q^{(n-1)\frac{|\lambda-\mu|}{n}} \mathcal{G}_{\lambda'/\mu'}(X;q^{-1}) =  \mathcal{G}_{\lambda/\mu}(0/-X;q)$, so Corollary \ref{cor:cauchy} also recovers the Dual Cauchy identity for LLT polynomials.

\begin{cor}[Dual Cauchy identity for LLT] For an $n$-core $\delta$,
\begin{align*} \prod_{i,j} \prod_{t=0}^{n-1} (1+x_iy_jq^{2t}) &= \sum_\lambda q^{(n-1)\frac{|\lambda-\mu|}{n}} \mathcal{G}_{\lambda/\mu}(X/0;q)\cdot \mathcal{G}_{\lambda'/\mu'}(Y/0;q^{-1})\\
&= \sum_\lambda \mathcal{G}_{\lambda/\mu}(X/0;q)\cdot \mathcal{G}_{\lambda/\mu}(0/-Y;q).
\end{align*}
\end{cor}

We may think of this relationship between Dual Cauchy and Cauchy identities for the LLT polynomials as the fact that not all semi-standard ribbon tableaux are also semi-standard (vertical) ribbon tableaux, so comparing horizontal to vertical fillings will return only finitely many options for tableaux shapes, but comparing horizontal to horizontal or vertical to vertical have infinitely many partitions that may be admissibly filled. However, since super ribbon tableaux combine the notions of ``horizontally" and ``vertically" semi-standard, the analogue of a Dual Cauchy Identity for super LLT polynomials is in fact the consequence of combining Proposition \ref{prop:flipover} with Corollary \ref{cor:cauchy}.

\begin{cor}[``Dual" Cauchy Identity for super LLT]
For $\delta$ an $n$-core, 
\begin{align*}\sum_{\lambda: \widetilde{\lambda}= \delta} q^{(n-1)\frac{|\lambda-\mu|}{n}} \mathcal{G}_{\lambda/\delta}(X/Y;q) \mathcal{G}_{\lambda'/\delta'}(W/Z;q^{-1}) &= \prod_{i,j,k,\ell} \prod_{t = 0}^{n-1} \frac{(1+q^{2t}x_iw_k)(1+q^{2t}y_jz_\ell)}{(1+ q^{2t}x_iz_\ell)(1+q^{2t}y_jw_k)}.
\end{align*}
\end{cor}

Examining the way we apply these operators to add or remove horizontal and vertical strips, from the view of particle interactions inside the individual ribbon, we can construct a lattice model whose partition function gives the super LLT polynomials. In Section \ref{sec:cauchyish}, we will then investigate how we can see relationships between the Cauchy identity above and a solvable \emph{Cauchy lattice model} built out of this lattice model.

\section{A Lattice Model for Super LLT Polynomials}\label{sec:latticemodel}

Many interesting classes of polynomials and special functions can be represented as the partition functions of \emph{solvable} lattice models, i.e. weighted functions on a grid system that satisfy Yang-Baxter equations. In this section, we show that super LLT polynomials appear as the partition functions of a \emph{n-ribbon} lattice model, an $n$-stranded generalization of the 5 vertex lattice model for Schur polynomials. 

Throughout this section, fix an integer $n \geq 1$.  To define a \emph{$n$-ribbon lattice model}, construct a two-dimensional square grid and place vertices at the intersections. Every vertical edge will be assigned a \emph{label} from the set $\{\wedge,\vee\}$; this label is sometimes also called a ``spin," but we will refrain from using that term to avoid confusion with the spin ($q$-power) of a ribbon. Every horizontal edge will be assigned an $n$-tuple of labels from the set $\{<,>\}$. However, distinguishing this model from merely being a collapsed version of a grid $n$ times taller, we will require a curious condition on these tuple labels: given a vertex $v$, denote the north, south, east and west entries of a vertex $v$ by $v_N, v_S,v_W$ and $v_S$ respectively.

\begin{defn}
	We say a $n$-ribbon vertex $v$ is \emph{admissible} if it satisfies the following conditions:
	\begin{enumerate}[nosep]
		\item The number of arrows pointing inwards equals the number of arrows pointing outwards.
		\item The labels $v_E(i) = v_W(i+1)$ for all $i \in \{1,2,\ldots, n-1\}$. Note that $v_E(n)$ need not be equal to $v_W(1)$.
	\end{enumerate}
\end{defn}

We may visualize this condition by expanding the tuple-labelled horizontal edge crossing a vertex into $n$ horizontal edges with one distinguished edge to represent that from $v_W(1)$ to $v_E(n)$ (see Figure \ref{fig:nvertex}). We will call this distinguished edge the \emph{twisted edge}, and every other horizontal edge will be called \emph{straight}. Then we can summarize the second condition concisely by requiring that arrows do not change along any straight (horizontal) edge.

\begin{figure}[h]
\centering
	\begin{tikzpicture}
		\draw [thick](0,0.7)--(0,-0.7);
		\draw [thick](0.7,0)--(-0.7,0);
		\node () at (0,1) {$v_N$};
		\node () at (0,-1) {$v_S$};
		\node () at (-1,0) {$v_W$};
		\node () at (1,0) {$v_E$};
		
	\end{tikzpicture}
	\quad\quad\quad
		\begin{tikzpicture}
			\coordinate (h) at (0,1);
		\coordinate (t) at (0,-1);
		\coordinate (r1) at (0.67,0.67);
		\coordinate (r2) at (0.67, 0.33);
		\coordinate (r3) at (0.67, 0);
		\coordinate (r4) at (0.67,-0.33);
		\coordinate (r5) at (0.67,-0.67);
		\coordinate (l1) at (-0.67,0.67);
		\coordinate (l2) at (-0.67, 0.33);
		\coordinate (l3) at (-0.67, 0);
		\coordinate (l4) at (-0.67,-0.33);
		\coordinate (l5) at (-0.67,-0.67);
	    \draw [thick](h)--(t);
	    \draw (l5) cos (0,0) sin (r1); 
	    \draw (l1)--(r2);
	    \draw [dotted] (l2)--(r3);
	    \draw (l3)--(r4);
	    \draw (l4)--(r5);

	    \node (vd) at (0,-1.2) {$v_S$};
	    \node (vu) at (0,1.2) {$v_N$};
	    
	    \node (vr1) at (1.16,0.67) {\tiny$v_E(n)$};
	    \node (vr2) at (1.4,0.33) {\tiny$v_E(n-1)$};
	    \node (vr4) at (1.16,-0.33) {\tiny$v_E(2)$};
	    \node (vr5) at (1.16,-0.67) {\tiny$v_E(1)$};
	    
	    \node (vl1) at (-1.16,0.67) {\tiny$v_W(n)$};
	    \node (vl2) at (-1.16,0) {\tiny$v_W(3)$};
	    \node (vl4) at (-1.16,-0.33) {\tiny$v_W(2)$};
	    \node (vl5) at (-1.16,-0.67) {\tiny$v_W(1)$};
	\end{tikzpicture}
	\caption{Two renderings of a $n$-ribbon vertex, where the one on the left uses tuple labels $v_W=(v_W(1),\cdots,v_W(n))$ and $v_E=(v_E(1)\cdots v_E(n))$ to collapse the strands. In both renderings, $v_N,v_S \in \{\wedge,\vee\}$ and $v_E(i),v_W(i)\in \{<,>\}$. We will tend to use the right one, as it is more diagramatically convenient for our applications.}
	\label{fig:nvertex}
\end{figure}
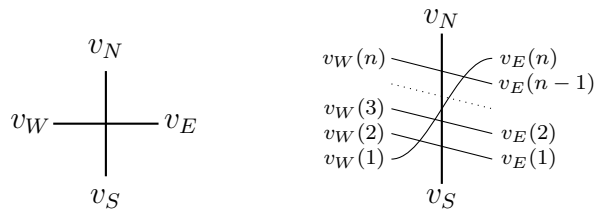

We call a set of label conditions on the boundary edges of the grid a \emph{system}, allowing the labels on interior edges to vary. A given assignment of labels for all these interior edges is called a \emph{state} of that system. We say a state is \emph{admissible} each of its vertices is admissible. Under the stranded rendering, we see that an admissible state has chains of arrows travelling $n$ vertices without changing label. That is, that every left (respectively right) arrow occurring at $v_E(n)$ for some vertex $v$ must be followed by $n-1$ consecutive left (respectively right) arrows as we travel to the right along its strand before it is allowed to change label (see Figure \ref{fig:left_arrows_following}). It is helpful to think of this condition in terms of a path model for particles: if we consider particles travelling through our lattice model along ``up" and ``left" arrows (where then edges labelled ``down" or ``right" have no particle), this condition tells us that particles \emph{must} travel in steps of $n$ vertices at a time.
 
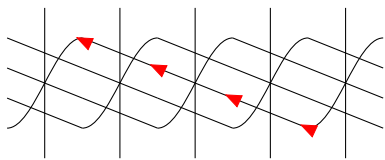
\begin{figure}[h]

\centering
\begin{tikzpicture}
    \foreach \x in {0,...,4}{
\ribbonvertex{\x}{0}{\x}
}
\arrowleftcolor{1-0.5+0.02}{0+3/5-0.04}{red}
\arrowleftcolor{2-0.5}{0+1/5}{red}
\arrowleftcolor{3-0.5}{0-1/5}{red}
\arrowleftcolor{4-0.5}{0-3/5}{red}
\end{tikzpicture}
\caption{The $v_E(4)$ entry of the leftmost vertex is a left arrow, which must be following by 3 left arrows from the right.}
\label{fig:left_arrows_following}
\end{figure}

\begin{defn} Given a skew partition $\lambda/\mu$ and an positive integer $r$, define the system of boundary conditions $\mathcal{B}_{\lambda/\mu}$ on a $n$-ribbon lattice model, letting $\rho = (r,r-1,...,2,1)$, as: 
\begin{itemize}
    \item there are $r$ rows and $\lambda_1 + r$ columns,
    \item all edges on the left and right boundaries are labelled $>$,
    \item numbering the columns left to right starting with 1, edges on the bottom boundary that appear as parts of $\lambda + \rho$ are labelled $\wedge$, while all other edges on the top boundary are labelled $\vee$, and
    \item similarly, edges on the top boundary that appear as parts of $\mu + \rho$ are labelled $\wedge$, while all other edges on the bottom boundary are labelled $\vee$.
\end{itemize}
\end{defn}

We then define a generating function on the lattice model.

\begin{defn}\label{def:ribbonweights} Let $\wt(v)$ denote the Boltzmann weight of a vertex. Define the Boltzmann weight of a state $\mathfrak{s}$ to be the product over the weights of its vertices and the Boltzmann weight of a system, more commonly known as the \emph{partition function} of that system, to be the sum over the weights of its states. That is,
\[ \wt(\mathfrak{s}) = \prod_{v\in \mathfrak{s}} \wt(v) \quad\quad \text{ and } \quad\quad Z(B) := \wt(\mathcal{B}) = \sum_{\mathfrak{s} \in \mathcal{B}} \wt(\mathfrak{s}).\]
Any vertex that is not admissible has weight 0. Given any admissible vertex $v$, its weight will depend on whether it is in a \emph{horizontal} row or a \emph{vertical} row. We will suggestively label the horizontal strip rows with numbers $i \in A$ and the vertical strip rows with numbers $i' \in A'$, and assign spectral parameters $x_i$ to horizontal strip rows and $y_{i'}$ to vertical strip rows. The weights for vertices in a horizontal strip row $i$ are given in Figure \ref{fig:RibbonWts} and those for vertices in a vertical strip row $i'$ are given in Figure \ref{fig:columnRibbonWts}.
\end{defn}

We now define two sets of Boltzmann weights for the admissible vertices, all of which are monomials in $q$ and the a spectral parameter assigned to the row in which a vertex appears. We suggestively label the first set as \emph{horizontal strip row} weights, shown in Figure \ref{fig:RibbonWts}, and the latter as \emph{vertical strip row} weights, shown in Figure \ref{fig:columnRibbonWts}. Note that in these tables, we label the different types of vertices based on the edges pointing into the vertex from either the column edge or the twisted edge. We may generally think of the twisted edge as controlling the power of the spectral parameter and the straight edges as controlling the power of $q$.

\begin{figure}[H]
\centering 
 \begin{tabular}{| M{1.1cm} |M{1.47cm} | M{1.47cm} | M{1.47cm} | M{1.47cm} | M{1.47cm} | M{1.47cm} |} 
 \hline
Label & SW & NS & SE & NW & EW & NE \\
 \hline 
Vertex&\qquad
\begin{tikzpicture}[scale=0.8]
	\goodlookingvertex{-4.5}{1.5}{3}
	\arrowup{-4.5}{-2.5+3}
		\arrowright{-4.5+0.67}{1.5+0.64}
		\arrowup{-4.5}{-0.5+3}
		\arrowright{-4.5-0.67}{1.5-0.64}
		
		\draw [blue](-4.5+0.67,1.5-0.16) ellipse (0.22cm and 0.64cm);
\end{tikzpicture}     &

\qquad \begin{tikzpicture}[scale=0.8]
\goodlookingvertex{-1.5}{1.5}{1}
	\arrowdown{-1.5}{2.5}
		\arrowup{-1.5}{0.5}
		\arrowright{-1.5+0.67}{1.5+0.64}
		\arrowleft{-1.5-0.67}{1.5-0.64}
		\draw [blue](-1.5+0.67,1.5-0.16) ellipse (0.22cm and 0.64cm);
\end{tikzpicture}  & 

\qquad \begin{tikzpicture}[scale=0.8]
\goodlookingvertex{-1.5}{1.5}{1}
	\arrowup{-1.5}{2.5}
		\arrowup{-1.5}{0.5}
		\arrowleft{-1.5+0.67}{1.5+0.64}
		\arrowleft{-1.5-0.67}{1.5-0.64}
\end{tikzpicture} & 
\qquad \begin{tikzpicture}[scale=0.8]
\goodlookingvertex{-1.5}{1.5}{1}
	\arrowdown{-1.5}{2.5}
		\arrowdown{-1.5}{0.5}
		\arrowright{-1.5+0.67}{1.5+0.64}
		\arrowright{-1.5-0.67}{1.5-0.64}
\end{tikzpicture} & 
\qquad \begin{tikzpicture}[scale=0.8]
\goodlookingvertex{1.5}{1.5}{2}
	\arrowup{1.5}{2.5}
		\arrowleft{1.5+0.67}{1.5+0.64}
		\arrowdown{1.5}{0.5}
		\arrowright{1.5-0.67}{1.5-0.64}
		\draw [blue](1.5+0.67,1.5-0.16) ellipse (0.22cm and 0.64cm);
\end{tikzpicture} & 
\qquad \begin{tikzpicture}[scale=0.8]
\goodlookingvertex{4.5}{1.5}{2}
\arrowdown{4.5}{1.5+1}
		\arrowleft{4.5+0.67}{1.5+0.64}
		\arrowdown{4.5}{1.5-1}
		\arrowleft{4.5-0.67}{1.5-0.64}
	\draw [blue](4.5+0.67,1.5-0.16) ellipse (0.22cm and 0.64cm);
\end{tikzpicture}
\\

\hline
Weight & $q^s$ & $q^s$ & 0 & 1 & $q^sx_i$ & $q^sx_i$  \\
\hline

 \end{tabular}
 \caption{The weights of vertices lying in \emph{horizontal strip row} $i$ for the $n$ super ribbon lattice model, where $s$ is the number of left arrows in the blue circled area.}
 \label{fig:RibbonWts}
\end{figure}

\begin{figure}[h]
\centering 
 \begin{tabular}{| M{1.1cm} |M{1.47cm} | M{1.47cm} | M{1.47cm} | M{1.47cm} | M{1.47cm} | M{1.47cm} |} 
 \hline
Label & SW & NS & SE & NW & EW & NE \\
 \hline 
Vertex&\qquad
\begin{tikzpicture}[scale=0.8]
	\goodlookingvertex{-4.5}{1.5}{3}
	\arrowup{-4.5}{-2.5+3}
		\arrowright{-4.5+0.67}{1.5+0.64}
		\arrowup{-4.5}{-0.5+3}
		\arrowright{-4.5-0.67}{1.5-0.64}
		
		\draw [blue](-4.5+0.67,1.5-0.16) ellipse (0.22cm and 0.64cm);
\end{tikzpicture}     &

\qquad \begin{tikzpicture}[scale=0.8]
\goodlookingvertex{-1.5}{1.5}{1}
	\arrowdown{-1.5}{2.5}
		\arrowup{-1.5}{0.5}
		\arrowright{-1.5+0.67}{1.5+0.64}
		\arrowleft{-1.5-0.67}{1.5-0.64}
\end{tikzpicture}  & 

\qquad \begin{tikzpicture}[scale=0.8]
\goodlookingvertex{-1.5}{1.5}{1}
	\arrowup{-1.5}{2.5}
		\arrowup{-1.5}{0.5}
		\arrowleft{-1.5+0.67}{1.5+0.64}
		\arrowleft{-1.5-0.67}{1.5-0.64}
\end{tikzpicture} & 
\qquad \begin{tikzpicture}[scale=0.8]
\goodlookingvertex{-1.5}{1.5}{1}
	\arrowdown{-1.5}{2.5}
		\arrowdown{-1.5}{0.5}
		\arrowright{-1.5+0.67}{1.5+0.64}
		\arrowright{-1.5-0.67}{1.5-0.64}
\end{tikzpicture} & 
\qquad \begin{tikzpicture}[scale=0.8]
\goodlookingvertex{1.5}{1.5}{2}
	\arrowup{1.5}{2.5}
		\arrowleft{1.5+0.67}{1.5+0.64}
		\arrowdown{1.5}{0.5}
		\arrowright{1.5-0.67}{1.5-0.64}
\end{tikzpicture} & 
\qquad \begin{tikzpicture}[scale=0.8]
\goodlookingvertex{4.5}{1.5}{2}
\arrowdown{4.5}{1.5+1}
		\arrowleft{4.5+0.67}{1.5+0.64}
		\arrowdown{4.5}{1.5-1}
		\arrowleft{4.5-0.67}{1.5-0.64}
\end{tikzpicture}
\\
\hline
Weight & $q^s$ & $1$ & $-y_{i'}$ & 1 & $-y_{i'}$ & $0$  \\
\hline

 \end{tabular}
 \caption{The weights of vertices lying in \emph{vertical strip row} $i' \in A'$ for the $n$ super ribbon lattice model, where $s$ is the number of left arrows in the blue circled area.}
 \label{fig:columnRibbonWts}
\end{figure} 

In situations where we intend to consider the partition function of a system, we may further specify our boundary conditions to include a choice of spectral parameters. 

\begin{defn}
Given two ordered alphabets $A, A'$, set a total ordering between them that respects their individual orderings. Then, let $\mathcal{B}_{\lambda/\mu}(X/Y)$ be the family of lattice models with boundary conditions $\mathcal{B}_{\lambda/\mu}$ under the additional requirements:
\begin{itemize}
    \item \emph{horizontal} rows labelled $i\in A$ have spectral parameters $x_i$,
    \item \emph{vertical} rows labelled $i' \in A'$ have spectral parameters $y_{i'}$, and
    \item reading from top to bottom, rows are labelled in increasing order according to the total order.
\end{itemize}
\end{defn}

\begin{theorem}\label{mainthm}
Given a skew partition $\lambda/\mu$ and a total order on alphabets $A,A'$, we have
 \[Z\left(\mathcal{B}_{\l/\mu}(X/Y)\right)=\G_{\l/\mu}^{(n)}(X/Y;q).\] 
\end{theorem}

To prove Theorem \ref{mainthm}, we construct a weight preserving bijection between the set of all super ribbon tableaux and the set of all admissible ribbon lattices with the corresponding boundary conditions. We start with an important lemma about individual ribbons.

\begin{lemma}\label{lem:ribbonbij}
There is a weight-preserving bijection between $n$-ribbons and ribbon lattice models with one row and $n+1$ columns, using either the weights from Figure \ref{fig:RibbonWts} or from Figure \ref{fig:columnRibbonWts}.
\end{lemma}

\begin{proof}
The bijection itself is easy to describe: we send the $n$-ribbon of shape $\lambda/\mu$ to the one-row lattice model with boundary conditions $\mathcal{B}_{\lambda/\mu}$. To show that it is weight-preserving is a bit more complicated, and we take this time to develop some machinery that will be eminently useful later.

Consider an $n$-ribbon of shape $\lambda/\mu$. Starting in the bottom right corner, label each vertical (resp. horizontal) edge red (resp. blue) and number their positions increasingly along each of the upper and lower edges of the ribbon. We call this the \emph{edge sequence path} of the ribbon. Note that red edges occur on labels that occur as parts of $\mu + \rho$ on the top boundary and on labels in $\lambda + \rho$ on the bottom boundary of the ribbon, and the largest label will always be $n+1$. We may therefore think of the bijection heuristically as ``stretching" the ribbon straight to lay atop the lattice model.

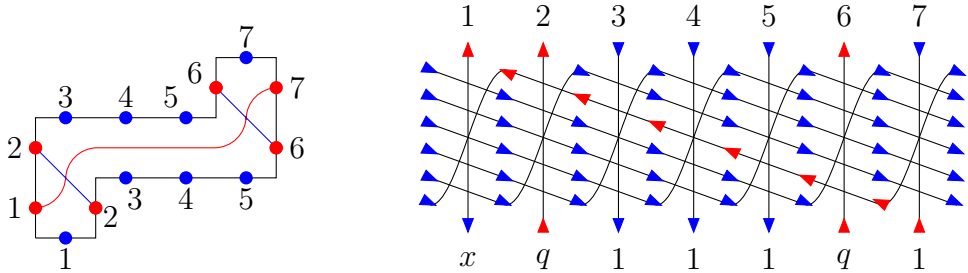
\begin{figure}[h]
	\centering
	\begin{tikzpicture}[scale=0.8]
		\draw (0,0)--(0,1)--(0,2)--(3,2)--(3,3)--(4,3)--(4,1)--(1,1)--(1,0)--(0,0);
		\coordinate (01) at (0.5,0);
		\coordinate (02) at (1,0.5);
		\coordinate (03) at (1.5,1);
		\coordinate (04) at (2.5,1);
		\coordinate (05) at (3.5,1);
		\coordinate (06) at (4,1.5);
		\coordinate (07) at (4,2.5);
		
		\coordinate (11) at (0,0.5);
		\coordinate (12) at (0,1.5);
		\coordinate (13) at (0.5,2);
		\coordinate (14) at (1.5,2);
		\coordinate (15) at (2.5,2);
		\coordinate (16) at (3,2.5);
		\coordinate (17) at (3.5,3);
		
		\node [circle, fill=blue,inner sep=0pt,minimum size=5pt] at (01) {};
		\node [circle, fill=red,inner sep=0pt,minimum size=5pt] at (02) {};
		\node [circle, fill=blue,inner sep=0pt,minimum size=5pt] at (03) {};
		\node [circle, fill=blue,inner sep=0pt,minimum size=5pt] at (04) {};
		\node [circle, fill=blue,inner sep=0pt,minimum size=5pt] at (05) {};
		\node [circle, fill=red,inner sep=0pt,minimum size=5pt] at (06) {};
        \node [circle, fill=red,inner sep=0pt,minimum size=5pt] at (07) {};
        \node [circle, fill=red,inner sep=0pt,minimum size=5pt] at (11) {};
        \node [circle, fill=red,inner sep=0pt,minimum size=5pt] at (12) {};
        \node [circle, fill=blue,inner sep=0pt,minimum size=5pt] at (13) {};
        \node [circle, fill=blue,inner sep=0pt,minimum size=5pt] at (14) {};
        \node [circle, fill=blue,inner sep=0pt,minimum size=5pt] at (15) {};
        \node [circle, fill=red,inner sep=0pt,minimum size=5pt] at (16) {};
        \node [circle, fill=blue,inner sep=0pt,minimum size=5pt] at (17) {};
        
        \node at (0.5,-0.35) {1};
        \node at (1.25,0.35){2};
        \node at (1.65,0.65) {3};
        \node at (2.5,0.65) {4};
        \node at (3.5,0.65) {5};
        \node at (4.35,1.5) {6};
        \node at (4.35,2.5) {7};
        
        \node at (-0.35,0.5) {1};
        \node at (-0.35,1.5){2};
        \node at (0.5,2.35) {3};
        \node at (1.5,2.35) {4};
        \node at (2.25,2.35) {5};
        \node at (2.65,2.75) {6};
        \node at (3.5,3.35) {7};

		\draw [rounded corners=12,color=red] (4,2.5)--(3.5,2.5) -- (3.5,1.5)--(0.5,1.5)--(0.5,0.5)--(0,0.5);
		\draw [color=blue] (12)--(02);
		\draw [color=blue] (16)--(06);	
		\node [circle, fill=red,inner sep=0pt,minimum size=5pt] at (02) {};
		 \node [circle, fill=red,inner sep=0pt,minimum size=5pt] at (12) {};
		 \node [circle, fill=red,inner sep=0pt,minimum size=5pt] at (16) {};
		 \node [circle, fill=red,inner sep=0pt,minimum size=5pt] at (06) {};
	\end{tikzpicture}
	\quad\quad\quad
	\begin{tikzpicture}[xscale=1, yscale = 0.9]
    \foreach \x in {1,...,7}{
\ribbonvertexsix{\x}{0}{\x}
}

\arrowdowncolor{1}{-1.3}{blue}
\arrowupcolor{2}{-1.3}{red}
\arrowdowncolor{3}{-1.3}{blue}
\arrowdowncolor{4}{-1.3}{blue}
\arrowdowncolor{5}{-1.3}{blue}
\arrowupcolor{6}{-1.3}{red}
\arrowupcolor{7}{-1.3}{red}

\arrowupcolor{1}{1.3}{red}
\arrowupcolor{2}{1.3}{red}
\arrowdowncolor{3}{1.3}{blue}
\arrowdowncolor{4}{1.3}{blue}
\arrowdowncolor{5}{1.3}{blue}
\arrowupcolor{6}{1.3}{red}
\arrowdowncolor{7}{1.3}{blue}

\arrowrightcolor{1-0.5}{0+1}{blue}
\arrowrightcolor{1-0.5}{0+3/5}{blue}
\arrowrightcolor{1-0.5}{0+1/5}{blue}
\arrowrightcolor{1-0.5}{0-1/5}{blue}
\arrowrightcolor{1-0.5}{-3/5}{blue}
\arrowrightcolor{1-0.5-0.02}{-1+0.04}{blue}

\arrowleftcolor{2-0.5+0.02}{0+1-0.04}{red}
\arrowrightcolor{2-0.5}{0+3/5}{blue}
\arrowrightcolor{2-0.5}{0+1/5}{blue}
\arrowrightcolor{2-0.5}{0-1/5}{blue}
\arrowrightcolor{2-0.5}{0-3/5}{blue}
\arrowrightcolor{2-0.5-0.02}{0-1+0.04}{blue}

\arrowrightcolor{3-0.5+0.02}{0+1-0.04}{blue}
\arrowleftcolor{3-0.5}{0+3/5}{red}
\arrowrightcolor{3-0.5}{0+1/5}{blue}
\arrowrightcolor{3-0.5}{0-1/5}{blue}
\arrowrightcolor{3-0.5}{0-3/5}{blue}
\arrowrightcolor{3-0.5-0.02}{0-1+0.04}{blue}

\arrowrightcolor{4-0.5+0.02}{0+1-0.04}{blue}
\arrowrightcolor{4-0.5}{0+3/5}{blue}
\arrowleftcolor{4-0.5}{0+1/5}{red}
\arrowrightcolor{4-0.5}{0-1/5}{blue}
\arrowrightcolor{4-0.5}{0-3/5}{blue}
\arrowrightcolor{4-0.5-0.02}{0-1+0.04}{blue}

\arrowrightcolor{5-0.5+0.02}{0+1-0.04}{blue}
\arrowrightcolor{5-0.5}{0+3/5}{blue}
\arrowrightcolor{5-0.5}{0+1/5}{blue}
\arrowleftcolor{5-0.5}{0-1/5}{red}
\arrowrightcolor{5-0.5}{0-3/5}{blue}
\arrowrightcolor{5-0.5-0.02}{0-1+0.04}{blue}

\arrowrightcolor{6-0.5+0.02}{0+1-0.04}{blue}
\arrowrightcolor{6-0.5}{0+3/5}{blue}
\arrowrightcolor{6-0.5}{0+1/5}{blue}
\arrowrightcolor{6-0.5}{0-1/5}{blue}
\arrowleftcolor{6-0.5}{0-3/5}{red}
\arrowrightcolor{6-0.5-0.02}{0-1+0.04}{blue}

\arrowrightcolor{7-0.5+0.02}{0+1-0.04}{blue}
\arrowrightcolor{7-0.5}{0+3/5}{blue}
\arrowrightcolor{7-0.5}{0+1/5}{blue}
\arrowrightcolor{7-0.5}{0-1/5}{blue}
\arrowrightcolor{7-0.5}{0-3/5}{blue}
\arrowleftcolor{7-0.5-0.02}{0-1+0.04}{red}

\arrowrightcolor{8-0.5+0.02}{0+1-0.04}{blue}
\arrowrightcolor{8-0.5}{0+3/5}{blue}
\arrowrightcolor{8-0.5}{0+1/5}{blue}
\arrowrightcolor{8-0.5}{0-1/5}{blue}
\arrowrightcolor{8-0.5}{0-3/5}{blue}
\arrowrightcolor{8-0.5-0.02}{0-1+0.04}{blue}

\node () at (1,-1.8) {$x$};
\node () at (2,-1.8) {$q$};
\node () at (3,-1.8) {$1$};
\node () at (4,-1.8) {$1$};
\node () at (5,-1.8) {$1$};
\node () at (6,-1.8) {$q$};
\node () at (7,-1.8) {$1$};

\node () at (1,1.8) {$1$};
\node () at (2,1.8) {$2$};
\node () at (3,1.8) {$3$};
\node () at (4,1.8) {$4$};
\node () at (5,1.8) {$5$};
\node () at (6,1.8) {$6$};
\node () at (7,1.8) {$7$};

\end{tikzpicture}
	\caption{The 6-ribbon for $\lambda = (4,4,1), \mu = (3,0,0)$ and its corresponding lattice model state with weights included.}
	\label{fig:edgeseq}
\end{figure}

Claim: there is only one state for this model. Since a ribbon may not contain any 2 by 2 squares, we may assign a particle action to the red edges of the ribbon: along each $n$-ribbon, the red dot in the upper right corner (labelled $n+1$) will travel to the lower left corner (labelled $1$), and all the remaining reds will travel up and to the left while retaining the same label. Thus, aside from column 1 and column $n+1$, all other columns have the same label on top and on bottom. Coupling this with the fact that the side boundaries are all right arrows, the inner inner labels of the lattice model are fully determined: the vertex in column 1 must be type EW, that in column $n+1$ must be type NS, and the remainder are either type NW (if column $i$ is blue in the ribbon) or type SW (if column $i$ is red in the ribbon). The remaining labels are determined by the fact that straight edges don't change label across a vertex, so the strand from column 1 to column $n+1$ is entirely left arrows and all the rest are right arrows.

Vertex types in hand, we return to the weighting. There are no SE or NE vertices in this model, and the only internal left arrows occur on type SW vertices, so the sets of weights in Figures \ref{fig:RibbonWts} and \ref{fig:columnRibbonWts} give the same partition function for this model up to choice of the spectral parameter $z$, which we specialize to $x_i$ or $-y_{i'}$ respectively. Since we have a single state, the partition function is a monomial: the EW vertex gives a power of $z$ and each SW vertex a power of $q$, while the NW and NS vertices don't change the weight, so our lattice model weight is $q^{\# SW}z$.

Recall that a lone ribbon has weight $q^{\text{height} - 1}z$, where $z$ is the desired spectral parameter. Referring to Figure \ref{fig:edgeseq}, note that the spin $q^{\text{height} - 1}$ counts the number of intersections between particles (i.e. the number of red columns between 1 and $n+1$), which are precisely the columns that become SW vertices.
\end{proof}

\begin{prop}\label{prop:hribbon}
Using the weights from Figure \ref{fig:RibbonWts}, there is a weight-preserving bijection between horizontal $n$-ribbon strips and one-row lattice models with non-zero weight.
\end{prop}

\begin{proof}
Sends a horizontal ribbon strip of shape $\lambda/\mu$ to the one row model with boundary conditions $B_{\lambda/\mu}$. To show this map is a bijection, we define an assignment of internal arrows that turns a ribbon filling of this strip into an admissible state by ``peeling off" successive ribbons. Starting at the rightmost ribbon in the strip, assign left arrows as in Lemma \ref{lem:ribbonbij} for each ribbon, momentarily ignoring the right arrows. That is, a ribbon with edge sequence labels $k,...,k+n$ for some $k$ will assign left arrows to the lattice model edges $v_{k,E}(n), v_{k+1,E}(n-1),....,v_{k+n,W}(1)$, which comprise the straight edge strand connecting vertices $k$ and $k+n$. Once all ribbons have been considered, fill the remaining edges with right arrows. Since no two ribbons start (or end) with the same labels, these choices of left arrow edges will be distinct over all ribbons. Since the path of left arrows assigned by the last ribbon starts with $v_E(n)$ on $v_1$ and that of the first ribbon ends with $v_E(1)$ on the last vertex, this assignment of left arrows will not impede the assignment of straight edge right arrow paths dictated by the left and right boundary conditions. 

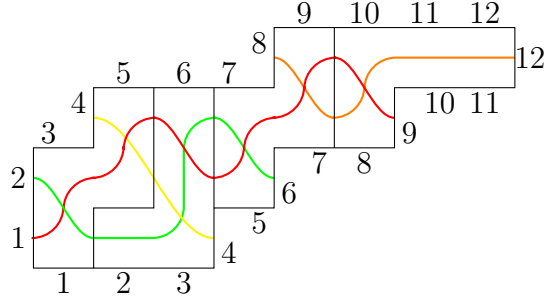
\begin{figure*}[h]
	\centering
	\begin{tikzpicture}[scale=0.8]
		\draw [rounded corners=12,color=orange,thick] (7,3.5)--(4.5,3.5)--(4.5,2.5)--(4,2.5);
		\draw [color=orange,thick] (4,2.5) cos (3.5,3) sin (3,3.5);
		\draw [color=red,thick] (5,2.5) cos (4.5,3) sin (4,3.5);
		\draw [color=green,thick] (3,1.5) cos (2.5,2) sin (2,2.5);
		\draw [rounded corners=12,color=red,thick] (4,3.5)--(3.5,3.5)--(3.5,2.5)--(2.5,2.5)--(2.5,1.5)--(2,1.5);
		\draw [rounded corners=12,color=green,thick] (2,2.5)--(1.5,2.5)--(1.5,0.5)--(0,0.5);
		\draw [color=yellow,thick] (2,0.5) cos (1,1.5);
		\draw [color=red,thick] (2,1.5) cos (1.5,2) sin (1,2.5);
		\draw [color=yellow,thick] (0,2.5) cos (1,1.5);
		\draw [color=red,rounded corners=12,thick] (1,2.5)--(0.5,2.5)--(0.5,1.5)--(-0.5,1.5)--(-0.5,0.5)--(-1,0.5);
		\draw [color=green,thick] (0,0.5) cos (-0.5,1) sin (-1,1.5);
		\draw (0,0)--(2,0)--(2,1)--(3,1)--(3,2)--(5,2)--(5,3)--(7,3)--(7,4)--(3,4)--(3,3)--(1,3)--(1,1)--(0,1)--cycle;
		\draw (2,1)--(2,3);
		\draw (4,2)--(4,4);
		\draw (1,3)--(0,3)--(0,2)--(-1,2)--(-1,1)--(-1,0)--(0,0);
		\node () at (-1.25,0.5) {1};
		\node () at (-1.25,1.5) {2};
		\node () at (-0.75,2.25) {3};
		\node () at (-0.25,2.75) {4};
		\node () at (0.5,3.25) {5};
		\node () at (1.5,3.25) {6};
		\node () at (2.25,3.25) {7};
		\node () at (2.75,3.75) {8};
		\node () at (3.5,4.25) {9};
		\node () at (4.5,4.25) {10};
		\node () at (5.5,4.25) {11};
		\node () at (6.5,4.25) {12};
		
		\node () at (-0.5,-0.25) {1};
		\node () at (0.5,-0.25) {2};
		\node () at (1.5,-0.25) {3};
		\node () at (2.25,0.25) {4};
		\node () at (2.75,0.75) {5};
		\node () at (3.25,1.25) {6};
		\node () at (3.75,1.75) {7};
		\node () at (4.5,1.75) {8};
		\node () at (5.25,2.25) {9};
		\node () at (5.75,2.75) {10};
		\node () at (6.5,2.75) {11};
		\node () at (7.25,3.5) {12};
	\end{tikzpicture}
	\caption{The edge sequence path of a horizontal strip.}
	\label{exmesp}
\end{figure*}

\begin{figure*}[h]

\centering
\begin{tikzpicture}[scale=0.85]
    \foreach \x in {0,1,...,11}{
\ribbonvertex{\x}{0}{\x}
}

\arrowdowncolor{0}{-1}{blue}
\arrowdowncolor{1}{-1}{blue}
\arrowdowncolor{2}{-1}{blue}
\arrowupcolor{3}{-1}{yellow}
\arrowdowncolor{4}{-1}{blue}
\arrowupcolor{5}{-1}{green}
\arrowdowncolor{6}{-1}{blue}
\arrowdowncolor{7}{-1}{blue}
\arrowupcolor{8}{-1}{red}
\arrowdowncolor{9}{-1}{blue}
\arrowdowncolor{10}{-1}{blue}
\arrowupcolor{11}{-1}{orange}

\arrowupcolor{0}{1}{red}
\arrowupcolor{1}{1}{green}
\arrowdowncolor{2}{1}{blue}
\arrowupcolor{3}{1}{yellow}
\arrowdowncolor{4}{1}{blue}
\arrowdowncolor{5}{1}{blue}
\arrowdowncolor{6}{1}{blue}
\arrowupcolor{7}{1}{orange}
\arrowdowncolor{8}{1}{blue}
\arrowdowncolor{9}{1}{blue}
\arrowdowncolor{10}{1}{blue}
\arrowdowncolor{11}{1}{blue}

\arrowrightcolor{0-0.5}{0+3/5}{blue}
\arrowrightcolor{0-0.5}{0+1/5}{blue}
\arrowrightcolor{0-0.5}{0-1/5}{blue}
\arrowrightcolor{0-0.5-0.03}{0-3/5+0.04}{blue}

\arrowleftcolor{1-0.5+0.02}{0+3/5-0.04}{red}
\arrowrightcolor{1-0.5}{0+1/5}{blue}
\arrowrightcolor{1-0.5}{0-1/5}{blue}
\arrowrightcolor{1-0.5-0.02}{0-3/5+0.04}{blue}

\arrowleftcolor{2-0.5+0.02}{0+3/5-0.04}{green}
\arrowleftcolor{2-0.5}{0+1/5}{red}
\arrowrightcolor{2-0.5}{0-1/5}{blue}
\arrowrightcolor{2-0.5-0.02}{0-3/5+0.04}{blue}

\arrowrightcolor{3-0.5}{0+3/5}{blue}
\arrowleftcolor{3-0.5}{0+1/5}{green}
\arrowleftcolor{3-0.5}{0-1/5}{red}
\arrowrightcolor{3-0.5-0.02}{0-3/5+0.04}{blue}

\arrowrightcolor{4-0.5}{0+3/5}{blue}
\arrowrightcolor{4-0.5}{0+1/5}{blue}
\arrowleftcolor{4-0.5}{0-1/5}{green}
\arrowleftcolor{4-0.5}{0-3/5}{red}

\arrowleftcolor{5-0.5+0.02}{0+3/5-0.04}{red}
\arrowrightcolor{5-0.5}{0+1/5}{blue}
\arrowrightcolor{5-0.5}{0-1/5}{blue}
\arrowleftcolor{5-0.5}{0-3/5}{green}

\arrowrightcolor{6-0.5}{0+3/5}{blue}
\arrowleftcolor{6-0.5}{0+1/5}{red}
\arrowrightcolor{6-0.5}{0-1/5}{blue}
\arrowrightcolor{6-0.5-0.02}{0-3/5+0.04}{blue}

\arrowrightcolor{7-0.5}{0+3/5}{blue}
\arrowrightcolor{7-0.5}{0+1/5}{blue}
\arrowleftcolor{7-0.5}{0-1/5}{red}
\arrowrightcolor{7-0.5-0.02}{0-3/5+0.04}{blue}

\arrowleftcolor{8-0.5+0.02}{0+3/5-0.04}{orange}
\arrowrightcolor{8-0.5}{0+1/5}{blue}
\arrowrightcolor{8-0.5}{0-1/5}{blue}
\arrowleftcolor{8-0.5}{0-3/5}{red}

\arrowrightcolor{9-0.5}{0+3/5}{blue}
\arrowleftcolor{9-0.5}{0+1/5}{orange}
\arrowrightcolor{9-0.5}{0-1/5}{blue}
\arrowrightcolor{9-0.5-0.02}{0-3/5+0.04}{blue}

\arrowrightcolor{10-0.5}{0+3/5}{blue}
\arrowrightcolor{10-0.5}{0+1/5}{blue}
\arrowleftcolor{10-0.5}{0-1/5}{orange}
\arrowrightcolor{10-0.5-0.02}{0-3/5+0.04}{blue}

\arrowrightcolor{11-0.5}{0+3/5}{blue}
\arrowrightcolor{11-0.5}{0+1/5}{blue}
\arrowrightcolor{11-0.5}{0-1/5}{blue}
\arrowleftcolor{11-0.5}{0-3/5}{orange}

\arrowrightcolor{12-0.5}{0+3/5}{blue}
\arrowrightcolor{12-0.5}{0+1/5}{blue}
\arrowrightcolor{12-0.5}{0-1/5}{blue}
\arrowrightcolor{12-0.5-0.03}{0-3/5+0.04}{blue}

\node () at (0,-1.8) {$x$};
\node () at (1,-1.8) {$qx$};
\node () at (2,-1.8) {$1$};
\node () at (3,-1.8) {$q^2$};
\node () at (4,-1.8) {$qx$};
\node () at (5,-1.8) {$q$};
\node () at (6,-1.8) {$1$};
\node () at (7,-1.8) {$qx$};
\node () at (8,-1.8) {$q$};
\node () at (9,-1.8) {$1$};
\node () at (10,-1.8) {$1$};
\node () at (11,-1.8) {$1$};

\node () at (0,1.5) {$1$};
\node () at (1,1.5) {$2$};
\node () at (2,1.5) {$3$};
\node () at (3,1.5) {$4$};
\node () at (4,1.5) {$5$};
\node () at (5,1.5) {$6$};
\node () at (6,1.5) {$7$};
\node () at (7,1.5) {$8$};
\node () at (8,1.5) {$9$};
\node () at (9,1.5) {$10$};
\node () at (10,1.5) {$11$};
\node () at (11,1.5) {$12$};
\end{tikzpicture}
\caption{The single row lattice state corresponding to the horizontal strip in Figure \ref{exmesp}, with colors indicating movements of particles. One can verify that the partition function is $q^7x^4$.}
\label{onerowlattice}
\end{figure*}
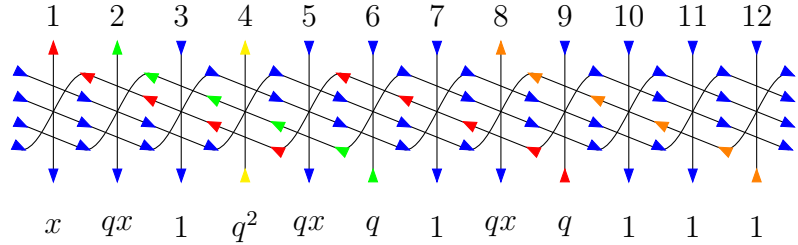

To check that this is an admissible lattice state, we need to check that each vertex is one of our admissible types: our assignment scheme ensures that straight edges won't change label (i.e. that $v_E(i) = v_W(i+1)$ for all vertices $v$), so it suffices to check the twisted edge/vertical edge combinations. As in Lemma \ref{lem:ribbonbij}, we may categorize by the coloring on the edge sequence: now we extend the edge sequence through the individual ribbons involved. Reading from bottom boundary to top boundary, if a column remains blue (no particle) throughout the entire strip, it generates a NW vertex; if remains red (particle) throughout the entire strip, it gives a SW vertex. If it begins red and ends blue, it makes a NS vertex, whereas if it begins blue and ends red it makes an EW vertex. The final option is for a label to start blue, turn red, and return to blue, which occurs when a moving path passes through an intermediate ribbon to become another ribbon's moving path (see column 5 of Figures \ref{exmesp} and \ref{onerowlattice}), which gives a NE vertex. Notice that there are no type SE vertices, so this lattice will have non-zero weight.

We may reverse this entire process to produce a ribbon filling from a lattice model: starting with the rightmost particle entering the lattice, start at the bottom boundary label for that column and step the ribbon left for each column it passes with no other particles on the vertical or twisted edge, and down if it ``encounters" another particle. Repeat this process for the remainder of the unfilled horizontal strip. Here it is important that the SE vertex is excluded, because this vertex comes from allowing the bottom-left-most square of a ribbon to touch the upper boundary of another ribbon rather than the bottom of the ribbon strip, which would violate the horizontal ribbon strip condition.

That the map is weight preserving follows naturally from the path interpretation of the ribbon lattices. In Lemma \ref{lem:ribbonbij}, the power of $x_i$ of the ribbon was assigned to the vertex where the long path exited the ribbon (type EW). In a ribbon strip, we see that path can exit the ribbon in either an EW vertex (if it is exiting the whole strip) or a NE vertex (if it is continuing into another ribbon), so both of these vertices contribute to the power of $x_i$. On the other hand, the power of $q$ counts the number of intersection of the paths, which is exactly the value of the spin as before, and all vertex types with path intersections (all except type NW) increment these powers of $q$ accordingly.
\end{proof}

\begin{prop}\label{prop:vribbon}
Using the weights from Figure \ref{fig:columnRibbonWts}, there is a weight-preserving bijection between vertical $n$-ribbon strips and one-row lattice models with non-zero weight.
\end{prop}

\begin{proof}
Consider the same bijection as in Proposition \ref{prop:hribbon}, sending a vertical ribbon strip of shape $\lambda/\mu$ to the one row model with boundary conditions $B_{\lambda/\mu}$ and assigning left arrows according to Lemma \ref{lem:ribbonbij} for each ribbon, then filling the remaining edges with right arrows. Label colorings of the edge sequence that produce NW, SW, NS, and EW vertices remain the same, but we also see SE vertices from labels that start red (particle), become blue (no particle), and return to red. Note that there will be no NE vertices, since these come from allowing the top-right-most square of a ribbon to touch the upper boundary of another ribbon rather than the bottom boundary of the ribbon strip, which is not allowed in a vertical ribbon strip, so the state will have nonzero weight. This also ensures that the reverse map will produce a vertical ribbon strip.

\begin{figure*}[h]
	\centering
	\begin{tikzpicture}[scale=0.8]
	    \draw [color=orange,thick] (2,6.5)--(6,6.5);
		\draw [rounded corners=12,color=red,thick] (3,3.5)--(0.5,3.5)--(0.5,2.5)--(0,2.5);
		\draw [color=green,thick] (1,2.5) cos (0.5,3) sin (0,3.5);
		\draw [color=green,thick] (2,1.5) cos (1.5,2) sin (1,2.5);
		\draw [color=cyan,thick] (1,0.5) cos (0.5,1) sin (0,1.5);
		\draw [color=magenta,thick] (3,4.5) cos (2.5,5) sin (2,5.5);
		\draw [color=yellow,rounded corners=12,thick] (4,5.5)--(2.5,5.5)--(2.5,4.5)--(1,4.5);
		\draw [color=violet,rounded corners=12,thick] (2,2.5)--(1.5,2.5)--(1.5,1.5)--(0.5,1.5)--(0.5,0.5)--(0,0.5);

		\draw (0,0)--(1,0)--(1,1)--(2,1)--(2,3)--(3,3)--(3,5)--(4,5)--(4,6)--(6,6)--(6,7)--(2,7)--(2,5)--(1,5)--(1,4)--(0,4)--cycle;
		\draw (0,2)--(1,2)--(1,3)--(2,3);
		\draw (1,4)--(3,4);
		\draw (2,6)--(4,6);
		\node () at (-.25,0.5) {1};
		\node () at (-.25,1.5) {2};
		\node () at (-.25,2.5) {3};
		\node () at (-0.25,3.5) {4};
		\node () at (0.25,4.25) {5};
		\node () at (.75,4.75) {6};
		\node () at (1.25,5.25) {7};
		\node () at (1.75,5.75) {8};
		\node () at (1.75,6.5) {9};
		\node () at (2.5,7.25) {10};
		\node () at (3.5,7.25) {11};
		\node () at (4.5,7.25) {12};
		\node () at (5.5,7.25) {13};
		
		\node () at (0.5,-0.25) {1};
		\node () at (1.25,0.25) {2};
		\node () at (1.75,0.75) {3};
		\node () at (2.25,1.5) {4};
		\node () at (2.25,2.25) {5};
		\node () at (2.75,2.75) {6};
		\node () at (3.25,3.5) {7};
		\node () at (3.25,4.25) {8};
		\node () at (3.75,4.75) {9};
		\node () at (4.25,5.25) {10};
		\node () at (4.75,5.75) {11};
		\node () at (5.5,5.75) {12};
		\node () at (6.25,6.5) {13};

	\end{tikzpicture}
	\caption{The edge sequence path of a vertical strip.}
	\label{edgevertical}
\end{figure*}
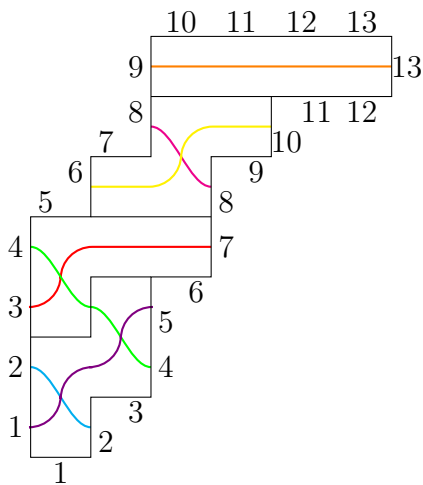

\begin{figure*}[h]

\centering
\begin{tikzpicture}[scale=0.85]
    \foreach \x in {0,1,...,12}{
\ribbonvertex{\x}{0}{\x}
}

\arrowdowncolor{0}{-1}{blue}
\arrowupcolor{1}{-1}{cyan}
\arrowdowncolor{2}{-1}{blue}
\arrowupcolor{3}{-1}{green}
\arrowupcolor{4}{-1}{violet}
\arrowdowncolor{5}{-1}{blue}
\arrowupcolor{6}{-1}{red}
\arrowupcolor{7}{-1}{magenta}
\arrowdowncolor{8}{-1}{blue}
\arrowupcolor{9}{-1}{yellow}
\arrowdowncolor{10}{-1}{blue}
\arrowdowncolor{11}{-1}{blue}
\arrowupcolor{12}{-1}{orange}

\arrowupcolor{0}{1}{violet}
\arrowupcolor{1}{1}{cyan}
\arrowupcolor{2}{1}{red}
\arrowupcolor{3}{1}{green}
\arrowdowncolor{4}{1}{blue}
\arrowupcolor{5}{1}{yellow}
\arrowdowncolor{6}{1}{blue}
\arrowupcolor{7}{1}{magenta}
\arrowupcolor{8}{1}{orange}
\arrowdowncolor{9}{1}{blue}
\arrowdowncolor{10}{1}{blue}
\arrowdowncolor{11}{1}{blue}
\arrowdowncolor{12}{1}{blue}

\arrowrightcolor{0-0.5}{0+3/5}{blue}
\arrowrightcolor{0-0.5}{0+1/5}{blue}
\arrowrightcolor{0-0.5}{0-1/5}{blue}
\arrowrightcolor{0-0.5-0.03}{0-3/5+0.04}{blue}

\arrowleftcolor{1-0.5+0.02}{0+3/5-0.04}{violet}
\arrowrightcolor{1-0.5}{0+1/5}{blue}
\arrowrightcolor{1-0.5}{0-1/5}{blue}
\arrowrightcolor{1-0.5-0.02}{0-3/5+0.04}{blue}

\arrowrightcolor{2-0.5+0.02}{0+3/5-0.04}{blue}
\arrowleftcolor{2-0.5}{0+1/5}{violet}
\arrowrightcolor{2-0.5}{0-1/5}{blue}
\arrowrightcolor{2-0.5-0.02}{0-3/5+0.04}{blue}

\arrowleftcolor{3-0.5}{0+3/5}{red}
\arrowrightcolor{3-0.5}{0+1/5}{blue}
\arrowleftcolor{3-0.5}{0-1/5}{violet}
\arrowrightcolor{3-0.5-0.02}{0-3/5+0.04}{blue}

\arrowrightcolor{4-0.5}{0+3/5}{blue}
\arrowleftcolor{4-0.5}{0+1/5}{red}
\arrowrightcolor{4-0.5}{0-1/5}{blue}
\arrowleftcolor{4-0.5}{0-3/5}{violet}

\arrowrightcolor{5-0.5+0.02}{0+3/5-0.04}{blue}
\arrowrightcolor{5-0.5}{0+1/5}{blue}
\arrowleftcolor{5-0.5}{0-1/5}{red}
\arrowrightcolor{5-0.5}{0-3/5}{blue}

\arrowleftcolor{6-0.5}{0+3/5}{yellow}
\arrowrightcolor{6-0.5}{0+1/5}{blue}
\arrowrightcolor{6-0.5}{0-1/5}{blue}
\arrowleftcolor{6-0.5-0.02}{0-3/5+0.04}{red}

\arrowrightcolor{7-0.5}{0+3/5}{blue}
\arrowleftcolor{7-0.5}{0+1/5}{yellow}
\arrowrightcolor{7-0.5}{0-1/5}{blue}
\arrowrightcolor{7-0.5-0.02}{0-3/5+0.04}{blue}

\arrowrightcolor{8-0.5+0.02}{0+3/5-0.04}{blue}
\arrowrightcolor{8-0.5}{0+1/5}{blue}
\arrowleftcolor{8-0.5}{0-1/5}{yellow}
\arrowrightcolor{8-0.5}{0-3/5}{blue}

\arrowleftcolor{9-0.5}{0+3/5}{orange}
\arrowrightcolor{9-0.5}{0+1/5}{blue}
\arrowrightcolor{9-0.5}{0-1/5}{blue}
\arrowleftcolor{9-0.5-0.02}{0-3/5+0.04}{yellow}

\arrowrightcolor{10-0.5}{0+3/5}{blue}
\arrowleftcolor{10-0.5}{0+1/5}{orange}
\arrowrightcolor{10-0.5}{0-1/5}{blue}
\arrowrightcolor{10-0.5-0.02}{0-3/5+0.04}{blue}

\arrowrightcolor{11-0.5}{0+3/5}{blue}
\arrowrightcolor{11-0.5}{0+1/5}{blue}
\arrowleftcolor{11-0.5}{0-1/5}{orange}
\arrowrightcolor{11-0.5-0.02}{0-3/5+0.04}{blue}

\arrowrightcolor{12-0.5}{0+3/5}{blue}
\arrowrightcolor{12-0.5}{0+1/5}{blue}
\arrowrightcolor{12-0.5}{0-1/5}{blue}
\arrowleftcolor{12-0.5}{0-3/5}{orange}

\arrowrightcolor{13-0.5}{0+3/5}{blue}
\arrowrightcolor{13-0.5}{0+1/5}{blue}
\arrowrightcolor{13-0.5}{0-1/5}{blue}
\arrowrightcolor{13-0.5-0.03}{0-3/5+0.04}{blue}

\node () at (0,-1.8) {$-y$};
\node () at (1,-1.8) {$q$};
\node () at (2,-1.8) {$-y$};
\node () at (3,-1.8) {$q^2$};
\node () at (4,-1.8) {$1$};
\node () at (5,-1.8) {$-y$};
\node () at (6,-1.8) {$1$};
\node () at (7,-1.8) {$q$};
\node () at (8,-1.8) {$-y$};
\node () at (9,-1.8) {$1$};
\node () at (10,-1.8) {$1$};
\node () at (11,-1.8) {$1$};
\node () at (12,-1.8) {$1$};

\node () at (0,1.5) {$1$};
\node () at (1,1.5) {$2$};
\node () at (2,1.5) {$3$};
\node () at (3,1.5) {$4$};
\node () at (4,1.5) {$5$};
\node () at (5,1.5) {$6$};
\node () at (6,1.5) {$7$};
\node () at (7,1.5) {$8$};
\node () at (8,1.5) {$9$};
\node () at (9,1.5) {$10$};
\node () at (10,1.5) {$11$};
\node () at (11,1.5) {$12$};
\node () at (12,1.5) {$13$};
\end{tikzpicture}
\caption{The single row lattice state corresponding to the vertical strip in Figure \ref{edgevertical}, with colors indicating movements of particles. One can verify that the partition function is $q^4(-y)^4$.}
\label{onerowlatticevert}
\end{figure*}
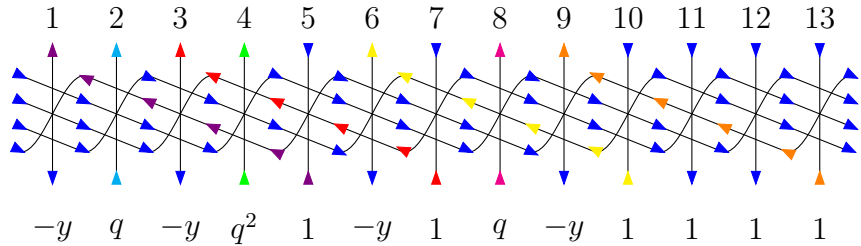

Again, the preservation of weight follows naturally from the path interpretation of the ribbon lattices. The weight of the lattice model obtains a power of $(-y_{i'})$ from each time a travelling particle exits out the top boundary. These exits occur on vertices of type SE and EW, and each of them will come from a single $n$-ribbon in the vertical strip.  To match the power of $q$, notice that unlike the horizontal strip lattice, we can't ``bounce'' an exiting travelling particle of one ribbon out through another ribbon, so all the intersections contributing to the spin in the vertical ribbon lattice happen on type SW vertices, which are precisely the labels mid-ribbon at which a given ribbon takes a step up, and therefore contribute a single power of $q$ to the weight of that ribbon.
\end{proof}

\begin{proof}[Proof of Theorem \ref{mainthm}]
Consider a super ribbon tableau of shape $\lambda/\mu$ under the total order $\prec$ on alphabets $A,A'$. Stripping off ribbons label by label generates a series of skew shapes according to the total order: we define a sequence of partitions  $\lambda_{j}$ such that the ribbon strip labelled $i_j \in A\cup A'$ will have shape $\lambda_{i_j}/\lambda_{i_{j-1}}$. Using this sequence, assign vertical edge labels to the state with boundary conditions $B_{\lambda/\mu}(X/Y)$. Thus, the columns will read $\mu = \lambda_0, \lambda_1, ...., \lambda_{r+s} = \lambda$, so that the $i_j$-th row is labelled $\lambda_{i_j}$ on the bottom and $\lambda_{i_{j-1}}$ on the top. Marshalling Propositions \ref{prop:hribbon} and \ref{prop:vribbon} together, this process defines a weight-preserving bijection and thus summing over all ribbon tableaux/states, the partition function of this system equals the desired super LLT polynomial.
\end{proof}

\section{Solvability of the Ribbon Lattice Model}
\label{sec:solvability}
Lattice models satisfying a Yang-Baxter equation are called \textit{exactly solvable}, or \textit{integrable}. From the lattice model perspective, the Yang Baxter equation gives a consistent way of effectively permuting the weights of an admissible state without altering the partition function. For certain sets of weights, it is an useful tool for showing symmetry or recursion relations on the partition functions. 


	
The Yang-Baxter Equation for a $n$-ribbon lattice model requires choosing Boltzmann weights for a set of new \emph{diagonal} vertices of in-degree $2n$ and out-degree $2n$, called $R^{(n)}$-vertices, such that the partition functions of the two lattice models in (\ref{YBE}) are equal for any set of fixed boundary conditions. For a complete treatment of Yang-Baxter equations of 2-$d$ square lattice models, we refer the reader to \cite[Section 1 \& 5]{BBF09} and \cite[Chapter 8 \& 9]{yang-baxter}.
\begin{equation}\label{YBE}
	Z\left(\ \tikz[baseline=.1ex,scale=0.8]{
		\spiderR{-0.2}{0}{1}
		\ribbonvertex{1.5}{1}{1}
		\ribbonvertex{1.5}{-1}{1}
		\node () at (2.5,1) {$v_i$};
		\node () at (2.5,-1) {$v_j$};
		\node () at (-0.2,1.5) {$R^{(n)}_{i,j}$}
	}\ \right)\ =\ 
Z \left(\ \tikz[baseline=.1ex,scale=0.8]{
		\spiderR{0.2}{0}{1}
		\ribbonvertex{-1.5}{1}{1}
		\ribbonvertex{-1.5}{-1}{1}
		\node () at (-2.5,1) {$v_j$};
		\node () at (-2.5,-1) {$v_i$};
		\node () at (0.2,1.5) {$R^{(n)}_{i,j}$}
	}\ \right)
\end{equation}

We shall view these $R^{(n)}$-vertices as a combination of $n$ individual diagonal $R^{(1)}$-vertices, which allows us to define the weight of a $R^{(n)}$-vertex as a product of weights of $n$ $R^{(1)}$-vertices. This process is commonly known as \emph{fusion}, as it results from the fusion procedure defined for tensor products of quantum group modules and their subquotients by Kulish, Reshetikhin, and Sklyanin \cite{KRSfusion}. In the case of this model, the associated quantum group is $U_q (\widehat{\frak{sl}}(1|r) )$, where $r$ is maximum number of parts in $\lambda$. 

\begin{defn}\label{def:R-vertex}
Let $r$ be a $R^{(n)}$-vertex depicted as $n$ stacked $R^{(1)}$ vertices. That is, depicting $r$ as
\begin{center}
\begin{tikzpicture}[baseline = -2]
\node (i) at (1,1) {$\boldsymbol{I}$};
\node (j) at (1,-1) {$\boldsymbol{J}$};
\node (k) at (-1,-1) {$\boldsymbol{K}$};
\node (l) at (-1,1) {$\boldsymbol{L}$};
\draw[thick] (i)--(k);
\draw[thick] (j)--(l);
\node at (2,0){$=$};
\end{tikzpicture}
\begin{tikzpicture}[baseline=0]
\spiderR{0}{0}{1}
\node at (1.7,1.6) {$I_1$};
\node at (1.7,1.2) {$I_2$};
\node at (1.7,0.8) {$\cdots$};
\node at (1.7,0.4) {$I_n$};

\node at (1.7,-1.6) {$J_n$};
\node at (1.7,-1.2) {$\cdots$};
\node at (1.7,-0.8) {$J_2$};
\node at (1.7,-0.4) {$J_1$};

\node at (-1.7,1.6) {$L_1$};
\node at (-1.7,1.2) {$L_2$};
\node at (-1.7,0.8) {$\cdots$};
\node at (-1.7,0.4) {$L_n$};

\node at (-1.7,-1.6) {$K_n$};
\node at (-1.7,-1.2) {$\cdots$};
\node at (-1.7,-0.8) {$K_2$};
\node at (-1.7,-0.4) {$K_1$};

\end{tikzpicture}

\end{center}
where $I_k,J_k,K_k,L_k\in \{>,<\}$ for $k\in[n]$, we set
\[r_k=
\tikz[baseline=.1ex,scale=0.7]{
\node (i) at (1,1) {$I_k$};
\node (j) at (1,-1) {$J_k$};
\node (k) at (-1,-1) {$K_k$};
\node (l) at (-1,1) {$L_k$};
\draw[thick] (i)--(k);
\draw[thick] (j)--(l);
}.
\]
and define $\wt(r)=\prod_{k=1}^n\wt(r_k)$
where $\wt(r_k)$ is chosen from Figure \ref{fig:verticalRwts} and depends on the types of rows being crossed. 
\end{defn}


Owing to the choice of $q$-powers on the weights in Figures \ref{fig:RibbonWts} and \ref{fig:columnRibbonWts} necessary to hit the LLT polynomials precisely on the nose, this Yang Baxter equation solution is not precisely a fusion model; that is to say, the weights of certain types of $R^{(1)}$-vertices contributing to the $R^{(n)}$ depend on the other types of vertices appearing. In each of the cases of swapping rows of the same type (horizontal - horizontal or vertical - vertical) this dependence affects only the power of $q$ on only one type of vertex. To define this normalization power for the $k$-th piece $r_k$, $k\in\{1,...,n\}$, of a $R^{(n)}$-vertex swapping horizontal strip rows, let
\begin{align*} \theta&=2\cdot{\#\{r_t = \text{S}|t>k \} },\\
\sigma&=\#\{r_t=\text{SS}|t>k\}+\#\{r_t=\text{NN}|t<k\}+\#\{r_t=\text{W}\}.
\intertext{Graphically, the number $\theta$ is $2$ times the number of other S-vertices below; the number $\sigma$ is the sum of the number of SS-vertices below, the number of NN-vertices above and the number of all W-vertices. Similarly, for $r_k$ in a $R^{(n)}$-vertex swapping vertical strip rows, we will need:}
\theta'&=2\cdot\#\{r_t = \text{N}|t>k \},\\
\sigma'&=\#\{r_t=\text{SS}|t<k\}+\#\{r_t=\text{NN}|t>k\}+\#\{r_t=\text{W}\}.
\intertext{Lastly, for $r_k$ in vertices that swap rows of different type, we need the quantities}
\theta''&=2\cdot\#\{r_t = \text{E}|t>k \},\\
\sigma'' &= \#\{r_t=\text{SS}|t<k\}+\#\{r_t=\text{NN}|t>k\}+\#\{r_t=\text{S}\}.
\end{align*}

\begin{figure}[h]
\begin{center}
 \begin{tabular}{| M{1.8cm} |M{2.4cm} | M{1.8cm} | M{1.8cm} | M{2.4cm} | M{1.8cm} | M{1.8cm} |} \hline
 label& N & SS  & W &E & NN &S \\
\hline 
$r_k$ &\north &  \southsouth & \west & \east & \northnorth & \south\\
 \hline
$\wt_{HH}(r_k)$ & $0$ & $x_j$ & $x_j$ & $x_i$ & $x_i$ & $\frac{q^{\theta}x_i-x_j}{q^{\theta+\sigma}} $ \Top\Bottom \\
\hline
$\wt_{VV}(r_k)$ & ${ q^{\sigma'}(y_j - q^{\theta'}y_i)}$ & $y_j$ & $y_i$ & $y_j$ & $y_i$ & $0$  \Top\Bottom\\
\hline
$\wt_{HV}(r_k)$ & $q^{\sigma''}(-y_j)$ & $q^{\sigma''}(-y_j)$ & $0$ & $\frac{q^{2(n-1)}x_i - q^{\theta''}y_j}{q^{\theta''}}$ & $q^{\sigma''}x_i$ & $q^{\sigma''}x_i$ \Top\Bottom  \\
\hline
 \end{tabular}\end{center}
 
\caption{The weights for the $k$-th $R^{(1)}$-vertex in an $R^{(n)}$-vertex swapping rows $i,j$, where $HH$ denotes swapping horizontal strip rows, $VV$ swapping vertical strip rows, and $HV$ swapping a horizontal strip row $i$ with a vertical strip row $j$.}\label{fig:verticalRwts}
\end{figure}

\begin{theorem}\label{thm:horizontalYBE}
Together with the horizontal strip row weights in Figure \ref{fig:RibbonWts}, the $HH$ Boltzmann weights for $R^{(n)}$-vertices given in Definition \ref{def:R-vertex} give a solution to the Yang-Baxter equation for any $n\geq 1$.
\end{theorem}

\begin{proof}
	The key to solving the Yang-Baxter equation for higher ribbon lattices is using the property that arrows do not change along the straight edges of rectangular vertices. This implies that most of the interior edges are fixed and only the twisted edges are variables. Therefore given any set of boundary conditions, the star-triangle Yang-Baxter identity can be illustrated as follows,
where we label the interior edges that are not fixed by the boundary condition in red.

	\begin{equation}\label{eq:startriangle}
	\sum_{\phi,\xi,\psi}
	\ \tikz[baseline=.1ex,scale=1.2]{
		\spiderR{0.2}{0}{1}
		\ribbonvertex{-1.5}{1}{1}
		\ribbonvertex{-1.5}{-1}{1}
		
		\node at (-1.5,-2.2) {$\alpha$};
	
		\node at (-2.3,1.6) {$b_1$};
		\node at (-2.3,1.2) {$b_2$};
		\node at (-2.3,0.8) {$\cdots$};
		\node at (-2.3,0.4) {$b_n$};
		
		\node at (-2.3,-1.6) {$c_n$};
		\node at (-2.3,-1.2) {$\cdots$};
		\node at (-2.3,-0.8) {$c_2$};
		\node at (-2.3,-0.4) {$c_1$};
		
		\node at (-1.5,2.2) {$a$};
		
		\node at (1.65,-1.6) {$\beta_n$};
		\node at (1.65,-1.2) {$\cdots$};
		\node at (1.65,-0.8) {$\beta_2$};
		\node at (1.65,-0.4) {$\beta_1$};
		
		\node at (1.65,1.6) {$\gamma_1$};
		\node at (1.65,1.2) {$\gamma_2$};
		\node at (1.65,0.8) {$\cdots$};
		\node at (1.65,0.4) {$\gamma_n$};
		
		\node [fill=white, inner sep= .7]at (-0.9,1.6) {$\color{red}\phi$};
		\node [fill=white, inner sep= .7]at (-0.9,1.2) {$b_1$};
		\node [fill=white, inner sep= .7]at (-0.9,0.8) {$b_2$};
		\node [fill=white, inner sep= .7]at (-0.9,0.4) {$\cdots$};
		
		\node [fill=white, inner sep= .7]at (-0.9,-1.6) {$\cdots$};
		\node [fill=white, inner sep= .7]at (-0.9,-1.2) {$c_2$};
		\node [fill=white, inner sep= .7]at (-0.9,-0.8) {$c_1$};
		\node [fill=white, inner sep= .7]at (-0.9,-0.4) {$\color{red}\psi$};
		
		\node [fill=white, inner sep= 1.6]at (-1.5,0) {\color{red}$\xi$};

	}\ 
	=
	\sum_{\theta,\sigma,\delta}\ \tikz[baseline=.1ex,scale=1.2]{
		\spiderR{-0.2}{0}{1}
		\ribbonvertex{1.5}{1}{1}
		\ribbonvertex{1.5}{-1}{1}
		\node at (1.5,-2.2) {$\alpha$};
		
		\node at (2.3,-1.6) {$\beta_n$};
		\node at (2.3,-1.2) {$\cdots$};
		\node at (2.3,-0.8) {$\beta_2$};
		\node at (2.3,-0.4) {$\beta_1$};
		
		\node at (2.3,1.6) {$\gamma_1$};
		\node at (2.3,1.2) {$\gamma_2$};
		\node at (2.3,0.8) {$\cdots$};
		\node at (2.3,0.4) {$\gamma_n$};
		
		\node [fill=white, inner sep= .7]at (0.9,1.6) {$\gamma_2$};
		\node [fill=white, inner sep= .7]at (0.9,1.2) {$\cdots$};
		\node [fill=white, inner sep= .7]at (0.9,0.8) {$\gamma_n$};
		\node [fill=white, inner sep= .7]at (0.9,0.4) {$\color{red}\theta$};
		
		\node [fill=white, inner sep= .7]at (0.9,-1.6) {$\color{red}\sigma$};
		\node [fill=white, inner sep= .7]at (0.9,-1.2) {$\beta_n$};
		\node [fill=white, inner sep= .7]at (0.9,-0.8) {$\cdots$};
		\node [fill=white, inner sep= .7]at (0.9,-0.4) {$\beta_2$};
		
		\node [fill=white, inner sep= 1.6]at (1.5,0) {\color{red}$\delta$};
		
		\node at (1.5,2.2) {$a$};
		
		\node at (-1.65,-1.6) {$c_n$};
		\node at (-1.65,-1.2) {$\cdots$};
		\node at (-1.65,-0.8) {$c_2$};
		\node at (-1.65,-0.4) {$c_1$};
		
		\node at (-1.65,1.6) {$b_1$};
		\node at (-1.65,1.2) {$b_2$};
		\node at (-1.65,0.8) {$\cdots$};
		\node at (-1.65,0.4) {$b_n$};
		
	}\ 
\end{equation}
Notice that the last $n-1$ pieces of the $R^{(n)}$-vertex on the left hand side are identical to the first $n-1$ pieces of the $R^{(n)}$-vertex on the right hand side. We may thus treat these weights of these vertices as a block, and to prove the YBE we must examine how this block interacts with the rectangular vertices and the remaining $R^{(1)}$ vertex on each side. We do this by enumerating all possible combinations of boundary conditions $(a,b_n,c_n,\alpha,\beta_1,\gamma_1)$ not specified by the block: we will call strands connected to these edges the ``underlying YBE."

If $q=1$, this block has equal weight on both sides of the YBE and does not interact with the rest of the vertices at all on either side, so \Cref{eq:startriangle} reduces to the 1-dimensional underlying YBE. When $q=1$, the horizontal strip row model recovers an 180 degree rotation of the Schur model $S^\Delta$ of \cite{BBF09} with $t=0$, and our diagonal weights specialize to their YBE solution, accounting for this rotation. 
Therefore the underlying YBE holds, and the only difficulty is to check that the weights involving powers of $q$ are matched up. 

Since each vertex involved has an equal number of arrows going in and coming out, these boundary conditions must observe the same property in order to be admissible: we then have $\binom{6}{3}$ choices dictated by choosing the ``in" arrows. However, the configurations with $\a = $ in, $\b_1 =$ in will have weight 0, as will those with $b_n =$ out, $a = $ out, since SE vertices have weight 0. Two of these cases overlap, so we are left with 14 sets of boundary conditions to check. Notice that in the other four weight 0 cases, one side of the underlying YBE does not have a SE vertex, but does have a $R^{(1)}$ vertex of type N, which is precisely why type $N$ must have weight 0. The set of these nonzero boundary conditions is given in Appendix \ref{sec:rybe}. We will show that, in each case, we may reduce the more complicated Yang-Baxter equation to that of the  underlying YBE. Since many of our vertices have similar weight, we may group into cases.

\paragraph{Case 1: \Cref{ybe1,ybe2,ybe3,ybe4,ybe5,ybe6,ybe7,ybe8}} First we consider \Cref{ybe1,ybe2,ybe3,ybe4,ybe5,ybe6,ybe7,ybe8}, which each have one state on each side of the YBE, but no NW vertices on either side. For example, \Cref{ybe1} is given below. 
\[\ybeleft101101100\ =\ \yberight101101001 \tag{A.1}\]
Denote the $R^{(n)}$-vertex on the left hand side $\calL$ and the one on the right hand side $\calR$. Denote the ribbon vertices on the left hand side $v_j,v_i$ and those on the right hand side $w_i,w_j$ as illustrated above, where $i,j$ are row indices. Note that $\calL_1$ and $\calR_n$ are the ``underlying" $R^{(1)}$ vertices. The rest of the $R^{(1)}$ vertices are part of the ``block," so $\calL_{t+1} = \calR_{t}$ for $t = 1,....,n-1$.
The equality we would like to prove is
\[\wt(v_i)\wt(v_j)\prod_{t=1}^{n}\wt(\calL_t)=\wt(w_i)\wt(w_j)\prod_{t=1}^{n}\wt(\calR_t)\tag{$*$}\label{eq:star}\]
In \Cref{ybe1,ybe2,ybe3,ybe4,ybe5,ybe6,ybe7,ybe8}, none of $v_i,v_j,w_i,w_j$ is of type NW, meaning that every left arrow in $\{b_1,\cdots,b_{n-1},c_1,\cdots,c_{n-1}\}$ will contribute to a power of $q$ in the weight of $v_j$ or $v_i$. For the same reason every left arrow in $\{\gamma_2,\cdots,\gamma_n,\beta_2,\cdots,\beta_n\}$ will contribute to a power of $q$ in the weight of $w_i$ or $w_j$.  By \Cref{def:R-vertex}, any piece of $\calL$ and $\calR$ has the same number of in and out arrows. So in an admissible configuration, the number of left arrows in $\{b_1,\cdots,b_{n-1},c_1,\cdots,c_{n-1}\}$ must be the same as the number of left arrows in $\{\gamma_2,\cdots,\gamma_n,\beta_2,\cdots,\beta_n\}$, so we have equal factors of $q$ on both sides coming from the rectangular vertices. Furthermore, since neither $\calL_1$ nor $\calR_n$ are of type S, we know that the underlying YBE does not depend on the vertices in the block. Therefore we only need to prove the following:
\[\prod_{t=2}^{n}\wt(\calL_t)=\prod_{t=1}^{n-1}\wt(\calR_t)\]
Since $\calL_{t+1}=\calR_{t}$ for all $t=1,...,n-1$, the above equality is true when there is no S-vertex in the block to be affected by $\calL_1$ or $\calR_n$. If there is, in \Cref{ybe1,ybe2,ybe3,ybe4}, $\calL_1$ contributes to a power of $q^{-1}$ for any S-vertex below, and $\calR_n$ contributes to a power of $q^{-1}$ for any S-vertex above. In \Cref{ybe5,ybe6,ybe7,ybe8}, $\calL_1$ and $\calR_1$ do not affect the weight of other pieces of $\calL$ and $\calR$.
Therefore, if we have an $S$ vertex in the block, i.e. if $\calL_{k+1}$ and $\calR_k$ are S-vertices for some $1\leq k\leq n-1$, their weight must be the same. This proves \Cref{ybe1,ybe2,ybe3,ybe4,ybe5,ybe6,ybe7,ybe8}.

\paragraph{Case 2: \Cref{ybe11}} Next we look at \Cref{ybe11}, which has one state on each side, each with one NW vertex. \[\ybeleft001001001\ =\ \yberight001001100\tag{A.9}\] 

We examine the different types of vertices that might appear in the block and how each affects the rectangular vertices. If $\calL_{t+1},\calR_{t}$ are SS-vertices for some $1\leq t\leq n-1$, then $\calL_t$ contributes to a power of $q$ in $\wt(v_i)$ as well as a power of $q^{-1}$ to $\wt(\calL_1)$, while $\calR_{t-1}$ has no effect on the weight of other vertices on the right hand side. 

Similarly, a type NN vertex has no effect on the left hand side, but contributes a power of $q$ to $\wt(w_i)$ and a power of $q^{-1}$ to $\wt(\calR_n)$ on the right hand side. A type W vertex affects both sides: it contributes to a power of $q$ in $\wt(v_i)$ and a $q^{-1}$ in $\wt(\calL_1)$ on the left hand side, as well as a matching power of $q$ in $\wt(w_i)$ and $q^{-1}$ in $\wt(\calR_n)$ on the right. A type E vertex has no effect on either side.

It remains to check type S: if $\calL_t,\calR_{t-1}$ are S-vertices for some $1\leq t\leq n-1$, then 
$\calL_t$ contributes to a $q$ in $\wt(v_i)$ while $\calR_{t-1}$ contributes to a $q$ in $\wt(w_j)$. Moreover, the $\sigma$-coefficients must be the same because they have the same pattern of NN, SS or W-vertices around them. We are only left to check the $\theta$-coefficients of the S-vertices. Supposing there are $m$ S-vertices in $\calL$ and $\calR$ (including $\calL_1$ and $\calR_n$), then the weight after ignoring the $\sigma$-coefficients will be
\[\prod_{t=1}^{m}(x_i-q^{-2(t-1)}x_j)\]
on both sides. Note that this type does technically affect the underlying YBE, in that the bottom-most S-vertex on each side plays the role of the $R^{(1)}$ S-vertex from the $n=1$ case, but the factor of $x_i$  coming from vertices $v_i,w_i$ respectively is unaffected by this change.
It follows that \Cref{ybe11} is true.

\paragraph{Case 3: \Cref{ybe14}} This is the last case with only one state per side, each with two NW vertices. 
\[\ybeleft000000000\ =\ \yberight000000000\tag{A.10}\]
This is the easiest case to check, because the underlying YBE and the block $R$-vertices have no interaction with each other, and none of the rectangular weights admit additional powers of $q$. Since the block's weight is then the same on both sides and the underlying YBE holds, we are done.

\paragraph{Case 4: \Cref{ybe9,ybe12}} The last four sets of boundary conditions each have one side with two states and one side with one. They split into two cases, depending on whether the additional state is on the left side or the right. First, we tackle the ones on the left. Consider \Cref{ybe9}.
\[\ybeleftb100001100+\ybeleft100001011\ =\ \yberight100001000	 \tag{A.11}\]
As in Case 2, most types of vertices that could appear in the block will only affect the partition function by contributing powers of $q$. For any $1\leq t\leq n-1$, if $\calK_{t+1}, \calL_{t+1}, \calR_t$ is an E-vertex, it doesn't contribute additional factors to either side. If it is a $W$-vertex, it contributes a factor of $q$ to $\wt(u_j)$ in the first term on the left hand side and a factor of $q$ to  $\wt(v_j)$, a factor of $q$ to $\wt(v_i)$, and a factor of $q^{-1}$ to $\wt(\calL_1)$ in the second term. On the right hand side, it contributes a power of $q$ to $\wt(w_i)$. A type NN vertex will contribute a power of $q$ to each term, specifically to vertices $\wt(u_j), \wt(v_j),$ and $\wt(w_i)$. Lastly, a type SS vertex contributes nothing to the right hand side or the first term on the left, but a power of $q$ to $\wt(v_i)$ and a power of $q^{-1}$ to $\wt(\calL_1)$ in the second term on the left.

Unlike in the rest of our considerations, having S-vertices in the block in this case affects the underlying YBE more than just superficially. Unless $q= 1$, we must consider the original factors involving $x_i,x_j$ as well as new factors created by the interaction of the block with the rectangular vertices, $\calK_1, \calL_1,$ and $\calR_n$. Let $s_\g$ be the number of left arrows in $\g_2,...,\g_n$, $s_b$ be the number of left arrows in $b_1,...,b_{n-1}$, and $s_c$ be the number of left arrows in $c_1,...,c_{n-1}$. Let $\#(-)$ return the number of vertices of a given type in the block.

We consider the right hand side first: ignoring weights that are internal to the block and thus constant across all states, we have a factor of $q^{s_\g}x_i$ from $w_i$ and a factor of $x_i$ from $\calR_n$.

On the left hand side, the first term garners a factor of $q^{s_b}x_i$ from $u_j$ and a factor of $q^{-\#(S)}x_j$ from $\calK_1$, since it is type NN and therefore contributes a power of $q^{-1}$ to every type S vertex below it. In the second term, all three non-block vertices contribute: $\wt(v_j) = q^{s_b}$, $\wt(v_i) = q^{s_c}x_i$, and  $\wt(\calL_1) = \frac{q^{2\#(S)}x_i - x_j}{q^{2\#(S) + \#(W) + \#(SS)}}$.

So, the new factor in the weight on the left hand side is
\begin{align*}
    x_i\left( q^{s_b - \#(S)}x_j + q^{s_b+s_c - \#(W) - \#(NN)}x_i - q^{s_b+s_c - 2\#(S) - \#(W) - \#(NN)}x_j\right).
\end{align*}
Noting that $s_b = \#(W) + \#(NN)$ and $s_c = \#(W) + \#(S) + \#(SS)$, we see that the $x_j$ terms within the parentheses cancel and the power of $q$ on the $x_i$ term becomes $\#(W) + \#(NN) + \#(S)$, which is precisely $s_\g$, so our YBE is satisfied.

Since rectangular vertex types SW and NS have the same weight, as do types EW and NE, \Cref{ybe12} can be proved using the same argument.

\paragraph{Case 5: \Cref{ybe10,ybe13}}
The last two sets of boundary conditions work very similarly to Case 4, except that the additional state falls on the right hand side. In \Cref{ybe10}, denote the vertices as follows:
\[\ybeleft001100000\ =\ \yberightb001100001+\yberight001100110\tag{A.13}\]

We again consider first the vertex types that affect the partition function only by a power of $q$. Type E has no effect on either side. Type W adds a power of $q$ to $\wt(v_i)$ on the left hand side, a power of $q$ to $\wt(y_j)$ in the first term on the right hand side, and powers $q, q,$ and $q^{-1}$ respectively to $\wt(w_i),\wt(w_j),$ and $\wt(\calR_n)$ in the second term on the right hand side. Types NN and SS switch places from Case 4: here, type SS adds a power of $q$ to every term, specifically to $\wt(v_i), \wt(y_j),$ and $\wt(w_j)$, whereas type NN has no effect on the left hand side or the first vertex on the right, but contributes a power of $q$ to $\wt(w_i)$ and $q^{-1}$ to $\wt(\calR_n)$ in the second term on the right.

As in the last case, having an S-vertex in the block does affect the underlying YBE: we will again consider all the weights that are not wholly internal to the block. Let $s_c, s_\g, \#(-)$ be as in Case 4 and let $s_\b$ similarly be the number of left arrows in $\b_2,...,\b_n$. 

We begin with the left hand side: we see contributing factors from $\wt(v_i) = q^{s_c}$ and $\wt(\calL_1) = x_i$.

On the right hand side, the first term is fairly straightforward: vertex $y_j$ contributes a power of $q^{s_\b}$ and $\calQ_n$ a power of $x_j$. As well, $\calQ_n$ adds a power of $q^{-1}$ to any S-vertex in the block, so we obtain an additional factor of $q^{-\#(S)}$.

The second term on the right is more complicated: since $\calR_n$ is type S, it affects the weight of any S-vertex in the block. This seems to run contrary to our desire to group the block vertices into a set of weights constant on every term. However, we may reuse the same trick from Case 2:  a S-vertex in the block that had $k$ S-vertices below it in any other term has $k+1$ below it in this second term, so its weight changes from $q^{-\sigma}\frac{q^{2k}x_i - x_j}{q^{2k}}$ to $q^{-\sigma}\frac{q^{2(k+1)}x_i - x_j}{q^{2(k+1)}}$. Note that these both have the same $\sigma$, which depends only on vertices in the block. Ignoring this factor, we see that the weight of all the S-vertices in $\calR$, including $\calR_n$ is $\prod_{k=0}^{\#(S)} (x_i - q^{-2k}x_j)$. In contrast, the corresponding part of the weight of S-vertices in any other term is  $\prod_{k=0}^{\#(S)-1} (x_i - q^{-2k}x_j)$, so our new factor is $x_i - q^{-2\#(S)}x_j$. We may think of this as cycling the $\theta$-parts of the weights of S-vertices in $\calR$: $\calR_n$ takes the place of the S-vertex above it, and so on, until our new factor comes from the top S-vertex in the block. We also obtain a factor of $q^{-\#(W)- \#(NN)}$ from the $\sigma$-part of $\wt(\calR_n)$, a factor of $q^{s_\g}$ from vertex $w_i$ and a factor of $q^{s_\b}$ from $w_j$.

Thus, the new factor in the weight on the right hand side is
\[q^{s_\b - \#(S)}x_j + q^{s_\b + s_\g -\#(W) - \#(NN)}(x_i - q^{-2\#(S)}x_j). \]
Noting that $s_\b = \#(W) + \#(SS)$ and $s_\g = \#(W) + \#(S) + \#(NN)$, we see that the $x_j$ terms in this factor cancel and power of $q$ on the remaining $x_i$ term is $\#(W) + \#(S) + \#(SS)$, which matches $s_c$. Thus the YBE holds.

As in Case 4, our two sets of boundary conditions \Cref{ybe10} and \Cref{ybe13} work very similarly. To swap to \Cref{ybe13}, we notice that rectangular vertex types NS and NE differ by a power of $x_i$ on the left hand side and $\calL_1$ swaps from type E (weight $x_i$) to type SS (weight $x_j)$, so we modify the weight from \Cref{ybe10} by an additional factor of $x_j$ on the left hand side. On the right hand side, we see only this same factor appear in the change from vertex type NS to NE and SW to EW in the two terms, respectively, so our YBE holds by the same arguments as above.

\end{proof}

\begin{theorem}\label{thm:verticalYBE}

Together with the weights in Figure \ref{fig:columnRibbonWts}, the $VV$ Boltzmann weights for $R^{(n)}$-vertices given in Figure \ref{fig:verticalRwts} give a solution to the Yang-Baxter equation for any $n\geq 1$.
\end{theorem}

\begin{proof}
The proof is nearly identical to that of Theorem \ref{thm:horizontalYBE}, so we sketch the broad strokes and leave the details to the reader. Since the weights in Figure \ref{fig:columnRibbonWts} set type NE vertices to have weight 0 instead of type SE as before, the set of boundary conditions to check is slightly different. The six weight 0 conditions in the proof of Theorem \ref{thm:horizontalYBE} will be nonzero cases, which the reader may find in Appendix \ref{sec:rybeC} as B1,...,B4, B5, and B7. Cases A2, A3, A4, A9, A12, and A14 replace them as the weight 0 cases, and cases A6 and A8 both acquire an additional state on one side, which we reproduce as cases B6 and B8. From there, splitting into cases, we divide by number of states per side of the YBE and how many SW vertices appear, as SW is now the ``unusual" vertex in terms of power of $q$:
\begin{itemize}
    \item Case 1 (one state per side, no SW vertices): conditions A1, A7, A10, A11, A13, B1, B2, B3.
    \item Case 2 (one state per side, one SW vertex per side): condition B4.
    \item Case 3 (one state per side, two SW vertices per side): condition A5.
    \item Case 4 (two states on the left, one on the right): conditions B5, B6, which have the same weights.
    \item Case 5 (one state on the left, two on the right): conditions B7, B8, and B7 = $y_i\cdot$B8.
\end{itemize}
Checking the interaction of the block with the underlying YBE in each case, we verify that the YBE does indeed hold.
\end{proof}

\begin{theorem}\label{thm:mixedYBE}

Together with the weights in Figures \ref{fig:RibbonWts} and \ref{fig:columnRibbonWts}, the $HV$ Boltzmann weights for $R^{(n)}$-vertices given in Figure \ref{fig:verticalRwts} give a solution to the Yang-Baxter equation for any $n\geq 1$.
\end{theorem}

\begin{proof}
As in the vertical case, we abbreviate the proof, since it is very similar to that of Theorem \ref{thm:horizontalYBE}. However, it is important to note the order of the row types: we can only attach a type HV vertex if the horizontal and vertical strip rows are in the correct order, i.e. horizontal to connect to the ``$i$"-strand of the $R^{(n)}$-vertex and vertical to connect to the ``$j$"-strand of the vertex. With this in mind, our boundary conditions split into the following behaviors, as indexed in Appendices \ref{sec:rybe} and \ref{sec:rybeC}:
\begin{itemize}
\item Case 0 (weight 0): A2, A3, A4, B1, B2, B3.
\item Case 1 (one state per side, no unusual vertices): A1, A7, A9, A12, A14, B4, B5, B7.
\item Case 2 (one state per side, a horizontal NW): A10.
\item Case 3 (one state per side, a vertical SW): A5.
\item Case 4 (two states on the left, one on the right): conditions A11, B6, which have the same weights.
\item Case 5 (one state on the left, two on the right): conditions A13, B8, which have the same weights.
\end{itemize}
\end{proof}

\begin{prop}[\cite{lam-ribbon-heisenberg}, Prop 30]
The function $\mathcal{G}_{\lambda/\mu}(X/Y;q)$ is symmetric in each of $X$ and $Y$ and does not depend on the total order fixed between $A$ and $A'$.
\end{prop}

\begin{proof}
For a variable set $X = (x_1,...,x_m)$, $Y = (y_{1'},...,y_{r'})$, consider the total ordering $1 < 2 < \cdots < m < 1' < 2' \cdots < r'.$ To prove each of these symmetries, we will use train arguments with the Yang-Baxter equation for different row types proven in Theorems \ref{thm:horizontalYBE}, \ref{thm:verticalYBE}, and \ref{thm:mixedYBE}. For symmetry in $X$, consider the partition function of the following lattice model, where the $R^{(n)}$-vertex is attached to horizontal strip rows $i, i+1$:

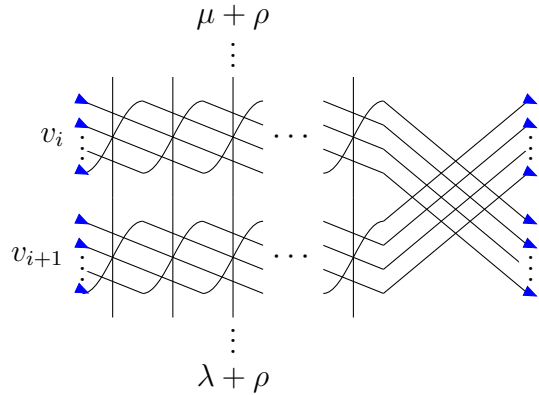
\begin{equation*}
	\tikz[baseline=.1ex,scale=0.8]{
		\spiderR{7.2}{0}{1}
		\ribbonvertex{1.5}{1}{1}
		\ribbonvertex{1.5}{-1}{1}
		\ribbonvertex{2.5}{1}{1}
		\ribbonvertex{2.5}{-1}{1}
		\ribbonvertex{3.5}{1}{1}
		\ribbonvertex{3.5}{-1}{1}
		\ribbonvertex{5.5}{1}{1}
		\ribbonvertex{5.5}{-1}{1}
		\node () at (4.5,1) {$\cdots$};
		\node () at (4.5,-1) {$\cdots$};
		\node () at (0.5,1) {$v_i$};
		\node () at (0.25,-1) {$v_{i+1}$};
		\node () at (3.5,3) {$\mu + \rho$};
		\node () at (3.5,-3) {$\lambda + \rho$};
		\node () at (3.5,2.5) {$\vdots$};
		\node () at (3.5,-2.25) {$\vdots$};
		
        \arrowrightcolor{1}{1.6}{blue}
        \arrowrightcolor{1}{1.2}{blue}
        \path[fill=white] (1,0.95) circle (.2);
        \node () at (1, 0.95) {\vdots};
        \arrowrightcolor{1}{0.425}{blue}
        
        \arrowrightcolor{1}{-0.4}{blue}
        \arrowrightcolor{1}{-0.8}{blue}
        \path[fill=white] (1,-1.05) circle (.2);
        \node () at (1, -1.05) {\vdots};
        \arrowrightcolor{1}{-1.55}{blue}
        
        \arrowrightcolor{8.45}{1.6}{blue}
        \arrowrightcolor{8.45}{1.2}{blue}
        \path[fill=white] (8.45,0.95) circle (.2);
        \node () at (8.45, 0.95) {\vdots};
        \arrowrightcolor{8.45}{0.4}{blue}
        
        \arrowrightcolor{8.45}{-0.4}{blue}
        \arrowrightcolor{8.45}{-0.8}{blue}
	\path[fill=white] (8.45,-1.05) circle (.2);
        \node () at (8.45, -1.05) {\vdots};
        \arrowrightcolor{8.45}{-1.6}{blue}}
\end{equation*}

Since all of the edges on the right boundary point right, the only choice for each $R^{(1)}$-vertex is type E, which gives this system a total weight of $x_i^n \cdot \mathcal{G}_{\lambda/\mu}^{(n)}(X/Y;q)$. As in Theorem \ref{thm:horizontalYBE}, push the diagonal vertex all the way to the right, column by column, until it emerges into the right boundary. This process gives the following system:

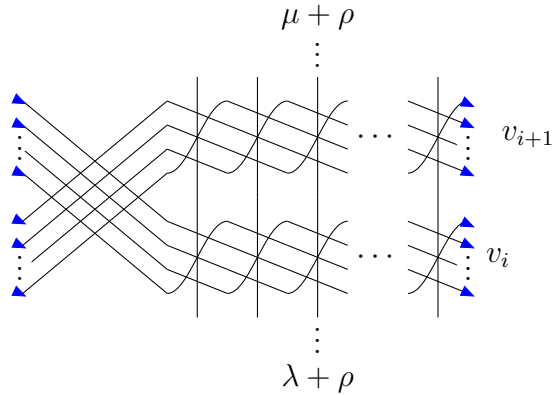
\begin{equation*}
	\tikz[baseline=.1ex,scale=0.8]{
		\spiderR{-0.2}{0}{1}
		\ribbonvertex{1.5}{1}{1}
		\ribbonvertex{1.5}{-1}{1}
		\ribbonvertex{2.5}{1}{1}
		\ribbonvertex{2.5}{-1}{1}
		\ribbonvertex{3.5}{1}{1}
		\ribbonvertex{3.5}{-1}{1}
		\ribbonvertex{5.5}{1}{1}
		\ribbonvertex{5.5}{-1}{1}
		\node () at (4.5,1) {$\cdots$};
		\node () at (4.5,-1) {$\cdots$};
		\node () at (6.5,-1) {$v_i$};
		\node () at (7,1) {$v_{i+1}$};
		\node () at (3.5,3) {$\mu + \rho$};
		\node () at (3.5,-3) {$\lambda + \rho$};
		\node () at (3.5,2.5) {$\vdots$};
		\node () at (3.5,-2.25) {$\vdots$};
		
        \arrowrightcolor{6}{1.55}{blue}
        \arrowrightcolor{6}{1.2}{blue}
        \path[fill=white] (6,0.95) circle (.2);
        \node () at (6, 0.95) {\vdots};
        \arrowrightcolor{6}{0.4}{blue}
        
        \arrowrightcolor{6}{-0.45}{blue}
        \arrowrightcolor{6}{-0.8}{blue}
        \path[fill=white] (6,-1.05) circle (.2);
        \node () at (6, -1.05) {\vdots};
        \arrowrightcolor{6}{-1.6}{blue}
        
        \arrowrightcolor{-1.45}{1.6}{blue}
        \arrowrightcolor{-1.45}{1.2}{blue}
        \path[fill=white] (-1.45,0.95) circle (.2);
        \node () at (-1.45, 0.95) {\vdots};
        \arrowrightcolor{-1.45}{0.4}{blue}
        
        \arrowrightcolor{-1.45}{-0.4}{blue}
        \arrowrightcolor{-1.45}{-0.8}{blue}
	\path[fill=white] (-1.45,-1.05) circle (.2);
        \node () at (-1.45, -1.05) {\vdots};
        \arrowrightcolor{-1.45}{-1.6}{blue}}
\end{equation*}

Similarly, the $R^{(n)}$-vertex in this system must be $n$ copies of the type $E$ vertex, so this side has weight $\mathcal{G}_{\lambda/\mu}^{(n)}(s_iX/Y;q) \cdot x_i^n$, where $s_iX = (x_1,...,x_{i+1},x_i,...,x_m)$. By Theorem \ref{thm:horizontalYBE}, these weights are equal, so $\mathcal{G}_{\lambda/\mu}^{(n)}(X/Y;q)$ is symmetric in the $X$ variables. The proof of symmetry in the $Y$ variables is nearly identical, excepting only that we attach the $R^{(n)}$-vertex to rows $i'$ and $(i+1)'$. To prove total order, we use a similar train argument to move horizontal strip rows gradually below vertical strip rows; since any total order respects the internal orderings on $\{1,...,m\}$ and $\{1',...,r'\}$, it is possible to achieve any total ordering desired using only the mixed-type $R^{(n)}$-vertices (i.e. type HV), and the resulting braid of all $E$ vertices will have the same weight on either side and thus cancel off as in previous cases.
\end{proof}

\section{Branching Rules}

 The structure of the ribbon lattice gives rise to combinatorial proofs of super LLT polynomial identities. For example, splitting along any row and computing the partition function of a grid in two different ways, we obtain branching rules for the super LLT polynomials.

\begin{prop} Given alphabets $A, A'$ and a total order, let $z_i$ stand in for the $x$ or $y$ variable at position $i$ in the total order. Let $r=|A| + |A'|$ and fix some $k=1,...., r$. Then the super LLT polynomials satisfy the following general branching rule:
\begin{align*}
    \mathcal{G}^{(n)}_{\lambda/\mu}(z_1,\ldots, z_r;q) &= \sum_{\gamma} \mathcal{G}^{(n)}_{\lambda/\gamma}(z_{k+1},\cdots, z_r;q)\cdot \mathcal{G}^{(n)}_{\g/\mu}(z_1,\cdots, z_k;q).
\end{align*}
where the sum runs over all partitions $\gamma$. Note that the only nonzero terms in the sum will be those $\gamma$ for which both $\lambda/\gamma$ and $\gamma/\mu$ admit super $n$-ribbon tableaux in the respective subsets of the total order.

Specifically, we may describe these $\gamma$ as partitions for which there exists a sequence of compositions $\gamma_0 = \mu+\rho, \gamma_1,..., \gamma_k = \gamma+\rho,\gamma_{k+1},...,\gamma_r=\lambda+\rho$ such that:
\begin{enumerate}
    \item $\gamma_i - \gamma_{i-1} \in nZ^{\ell(\lambda)}$ for all $i$.
    \item If row $i$ is a horizontal strip row, $(\gamma_i)_j \neq (\gamma_{i-1})_k$ for all $j\neq k$.
    \item If row $i$ is a vertical strip row, $(\gamma_i - \gamma_{i-1})_j \in\{0,n\}$ for all $j$.
\end{enumerate}
\end{prop}

\begin{proof}
Consider the boundary conditions $\mathcal{B}_{\lambda/\mu}(X/Y)$ and let $z_i$ denote the spectral parameter on row $i$. Slicing the diagram horizontally along the vertical edges between rows $k$ and $k+1$, consider all choices of labels for the sliced edges, which will each be of the form $\gamma+\rho$ for some partition $\gamma$. 

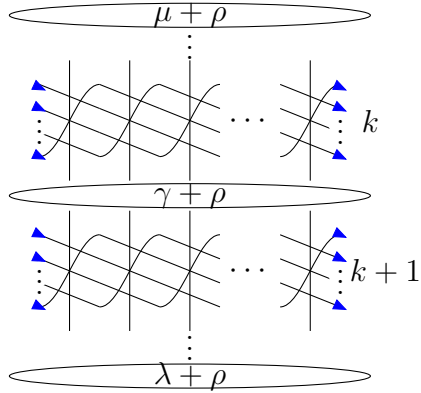
\begin{center}
\begin{tikzpicture}[scale = 0.8]
		\ribbonvertex{1.5}{1.25}{1}
		\ribbonvertex{1.5}{-1.25}{1}
		\ribbonvertex{2.5}{1.25}{1}
		\ribbonvertex{2.5}{-1.25}{1}
		\ribbonvertex{3.5}{1.25}{1}
		\ribbonvertex{3.5}{-1.25}{1}
		\ribbonvertex{5.5}{1.25}{1}
		\ribbonvertex{5.5}{-1.25}{1}
		\draw (3.5,3) ellipse (3cm and 0.2cm);
		\draw (3.5,0) ellipse (3cm and 0.2cm);
		\draw (3.5,-3) ellipse (3cm and 0.2cm);		
		\node () at (4.5,1.25) {$\cdots$};
		\node () at (4.5,-1.25) {$\cdots$};
		\node () at (6.5,1.25) {$k$};
		\node () at (6.75,-1.25) {$k+1$};
		\node () at (3.5,3) {$\mu + \rho$};
		\node () at (3.5,0) {$\gamma + \rho$};
		\node () at (3.5,-3) {$\lambda + \rho$};
		\node () at (3.5,-2.4) {$\vdots$};
		\node () at (3.5,2.6) {$\vdots$};
		
        \arrowrightcolor{6}{1.8}{blue}
        \arrowrightcolor{6}{1.45}{blue}
        \path[fill=white] (6,1.2) circle (.2);
        \node () at (6, 1.2) {\vdots};
        \arrowrightcolor{6}{0.65}{blue}
        
        \arrowrightcolor{6}{-0.65}{blue}
        \arrowrightcolor{6}{-1.05}{blue}
        \path[fill=white] (6,-1.3) circle (.2);
        \node () at (6, -1.3) {\vdots};
        \arrowrightcolor{6}{-1.85}{blue}
        
         \arrowrightcolor{1}{1.8}{blue}
        \arrowrightcolor{1}{1.45}{blue}
        \path[fill=white] (1,1.2) circle (.2);
        \node () at (1, 1.2) {\vdots};
        \arrowrightcolor{1}{0.65}{blue}
        
        \arrowrightcolor{1}{-0.65}{blue}
        \arrowrightcolor{1}{-1.05}{blue}
        \path[fill=white] (1,-1.3) circle (.2);
        \node () at (1, -1.3) {\vdots};
        \arrowrightcolor{1}{-1.85}{blue}
\end{tikzpicture}
\end{center}

Then rows $1,...,k$ will have boundary conditions $\mathcal{B}_{\gamma/\mu}$ and rows $k+1,...,r$ will have boundary conditions $\mathcal{B}_{\lambda/\mu}$. If either Note that the only nonzero terms in this sum will be those for $\gamma$ such that $\gamma/\mu, \lambda/\gamma$ are both tileable with $n$-ribbons such that ribbons labelled $i$ form a horizontal strip if $z_i = x_i$ and a vertical strip if $z_i = y_i$.

Using the particle description of the model, the composition $\gamma_i$ labels the columns between row $i-1$ and row $i$ containing a particle. The first condition follows from the fact that particles can only move along a row in multiples of $n$ vertices. Within this condition, horizontal rows require that the same label only appear in $\gamma_{i-1}$ and $\gamma_i$ if it comes from a particle passing straight up through the row, i.e. a type SW vertex as opposed to a type SE. On the other hand, vertical rows require that each particle must either pass straight through or travel one single step of $n$ particles, since the type NE vertex which allows a particle to travel multiple sets of $n$ particles has weight 0.
\end{proof}

\section{Implications of the Cauchy Identity for the Lattice Model}\label{sec:cauchyish}


A slight adaptation of the super LLT lattice model, obtained by reversing the direction of travel of the particles from left to right, allows us to give a combined \emph{Cauchy model}, suggestively named to relate to the Cauchy identity proved in Section \ref{sec:heisenberg}. In the manner of similar Cauchy identities for similar polynomials, we will stack the new adapted model together with the original model.

We first present the adapted model, whose partition function also gives the super LLT polynomials.

\begin{figure}[H]
\centering 
 \begin{tabular}{| M{1.1cm} |M{1.47cm} | M{1.47cm} | M{1.47cm} | M{1.47cm} | M{1.47cm} | M{1.47cm} |} 
 \hline
Label & SW & NS & SE & NW & EW & NE \\
 \hline 
Vertex&\qquad
\begin{tikzpicture}[scale=0.8]
	\goodlookingvertex{-4.5}{1.5}{3}
	\arrowup{-4.5}{-2.5+3}
		\arrowright{-4.5+0.67}{1.5+0.64}
		\arrowup{-4.5}{-0.5+3}
		\arrowright{-4.5-0.67}{1.5-0.64}
		
		\draw [blue](-4.5+0.67,1.5-0.16) ellipse (0.22cm and 0.64cm);
\end{tikzpicture}     &

\qquad \begin{tikzpicture}[scale=0.8]
\goodlookingvertex{-1.5}{1.5}{1}
	\arrowdown{-1.5}{2.5}
		\arrowup{-1.5}{0.5}
		\arrowright{-1.5+0.67}{1.5+0.64}
		\arrowleft{-1.5-0.67}{1.5-0.64}
		\draw [blue](-1.5+0.67,1.5-0.16) ellipse (0.22cm and 0.64cm);
\end{tikzpicture}  & 

\qquad \begin{tikzpicture}[scale=0.8]
\goodlookingvertex{-1.5}{1.5}{1}
	\arrowup{-1.5}{2.5}
		\arrowup{-1.5}{0.5}
		\arrowleft{-1.5+0.67}{1.5+0.64}
		\arrowleft{-1.5-0.67}{1.5-0.64}
\end{tikzpicture} & 
\qquad \begin{tikzpicture}[scale=0.8]
\goodlookingvertex{-1.5}{1.5}{1}
	\arrowdown{-1.5}{2.5}
		\arrowdown{-1.5}{0.5}
		\arrowright{-1.5+0.67}{1.5+0.64}
		\arrowright{-1.5-0.67}{1.5-0.64}
\end{tikzpicture} & 
\qquad \begin{tikzpicture}[scale=0.8]
\goodlookingvertex{1.5}{1.5}{2}
	\arrowup{1.5}{2.5}
		\arrowleft{1.5+0.67}{1.5+0.64}
		\arrowdown{1.5}{0.5}
		\arrowright{1.5-0.67}{1.5-0.64}
		\draw [blue](1.5+0.67,1.5-0.16) ellipse (0.22cm and 0.64cm);
\end{tikzpicture} & 
\qquad \begin{tikzpicture}[scale=0.8]
\goodlookingvertex{4.5}{1.5}{2}
\arrowdown{4.5}{1.5+1}
		\arrowleft{4.5+0.67}{1.5+0.64}
		\arrowdown{4.5}{1.5-1}
		\arrowleft{4.5-0.67}{1.5-0.64}
	\draw [blue](4.5+0.67,1.5-0.16) ellipse (0.22cm and 0.64cm);
\end{tikzpicture}
\\

\hline
$wt_{\widetilde{H}}$ & $0$ & $q^tx_i$ & $q^t$ & $q^tx_i$ & $q^t$ & $1$  \\
\hline
$wt_{\widetilde{V}}$ & $-y_i$ & $-y_i$ & $q^t$ & 0 & $1$ & $1$\\
\hline

 \end{tabular}
 \caption{The weights of vertices lying in row $i$ for the alternate $n$ super ribbon lattice model, where $t$ is the number of \emph{right} arrows in the blue circled area.}
 \label{fig:CauchyWts}
\end{figure} 

\begin{defn}
Given a total ordering on alphabets $A, A'$, a skew partition $\lambda/\mu$, and a positive integer $r$, let $\widetilde{\mathcal{B}}_{\lambda/\mu}(X/Y)$ be family of lattice models with weights determined in Figure \ref{fig:CauchyWts} and boundary conditions:
\begin{itemize}
\item edges on the left and right boundaries are labelled $<$,
\item edges on the top boundary are labelled $\wedge$ on parts of $\lambda + \rho$  and $\vee$ else (numbering columns left to right starting with 1 as before),
\item edges on the bottom boundary are labelled $\wedge$ on parts of $\mu + \rho$ and $\vee$ else,
    \item \emph{horizontal} strip rows labelled $i\in A$ have spectral parameters $x_i$,
    \item \emph{vertical} strip rows labelled $i' \in A'$ have spectral parameters $y_{i'}$, and
    \item reading from top to bottom, rows are labelled in increasing order according to the total order.
\end{itemize}
\end{defn}

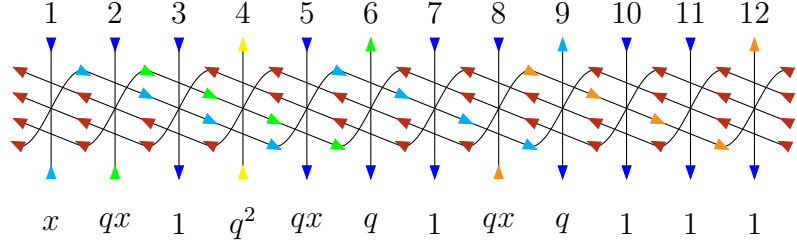
\begin{figure}
\centering
\begin{tikzpicture}[scale=0.85]
    \foreach \x in {0,1,...,11}{
\ribbonvertex{\x}{0}{\x}
}

\arrowdowncolor{0}{1}{blue}
\arrowdowncolor{1}{1}{blue}
\arrowdowncolor{2}{1}{blue}
\arrowupcolor{3}{1}{yellow}
\arrowdowncolor{4}{1}{blue}
\arrowupcolor{5}{1}{green}
\arrowdowncolor{6}{1}{blue}
\arrowdowncolor{7}{1}{blue}
\arrowupcolor{8}{1}{cyan}
\arrowdowncolor{9}{1}{blue}
\arrowdowncolor{10}{1}{blue}
\arrowupcolor{11}{1}{BurntOrange}

\arrowupcolor{0}{-1}{cyan}
\arrowupcolor{1}{-1}{green}
\arrowdowncolor{2}{-1}{blue}
\arrowupcolor{3}{-1}{yellow}
\arrowdowncolor{4}{-1}{blue}
\arrowdowncolor{5}{-1}{blue}
\arrowdowncolor{6}{-1}{blue}
\arrowupcolor{7}{-1}{BurntOrange}
\arrowdowncolor{8}{-1}{blue}
\arrowdowncolor{9}{-1}{blue}
\arrowdowncolor{10}{-1}{blue}
\arrowdowncolor{11}{-1}{blue}

\arrowleftcolor{0-0.5}{0+3/5}{BrickRed}
\arrowleftcolor{0-0.5}{0+1/5}{BrickRed}
\arrowleftcolor{0-0.5}{0-1/5}{BrickRed}
\arrowleftcolor{0-0.5-0.03}{0-3/5+0.04}{BrickRed}

\arrowrightcolor{1-0.5+0.02}{0+3/5-0.04}{cyan}
\arrowleftcolor{1-0.5}{0+1/5}{BrickRed}
\arrowleftcolor{1-0.5}{0-1/5}{BrickRed}
\arrowleftcolor{1-0.5-0.02}{0-3/5+0.04}{BrickRed}

\arrowrightcolor{2-0.5+0.02}{0+3/5-0.04}{green}
\arrowrightcolor{2-0.5}{0+1/5}{cyan}
\arrowleftcolor{2-0.5}{0-1/5}{BrickRed}
\arrowleftcolor{2-0.5-0.02}{0-3/5+0.04}{BrickRed}

\arrowleftcolor{3-0.5}{0+3/5}{BrickRed}
\arrowrightcolor{3-0.5}{0+1/5}{green}
\arrowrightcolor{3-0.5}{0-1/5}{cyan}
\arrowleftcolor{3-0.5-0.02}{0-3/5+0.04}{BrickRed}

\arrowleftcolor{4-0.5}{0+3/5}{BrickRed}
\arrowleftcolor{4-0.5}{0+1/5}{BrickRed}
\arrowrightcolor{4-0.5}{0-1/5}{green}
\arrowrightcolor{4-0.5}{0-3/5}{cyan}

\arrowrightcolor{5-0.5+0.02}{0+3/5-0.04}{cyan}
\arrowleftcolor{5-0.5}{0+1/5}{BrickRed}
\arrowleftcolor{5-0.5}{0-1/5}{BrickRed}
\arrowrightcolor{5-0.5}{0-3/5}{green}

\arrowleftcolor{6-0.5}{0+3/5}{BrickRed}
\arrowrightcolor{6-0.5}{0+1/5}{cyan}
\arrowleftcolor{6-0.5}{0-1/5}{BrickRed}
\arrowleftcolor{6-0.5-0.02}{0-3/5+0.04}{BrickRed}

\arrowleftcolor{7-0.5}{0+3/5}{BrickRed}
\arrowleftcolor{7-0.5}{0+1/5}{BrickRed}
\arrowrightcolor{7-0.5}{0-1/5}{cyan}
\arrowleftcolor{7-0.5-0.02}{0-3/5+0.04}{BrickRed}

\arrowrightcolor{8-0.5+0.02}{0+3/5-0.04}{BurntOrange}
\arrowleftcolor{8-0.5}{0+1/5}{BrickRed}
\arrowleftcolor{8-0.5}{0-1/5}{BrickRed}
\arrowrightcolor{8-0.5}{0-3/5}{cyan}

\arrowleftcolor{9-0.5}{0+3/5}{BrickRed}
\arrowrightcolor{9-0.5}{0+1/5}{BurntOrange}
\arrowleftcolor{9-0.5}{0-1/5}{BrickRed}
\arrowleftcolor{9-0.5-0.02}{0-3/5+0.04}{BrickRed}

\arrowleftcolor{10-0.5}{0+3/5}{BrickRed}
\arrowleftcolor{10-0.5}{0+1/5}{BrickRed}
\arrowrightcolor{10-0.5}{0-1/5}{BurntOrange}
\arrowleftcolor{10-0.5-0.02}{0-3/5+0.04}{BrickRed}

\arrowleftcolor{11-0.5}{0+3/5}{BrickRed}
\arrowleftcolor{11-0.5}{0+1/5}{BrickRed}
\arrowleftcolor{11-0.5}{0-1/5}{BrickRed}
\arrowrightcolor{11-0.5}{0-3/5}{BurntOrange}

\arrowleftcolor{12-0.5}{0+3/5}{BrickRed}
\arrowleftcolor{12-0.5}{0+1/5}{BrickRed}
\arrowleftcolor{12-0.5}{0-1/5}{BrickRed}
\arrowleftcolor{12-0.5-0.03}{0-3/5+0.04}{BrickRed}

\node () at (0,-1.8) {$x$};
\node () at (1,-1.8) {$qx$};
\node () at (2,-1.8) {$1$};
\node () at (3,-1.8) {$q^2$};
\node () at (4,-1.8) {$qx$};
\node () at (5,-1.8) {$q$};
\node () at (6,-1.8) {$1$};
\node () at (7,-1.8) {$qx$};
\node () at (8,-1.8) {$q$};
\node () at (9,-1.8) {$1$};
\node () at (10,-1.8) {$1$};
\node () at (11,-1.8) {$1$};

\node () at (0,1.5) {$1$};
\node () at (1,1.5) {$2$};
\node () at (2,1.5) {$3$};
\node () at (3,1.5) {$4$};
\node () at (4,1.5) {$5$};
\node () at (5,1.5) {$6$};
\node () at (6,1.5) {$7$};
\node () at (7,1.5) {$8$};
\node () at (8,1.5) {$9$};
\node () at (9,1.5) {$10$};
\node () at (10,1.5) {$11$};
\node () at (11,1.5) {$12$};
\end{tikzpicture}
\caption{The single horizontal strip row lattice state for $\mu = (4,1), \lambda = (8,6,4,3)$, and $n=4$ in the alternate model. Note that the weight is $q^7x^4$, the same as that of the state in Figure \ref{onerowlattice}.}
\label{onerowlatticeAlt}
\end{figure}

\begin{theorem}\label{altmainthm}
Given a skew partition $\lambda/\mu$ and a total order on alphabets $A,A'$, we have
 \[Z\left(\widetilde{\mathcal{B}}_{\l/\mu}(X/Y)\right)=\G_{\l/\mu}^{(n)}(X/Y;q).\] 
\end{theorem}

The proof of Theorem \ref{altmainthm} is entirely analogous to that of Theorem \ref{mainthm}, viewing particles in the system as travelling up and right as opposed to up and left. It is perhaps helpful to think of this alternate model as adding $n$ ribbon strips onto $\mu$ to reach $\lambda$, whereas our original model removed $n$ ribbon strips from $\lambda$ to obtain $\mu$.

Notably, this interpretation allows us to stack our original model atop or below this new one and evaluate the results to give the sum sides of the generic Cauchy identity for super LLT polynomials. If the original model is all horizontal strip rows and the alternate all vertical strip rows (or vice versa), only finitely many columns are necessary to account for all possible partitions that could appear on the boundary between original and alternate models. However, if there is a mix of row types, we need to account for the fact that horizontal strip rows may be arbitrarily long (i.e., their skew partitions have arbitrarily large parts) and vertical strip rows may be arbitrarily tall (i.e., their partitions have arbitrarily many parts). We do this by taking inspiration from the Fock space: given infinitely many columns indexed in $\Z$ and a partition $\mu$, place particles (up arrows) on parts of $\mu + \rho$ as well as on non-positive columns. This method produces an infinite sea of particles to the left of the columns involved in $\mu + \rho$ and an infinite void of holes to the right, mimicking the particle-hole interaction of Hamiltonian operators on the Fock space.

\begin{defn}
Consider four ordered alphabets $A,A',B,B'$ together with total orderings $\prec_A$ on $A\cup A'$ and $\prec_B$ on $B\cup B'$ that respect the orderings on $A,A'$ and $B,B'$ respectively. A \emph{overarching total order} $I$ is a total ordering on $A,A',B,B'$ that respects the individual orderings on each alphabet as well as the intermediate total orderings $\prec_A$ and $\prec_B$. 

For example, let $I_A$ be the overarching total order where $a_i < b_j$ for any $a_i \in A \cup A', b_j \in B\cup B'$, and $I_B$ be the overarching total order where $a_i > b_j$ for all $a_i \in A \cup A', b_j \in B\cup B'$.
\end{defn}


\begin{defn}
Consider two pairs of totally ordered alphabets $A/A'$ and $B/B'$, with spectral parameters $X/Y$ and $W/Z$, respectively. Let $\mu, \nu$ be partitions, padded with zero parts to have the same length, and define an overarching total order $I$ on $A,A',B,B'$ that respects all existing orderings. We define the \emph{Cauchy system} $\mathcal{C}_{\mu,\nu, I}(X/Y, W/Z)$ to have the following boundary conditions and spectral parameters on a semi-infinite lattice:
\begin{itemize}
\item columns are labelled left to right with indices in $\Z$, and designating a fixed column 1,
\item edges on the top boundary are labelled by $\mu+\rho$ in the manner described above,
\item edges on the bottom boundary are labelled by $\nu+\rho$,
\item rows are labelled in increasing order from top to bottom according to the overarching total order $I$,
\item horizontal strip rows labelled by $A$ (respectively $B$) take weights $wt_H$ (respectively $wt_{\widetilde{H}}$) and spectral parameters $x_i$ (respectively $w_i$), and
\item vertical strip rows labelled by $A'$ (respectively $B'$) take weights $wt_V$ (respectively $wt_{\widetilde{V}}$) and spectral parameters $y_{i'}$ (respectively $z_{i'}$).
\item side boundary edges on $A/A'$ rows are labelled $>$ and those on $B/B'$ rows are labelled $<$.
\end{itemize}
\end{defn}

Here, we consider the spectral parameters as formal variables in order to remove issues of convergence of the partition function, since the semi-infinite model will produce infinitely many states for some choices of boundary conditions on the Cauchy model.

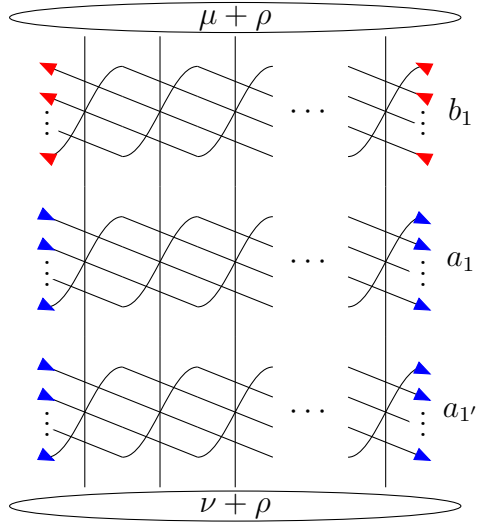
\begin{figure}
    \centering
    \begin{tikzpicture}
		\ribbonvertex{1.5}{1}{1}
		\ribbonvertex{1.5}{-1}{1}
		\ribbonvertex{2.5}{1}{1}
		\ribbonvertex{2.5}{-1}{1}
		\ribbonvertex{3.5}{1}{1}
		\ribbonvertex{3.5}{-1}{1}
		\ribbonvertex{5.5}{1}{1}
		\ribbonvertex{5.5}{-1}{1}
		\ribbonvertex{1.5}{3}{1}
		\ribbonvertex{2.5}{3}{1}
		\ribbonvertex{3.5}{3}{1}
		\ribbonvertex{5.5}{3}{1}
		\draw (3.5,4.25) ellipse (3cm and 0.2cm);
		\draw (3.5,-2.25) ellipse (3cm and 0.2cm);		
		\node () at (4.5,1) {$\cdots$};
		\node () at (4.5,-1) {$\cdots$};
		\node () at (4.5,3) {$\cdots$};

		\node () at (6.5,1) {$a_1$};
		\node () at (6.5,-1) {$a_{1'}$};
		\node () at (6.5,3) {$b_1$};
		\node () at (3.5,4.25) {$\mu + \rho$};
		\node () at (3.5,-2.25) {$\nu + \rho$};
		
        \arrowrightcolor{6}{1.55}{blue}
        \arrowrightcolor{6}{1.2}{blue}
        \path[fill=white] (6,0.95) circle (.2);
        \node () at (6, 0.95) {\vdots};
        \arrowrightcolor{6}{0.4}{blue}
        
        \arrowrightcolor{6}{-0.45}{blue}
        \arrowrightcolor{6}{-0.8}{blue}
        \path[fill=white] (6,-1.05) circle (.2);
        \node () at (6, -1.05) {\vdots};
        \arrowrightcolor{6}{-1.6}{blue}
        
         \arrowrightcolor{1}{1.6}{blue}
        \arrowrightcolor{1}{1.2}{blue}
        \path[fill=white] (1,0.95) circle (.2);
        \node () at (1, 0.95) {\vdots};
        \arrowrightcolor{1}{0.4}{blue}
        
        \arrowrightcolor{1}{-0.4}{blue}
        \arrowrightcolor{1}{-0.8}{blue}
        \path[fill=white] (1,-1.05) circle (.2);
        \node () at (1, -1.05) {\vdots};
        \arrowrightcolor{1}{-1.6}{blue}
        
        \arrowleftcolor{1}{3.6}{red}
        \arrowleftcolor{1}{3.2}{red}
        \path[fill=white] (1,2.75) ellipse (.15 and .15);
        \node () at (1, 2.95) {\vdots};
        \arrowleftcolor{1}{2.4}{red}

        \arrowleftcolor{6}{3.6}{red}
        \arrowleftcolor{6}{3.2}{red}
        \path[fill=white] (6,2.75) ellipse (.15 and .15);
        \node () at (6, 2.95) {\vdots};
        \arrowleftcolor{6}{2.4}{red}
\end{tikzpicture}
    \caption{An example of boundary conditions for the Cauchy lattice model, under the overarching order $b_1 < a_1 < a_{1'}$ for $A = \{a_1\},A' = \{a_{i'}\}, B = \{b_1\}$.}
    \label{fig:cauchybdry}
\end{figure}

\begin{prop}\label{cauchysides}
If we choose the overarching order $I_A$, then 
\[ Z(\mathcal{C}_{\mu,\nu,I_A}(X/Y, W/Z)) = \sum_{\lambda} \mathcal{G}_{\lambda/\mu}(X/Y;q) \cdot \mathcal{G}_{\lambda/\nu}(W/Z;q),\]
where the sum runs over all partitions $\lambda$.
If instead we choose the order $I_B$, then
\[ Z(\mathcal{C}_{\mu,\nu,I_B}(X/Y, W/Z)) = \sum_{\lambda} \mathcal{G}_{\mu/\lambda}(X/Y;q) \cdot \mathcal{G}_{\nu/\lambda}(W/Z;q),\]
where the sum runs over all $\lambda$.
\end{prop}

\begin{proof}
For the first case, the overarching order places the $A$ and $A'$ rows above the $B$ and $B'$ rows. Slicing along the lattice model in between these two pieces, the cut edge is labelled by $\lambda + \rho$ for some partition $\lambda$. Note: it is possible that $\lambda$ has more parts than $\mu$ or $\nu$, in which case some of the particles travelling to fill parts of $\lambda + \rho$ will be coming from the $-\infty,...,0$ columns.  Since the partition function does not depend on the column numbers at all, shifting the designated ``1" column left to accommodate all travelling particles results in the same partition function and is equivalent to padding $\mu,\nu$ with zero parts to have the same number of parts as $\lambda$. The top half then has boundary conditions $B_{\lambda/\mu}(X/Y)$ and gives $\mathcal{G}_{\lambda/\mu}(X/Y;q)$, whereas the bottom half has boundary conditions $\widetilde{B}_{\lambda/\nu}(W/Z)$ and gives $\mathcal{G}_{\lambda/\nu}(W/Z)$. The second case follows similarly, with the additional note that any partition $\lambda$ such that $\mu/\lambda $ and $\nu/\lambda$ are both skew shapes will have at most the same number of parts as $\mu,\nu$. If $\lambda \not\subset \mu$ or $\lambda\not\subset\nu$ then one of the super LLT polynomials in the sum will be zero, i.e. there will be no non-zero filling of the lattice model, so we may sum over all partitions $\lambda$.
\end{proof}

\begin{theorem}\label{CauchymixedYBE}
Along with the $R$-vertex weights in Figure \ref{fig:CauchyRwts} and Appendix \ref{restofCauchyRwts}, the $H,\widetilde{H}, V$ and $\widetilde{V}$ weights satisfy mixed YBEs in all combinations.
\end{theorem}

\begin{proof}
These proofs are analogous to those of Theorems \ref{thm:horizontalYBE}, \ref{thm:verticalYBE}, and \ref{thm:mixedYBE} so we omit them for length. 
\end{proof}

As in the Yang-Baxter equation for the original super model, similar factors of $q$ appear on many of these vertices: these factors adjust for the fact that in each of our models, there is one vertex with a different spin than the others. For an $R^{(1)}$ vertex $r_k$, set 
\begin{align*}
\tau &= \#\{r_t =SS | t>k\} + \#\{r_t = NN | t<k\}, \text{ i.e.}\\ &=\# SS \text{ below} + \# NN \text{ above }, \text{ and }\\
\kappa &= \#\{r_t =SS | t<k\} + \#\{r_t = NN | t>k\}, \text{ i.e.}\\
&=\# SS \text{ above} + \# NN \text{ below }.
\end{align*} 
If we equidistributed the spins across the Boltzmann weights (for instance, so all vertices in the $\widetilde{H}$ model garnered a power of $q^t$, including type NE), this power appear in every piece of the  $R^{(n)}$ vertex as $q^{c\tau}$ for Cauchy model YBEs of type $\widetilde{A}B$ and as $q^{c\kappa}$ for those of type $A\widetilde{B}$, where $c\in \{0,1,2\}$ is the number of horizontal strip rows of either model interacting in the YBE. However, this modification would remove all interesting information coming from the tableau spin statistic in the individual models, so we will simply note that it is unsurprising to see this factor show up on so many of our vertices. Also note that this is nearly the same as the power $\sigma$ appearing in Section \ref{sec:solvability}, as  $\sigma = \tau + \#\{r_t = W\}$, and as in that section, each Cauchy model YBE contains one vertex type with binomial weight. 

\emph{However}, in the case of the Cauchy model YBEs, rather than the standard fusion of multiplying the $R^{(1)}$ weights together to obtain the $R^{(n)}$ weight, as described in Section \ref{sec:solvability}, a curious ``pre-fusion" of two vertices occurs. For Cauchy model YBEs of type $\widetilde{A}B$, if the $R^{(n)}$ vertex contains at least one type $NN$ vertex appearing above at least one type $SS$ vertex, the first instances of each of these vertices combine into a ``type NN/SS" vertex, whose weight is not merely wt(NN)wt(SS), while any remaining NN and SS vertices retain their usual weight. Fusion into the $R^{(n)}$ vertex weight then proceeds by multiplication, using this new vertex. If this condition does not occur, for instance if all type SS vertices appear above any type NN vertex, the normal fusion process occurs.
Similarly, for Cauchy model YBEs of type $A\widetilde{B}$, this ``pre-fusion" combines the first instance of a type SS vertex with the first NN vertex below it to create ``type SS/NN" vertex.

\begin{rmk}
Unlike the R-vertices considered in Section \ref{sec:solvability}, whose weights arose from fusion on tensor products of quantum group modules, this set of mixed R-vertices is fascinating because \emph{there is no currently known quantum algebraic object that corresponds to this sort of ``pre-fusion" procedure.}
\end{rmk}

\begin{figure}[h]
\begin{center}
\begin{tabular}{|c|c|c|}
\hline \multicolumn{3}{|c|}{Type $\widetilde{V}H$ weights}\\
\hline Label & $r_k$ & Weight\\
\hline W& \scalebox{0.7}{\west} &$(1-q^{2 \# W \text{ above}}xy)q^{\tau}$\\
\hline S& \scalebox{0.7}{\south} & $1$\\
\hline SS*& \scalebox{0.7}{\southsouth} & $-xy\cdot q^{\# SS \text{ below } +\#W}$\\
\hline NN*& \scalebox{0.7}{\northnorth} &$q^{\# NN\text{ above } + \#W}$\\
\hline N& \scalebox{0.7}{\north} & $-q^{n-1-\#S}$\\
\hline E& \scalebox{0.7}{\east} & $q^{\tau + \#W}$\\
\hline\hline$\displaystyle{\frac{NN}{SS}}$& & $q^{3\tau} - q^{3\tau + 2\#W}xy - q^{\tau}$\\\hline
\end{tabular}
\end{center}
\caption{The weights for the $k$-th $R^{(1)}$-vertex in an $R^{(n)}$-vertex swapping rows $i,j$, where tilde vertical strip row $i$ is in the alternate model and horizontal strip row $j$ is in the original model, as well as the ``pre-fusion" weight of the NN/SS vertex, if one arises.}\label{fig:CauchyRwts}
\end{figure}

In the spirit of previous lattice model proofs for Cauchy identities, such as that for factorial Schur functions given in \cite{BMN}, we attach a braid of $R^{(n)}$ vertices to one side of the Cauchy model, swapping all the rows of the alternate model past those of the original model (see Figure \ref{cauchycartoon}). We may also consider the partition function obtained by affixing this braid on the other side of the model: in a finite model, these partition functions would be equal due to the Yang-Baxter equation. Here, however, on each side, we need to consider whether we may truncate to a finite model with the same partition function by removing infinite swathes of columns with weight one. It would also be interesting to consider whether there is a generalization of the ``infinite source" style of model used in \cite{BMP} that would apply to superLLT polynomials, in effect collecting all non-positive columns into one column on which infinitely many particles are allowed to travel.

The situation is further complicated by the fact that the two types of R-vertices that ``pre-fuse" (types SS and NN) do appear in the Cauchy braids for either side if a particle is travelling through a sufficiently large braid. If the $R^{(n)}$ vertices fused completely, it would be possible to compute the partition function on each of the $n$ copies of the braid individually and multiply the results together. However, since this is not the case, calculating the partition function for the braid involves careful bookkeeping of all possible combinations of vertex types on each copy of the braid.

\begin{lemma}\label{particlesinbraid}
Thinking of a Cauchy braid of mixed YBEs of either type $\widetilde{A}B$ or type $A\widetilde{B}$ in terms of particles passing through the lattice from bottom to top, the number of particles in the system must remain constant. That is, for the $n$-strand Cauchy braid given by
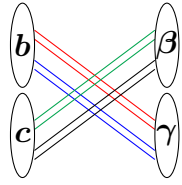
\[\begin{tikzpicture}[scale=0.8]
\draw [red](2,3.25) -- (4,1.75);
\draw[red] (2,3.1) -- (4,1.6);
\draw [blue](2,2.75) -- (4,1.25);
\draw[blue] (2,2.6) -- (4,1.1);
\draw[Green](2,1.75) -- (4,3.25);
\draw[Green] (2,1.6) -- (4,3.1);
\draw (2,1.25) -- (4,2.75);
\draw (2,1.1) -- (4,2.6);
\draw (4.2,3) ellipse (0.2cm and 0.7cm);		
\node () at (4.2,3) {$\boldsymbol{\beta}$};
\draw (4.2,1.5) ellipse (0.2cm and 0.7cm);		
\node () at (4.2,1.5) {$\boldsymbol{\gamma}$};
\draw (1.8,3) ellipse (0.2cm and 0.7cm);		
\node () at (1.8,3) {$\boldsymbol{b}$};
\draw (1.8,1.5) ellipse (0.2cm and 0.7cm);		
\node () at (1.8,1.5) {$\boldsymbol{c}$};
\end{tikzpicture}\]
we must have $\#\{\text{particles entering from }\boldsymbol{c} \text{ or }\boldsymbol{\gamma}\} = \#\{\text{particles exiting from }\boldsymbol{b} \text{ or }\boldsymbol{\beta}\}$.
\end{lemma}

\begin{proof}
Considering first the case of type $\widetilde{A}B$, color the six diagonal vertices with particles according to the rule established for rows: particles travel along left arrows on original model (i.e. $B$) strands and right arrows on alternate model (i.e. $\widetilde{A}$) strands.
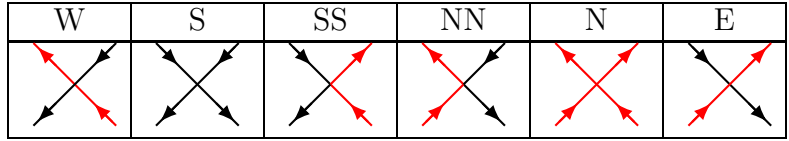
\[
\begin{tabular}{|c|c|c|c|c|c|}
\hline W&S&SS&NN&N&E\\ 
\hline \scalebox{1}{\westp} &\scalebox{1}{\southp}& \scalebox{1}{\southsouthp} &  \scalebox{1}{\northnorthp} & \scalebox{1}{\northp} & \scalebox{1}{\eastp}\\
\hline
\end{tabular}
\]
Since each individual vertex obeys the desired condition, a Cauchy braid built out of admissible vertices will as well. The statement follows similarly for type $A\widetilde{B}$.
\end{proof}

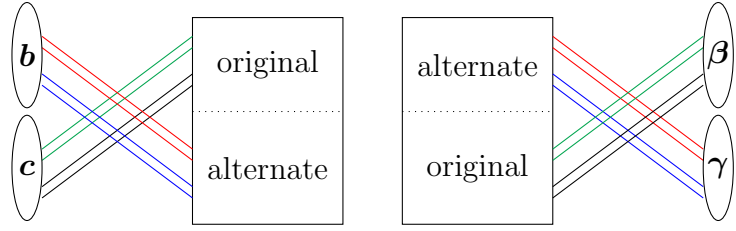
\begin{figure}
    \centering
\begin{tikzpicture}[scale = 1]
\draw (0,0.75) -- (2,0.75) -- (2,3.5) -- (0,3.5) -- cycle;
\draw [Green](0,3.25) -- (-2,1.75);
\draw[Green] (0,3.1) -- (-2,1.6);
\draw (0,2.75) -- (-2,1.25);
\draw (0,2.6) -- (-2,1.1);
\draw[red](0,1.75) -- (-2,3.25);
\draw[red] (0,1.6) -- (-2,3.1);
\draw [blue](0,1.25) -- (-2,2.75);
\draw[blue] (0,1.1) -- (-2,2.6);
\draw[dotted] (0,2.25) -- (2,2.25);
\draw (-2.2,3) ellipse (0.2cm and 0.7cm);		
\node () at (-2.2,3) {$\boldsymbol{b}$};
\draw (-2.2,1.5) ellipse (0.2cm and 0.7cm);		
\node () at (-2.2,1.5) {$\boldsymbol{c}$};
\node() at (1,1.5) {\text{alternate}};
\node() at (1,2.875) {\text{original}};
\end{tikzpicture}   
\hspace{0.5cm}
\begin{tikzpicture}[scale=1]
\draw (0,0.75) -- (2,0.75) -- (2,3.5) -- (0,3.5) -- cycle;
\draw [red](2,3.25) -- (4,1.75);
\draw[red] (2,3.1) -- (4,1.6);
\draw [blue](2,2.75) -- (4,1.25);
\draw[blue] (2,2.6) -- (4,1.1);
\draw[Green](2,1.75) -- (4,3.25);
\draw[Green] (2,1.6) -- (4,3.1);
\draw (2,1.25) -- (4,2.75);
\draw (2,1.1) -- (4,2.6);
\draw[dotted] (0,2.25) -- (2,2.25);
\node() at (1,2.875) {\text{alternate}};
\node() at (1,1.5) {\text{original}};
\draw (4.2,3) ellipse (0.2cm and 0.7cm);		
\node () at (4.2,3) {$\boldsymbol{\beta}$};
\draw (4.2,1.5) ellipse (0.2cm and 0.7cm);		
\node () at (4.2,1.5) {$\boldsymbol{\gamma}$};
\end{tikzpicture}

\caption{Braids of $R^{(n)}$ vertices attached to the Cauchy model.}
\label{cauchycartoon}
\end{figure}

\begin{conj}\label{infside}
Set $\boldsymbol{\beta},\boldsymbol{c} = $ ``right" and $\boldsymbol{\gamma},\boldsymbol{b} =$ ``left".  Attaching a braid to the right hand side of the Cauchy model with ordering $I_B$ (as in the right hand figure of Figure \ref{cauchycartoon}) gives the partition function
\[\prod_{i,j,k,\ell} \prod_{t = 0}^{n-1}\frac{(1- q^{2t}x_iz_\ell)(1-q^{2t}y_jw_k)}{(1-q^{2t}x_iw_k)} \cdot \sum_\lambda \mathcal{G}_{\mu/\lambda}(X/Y;q) \cdot \mathcal{G}_{\nu/\lambda}(W/Z;q),\]
On the other hand, attaching a braid to the left hand side of the Cauchy model with ordering $I_A$ gives the partition function
\[\left(\prod_{j,\ell} \prod_{t = 0}^{n-1} (1-q^{2t}y_jz_\ell) . \sum_{\lambda} \mathcal{G}_{\lambda/\mu}(X/Y;q) \cdot \mathcal{G}_{\lambda/\nu}(W/Z;q)\right) + O(Y^\infty)O(Z^\infty),\]
where the notation $O(Y^\infty)$ denotes terms with infinitely many powers of spectral parameters in the set $Y$.
\end{conj}

\begin{rmk}
Setting $\boldsymbol{\beta},\boldsymbol{c} = $ ``left" and $\boldsymbol{\gamma},\boldsymbol{b} =$ ``right" and swapping sides on which the braid attaches (i.e. using type $\widetilde{A}B$ instead of type $A\widetilde{B}$ R-vertices) produces an analogous result.
\end{rmk}

The latter statement is actually fairly straightforward to prove: if no particles pass through the Cauchy braid, all its vertices are of type N, so it has weight $\prod_{j,\ell} \prod_{t = 0}^{n-1} (1-q^{2t}y_jz_\ell)$ and the remaining piece of the lattice has boundary conditions  $\mathcal{C}_{\mu,\nu,I_B}$, so Proposition \ref{cauchysides} provides the rest of that term. On the other hand, suppose a particle enters the Cauchy braid: since all of the particles entering from boundary conditions are required to travel away from the braid in the alternate model, we may think of this particle as originating in the braid, travelling right through the alternate model, up into the original model, then back left into the braid to create a loop. Since columns can only carry one particle at a time, this particle must travel through all non-positive columns in the alternate model before it can travel up into the original; however, horizontal strip rows have weight 0 on SE (respectively SW) vertices in the original (respectively alternate) model, so infinitely many of these steps must occur in vertical strip rows, amassing a weight of $O(Y^\infty)$ in the alternate model section and $O(Z^\infty)$ in the original model section.

The first claim is the more difficult, because it will require precisely evaluating the Cauchy braid in circumstances where there may be particles travelling through the braid (in fact, looping as in the previous case), taking into account the possibility of pre-fusion as well. Excepting this difficulty, notice that in any state, cutting off the loops and considering the remaining boundary conditions, we see that they match $\mathcal{C}_{\mu,\nu,I_A}$, so again Proposition \ref{cauchysides} gives the desired sum over superLLT polynomials. 

\begin{rmk}
For circumstances where we only need finitely many columns to represent all possible $\lambda$ that could appear on the boundary (i.e., when the original model is all horizontal strip rows and the alternate all vertical strip rows, or vice versa), the error terms in $O(Y^\infty)O(Z^\infty)$ appearing in Conjecture \ref{infside} are zero, so truncating this Cauchy lattice model recovers precisely the dual Cauchy identity for LLT polynomials discussed in Section \ref{sec:heisenberg}. 
\end{rmk}

\bibliographystyle{alpha}
\bibliography{main.bib}

\begin{appendices}
\section{Reduced Yang-Baxter Equation}\label{sec:rybe}
\setcounter{equation}{0}
\numberwithin{equation}{section}
After eliminating the $0$-weighted terms, the Yang-Baxter Equation can be simplified to the following system of equations.

\begin{center}
\begin{longtable}{cc}\begin{minipage}[c]{0.45\linewidth}{
\begin{equation}\label{ybe1}
\ybeleftn101101100\ =\ \yberightn101101001	
\end{equation}}\end{minipage}
& 
\begin{minipage}[c]{0.45\linewidth}{\begin{equation}\label{ybe2}
\ybeleftn101011101\ =\ \yberightn101011001	
\end{equation}}\end{minipage}
\\
\begin{minipage}[c]{0.45\linewidth}
\begin{equation}\label{ybe3}
\ybeleftn011101100\ =\ \yberightn011101101
\end{equation}
\end{minipage}
& 
\begin{minipage}[c]{0.45\linewidth}
\begin{equation}\label{ybe4}
\ybeleftn011011101\ =\ \yberightn011011101
\end{equation}
\end{minipage}
\\
\begin{minipage}[c]{0.45\linewidth}
\begin{equation}\label{ybe5}
\ybeleftn100100010\ =\ \yberightn100100010	
\end{equation}
\end{minipage} 
& 
\begin{minipage}[c]{0.45\linewidth}
\begin{equation}\label{ybe6}
\ybeleftn100010011\ =\ \yberightn100010010	
\end{equation}
\end{minipage}
\\
\begin{minipage}[c]{0.45\linewidth}
\begin{equation}\label{ybe7}
\ybeleftn010010011\ =\ \yberightn010010110	
\end{equation}
\end{minipage}
&
\begin{minipage}[c]{0.45\linewidth}
\begin{equation}\label{ybe8}
\ybeleftn010100010\ =\ \yberightn010100110	
\end{equation}
\end{minipage}
\\
\begin{minipage}[c]{0.45\linewidth}
\begin{equation}\label{ybe11}
\ybeleftn001001001\ =\ \yberightn001001100
\end{equation}
\end{minipage}
&
\begin{minipage}[c]{0.45\linewidth}
\begin{equation}\label{ybe14}
\ybeleftn000000000\ =\ \yberightn000000000
\end{equation}
\end{minipage}
\\
\multicolumn{2}{c}{
\begin{minipage}[c]{0.7\linewidth}
\begin{equation}\label{ybe9}
\ybeleftn100001100+\ybeleftn100001011\ =\ \yberightn100001000
\end{equation}
\end{minipage}}
\\
\multicolumn{2}{c}{
\begin{minipage}[c]{0.7\linewidth}
\begin{equation}\label{ybe12}
\ybeleftn010001100+\ybeleftn010001011\ =\ \yberightn010001100
\end{equation}
\end{minipage}}
\\
\multicolumn{2}{c}{
\begin{minipage}[c]{0.7\linewidth}
\begin{equation}\label{ybe10}
\ybeleftn001100000\ =\ \yberightn001100001+\yberightn001100110
\end{equation}
\end{minipage}}
\\
\multicolumn{2}{c}{
\begin{minipage}[c]{0.7\linewidth}
\begin{equation}\label{ybe13}
\ybeleftn001010001\ =\ \yberightn001010001+\yberightn001010110
\end{equation}
\end{minipage}}
\end{longtable}
\end{center}

\section{Reduced Yang-Baxter Equation for Column Weights}\label{sec:rybeC}
\setcounter{equation}{0}
\numberwithin{equation}{section}
When considering the YBE for the Vertical-Strip Weights, we replace boundary conditions A2,3,4,9,12, and 14, which all contain a NE vertex, with the following new boundary conditions, which were excluded from consideration before on account of containing a SE vertex. We also have two sets of boundary conditions from before (A6 and A8) which acquire an extra state: these are B6 and B8, respectively.

\begin{center}
\begin{longtable}{cc}\begin{minipage}[c]{0.45\linewidth}{
\begin{equation}\label{ybe1b}
\ybeleftn111111111\ =\ \yberightn111111111	
\end{equation}}\end{minipage}
& 
\begin{minipage}[c]{0.45\linewidth}{\begin{equation}\label{ybe2b}
\ybeleftn011110011\ =\ \yberightn011110111	
\end{equation}}\end{minipage}
\\
\begin{minipage}[c]{0.45\linewidth}{
\begin{equation}\label{ybe3b}
\ybeleftn110011111\ =\ \yberightn110011110	
\end{equation}}\end{minipage}
& 
\begin{minipage}[c]{0.45\linewidth}{\begin{equation}\label{ybe4b}
\ybeleftn110110110\ =\ \yberightn110110011	
\end{equation}}\end{minipage}
\\
\multicolumn{2}{c}{
\begin{minipage}[c]{0.7\linewidth}
\begin{equation}\label{ybe5b}
\ybeleftn101110011+\ybeleftn101110100\ =\ \yberightn101110011
\end{equation}
\end{minipage}}
\\
\multicolumn{2}{c}{
\begin{minipage}[c]{0.7\linewidth}
\begin{equation}\label{ybe6b}
\ybeleftn100010011 + \ybeleftn100010100\ =\ \yberightn100010010	
\end{equation}
\end{minipage}}
\\
\multicolumn{2}{c}{
\begin{minipage}[c]{0.7\linewidth}
\begin{equation}\label{ybe7b}
\ybeleftn110101110\ =\ \yberightn110101110+\yberightn110101001
\end{equation}
\end{minipage}}
\\
\multicolumn{2}{c}{
\begin{minipage}[c]{0.7\linewidth}
\begin{equation}\label{ybe8b}
\ybeleftn010100010\ =\ \yberightn010100110+\yberightn010100001
\end{equation}
\end{minipage}}
\end{longtable}
\end{center}

\section{$R^{(1)}$ Vertex Weights for Cauchy Mixed YBES}\label{restofCauchyRwts}

Let $\kappa = \#NN \text{ below } + \#SS\text{ above}$, and let $\tau = \#SS\text{ below } +\# NN\text{ above}$.

\begin{multicols}{2}

\begin{tabular}{|c|c|}
\hline \multicolumn{2}{|c|}{$\widetilde{H}H$}\\
\hline Type& Weight\\
\hline W& $q^\tau$\\
\hline S& $(1-q^{2n-2-2\#S \text{ below}}xy)$\\
\hline SS*& $q^{n-1-\#S}xy$\\
\hline NN*& $q^{n-1-\#S}$\\
\hline N& $-q^{n-1-\#S}$\\
\hline E& $q^{\tau - \#W}$\\
\hline$\frac{NN}{SS}$ & $-q^{2n-2}xy + q^{2\tau + 2\#S} - q^{2\#S}$\\\hline
\end{tabular}

\vspace{0.5cm}

\begin{tabular}{|c|c|}
\hline \multicolumn{2}{|c|}{$\widetilde{H}V$}\\
\hline Type& Weight\\
\hline W& $q^\tau$\\
\hline S& $1$\\
\hline SS*& $(-xy)q^{\# SS \text{ below} + \#E}$\\
\hline NN*& $q^{\# NN\text{ above} + \#E}$\\
\hline N& $-q^{n-1-\#S}$\\
\hline E& $(1-q^{2\#E \text{ below}}xy)q^{\tau + \#W}$\\
\hline$\frac{NN}{SS}$ & $q^{3\kappa}xy - q^{3\tau + 2\#E} - q^{\tau}$\\\hline
\end{tabular}

\vspace{0.5cm}

\begin{tabular}{|c|c|}
\hline \multicolumn{2}{|c|}{$\widetilde{V}H$}\\
\hline Type& Weight\\
\hline W& $(1-q^{2 \# W \text{ below}}xy)q^{\tau}$\\
\hline S& $1$\\
\hline SS*& $-xy\cdot q^{\# SS \text{ below } +\#W}$\\
\hline NN*& $q^{\# NN\text{ above } + \#W}$\\
\hline N& $-q^{n-1-\#S}$\\
\hline E& $q^{\tau + \#W}$\\
\hline$\frac{NN}{SS}$ & $q^{3\tau} - q^{3\tau + 2\#W}xy - q^{\tau}$\\\hline
\end{tabular}

\vspace{0.5cm}

\begin{tabular}{|c|c|}
\hline \multicolumn{2}{|c|}{$\widetilde{V}V$}\\
\hline Type& Weight\\
\hline W& $q^\tau$\\
\hline S& $1$\\
\hline SS*& $xy$\\
\hline NN*& $1$\\
\hline N& $(xy - q^{2\#N \text{ below}})q^{\#W + \#E}$\\
\hline E& $q^{\tau +\#W}$\\
\hline$\frac{NN}{SS}$ & $xy + q^{2\tau + 2\#N} - q^{2\#N}$\\\hline
\end{tabular}

\vspace{0.5cm}

\begin{tabular}{|c|c|}
\hline \multicolumn{2}{|c|}{$H\widetilde{H}$}\\
\hline Type& Weight\\
\hline W& $q^{-\tau}$\\
\hline S& $-q^{1-n+2\kappa}(q^{2n-2+2\#S \text{ below}}xy - 1)$\\
\hline SS*& $q^{n-1-\# NN \text{ above}}$\\
\hline NN*& $xy\cdot q^{n-1-\# SS\text{ below}}$\\
\hline N& $q^{\kappa - \tau + \#S}$\\
\hline E& $q^{-\tau - \#W}$\\
\hline$\frac{SS}{NN}$ & $\left(q^{2n-2 + 2\kappa}xy - q^{2\kappa - 2\#S} + q^{-2\#S}\right)\cdot q^{n-1-\tau}$\\\hline
\end{tabular}

\vspace{0.5cm}

\begin{tabular}{|c|c|}
\hline \multicolumn{2}{|c|}{$H\widetilde{V}$}\\
\hline Type& Weight\\
\hline W& $q^{\kappa - \#E}$\\
\hline S& $-q^{1-n+2\kappa + \#N}$\\
\hline SS*& $q^{\# SS \text{ above}}$\\
\hline NN*& $-xy\cdot q^{\# NN\text{ below}}$\\
\hline N& $q^{2\kappa}$\\
\hline E& $(1-q^{2n - 2 - 2 \# E \text{ below}}xy)q^{-\tau}$\\
\hline$\frac{SS}{NN}$ & $-q^{3\kappa}xy - q^{2(\#E - (n-1)) + 5\kappa} + q^{2(\#E - (n-1)) + 3\kappa)}$\\\hline
\end{tabular}

\vspace{0.5cm}

\begin{tabular}{|c|c|}
\hline \multicolumn{2}{|c|}{$V\widetilde{H}$}\\
\hline Type& Weight\\
\hline W& $(1-q^{2n - 2 - 2 \# W \text{ below}}xy)q^{-\tau}$\\
\hline S& $-q^{1-n+2\kappa + \#N}$\\
\hline SS*& $q^{\# SS \text{ above}}$\\
\hline NN*& $-xy\cdot q^{\# NN\text{ below}}$\\
\hline N& $q^{2\kappa}$\\
\hline E& $q^{\kappa - \#W}$\\
\hline$\frac{SS}{NN}$ & $-q^{3\kappa}xy - q^{2(\#W - (n-1)) + 5\kappa} + q^{2(\#W - (n-1)) + 3\kappa)}$\\\hline
\end{tabular}

\vspace{0.5cm}

\hspace{2cm} \begin{tabular}{|c|c|}
\hline \multicolumn{2}{|c|}{$V\widetilde{V}$}\\
\hline Type& Weight\\
\hline W& $q^{-\tau}$\\
\hline S& $-q^{1-n+\kappa + \#N - \tau}$\\
\hline SS*& $q^{-\# NN \text{ above}}$\\
\hline NN*& $xy\cdot q^{-\# SS\text{ below}}$\\
\hline N& $(1- q^{2\#N \text{ below}}xy)q^{2\kappa}$\\
\hline E& $q^{\tau - \#W}$\\
\hline$\frac{SS}{NN}$ & $\left(q^{2\kappa}xy - q^{2\kappa - 2\#N} + q^{-2\#N}\right)\cdot q^{-\tau} $\\\hline
\end{tabular}
\end{multicols}
\end{appendices}

\end{document}